\documentclass[imslayout,preprint,natbib]{imsart}

\usepackage{amsthm,amsmath,amssymb}
\usepackage[pdftex]{graphicx}
\graphicspath{{./}{./figures/}{../subpoisson/figures/}{../R/}}
\usepackage[pdftex,colorlinks,%
           bookmarks=true,%
           pdftitle=Clustering\,\ percolation\ and\ directionally\ convex\ ordering\ of\ point\ processes,%
           pdfauthor=B.\ Blaszczyszyn%
          \ -\  D.\ Yogeshwaran]{hyperref}

\usepackage{hypernat}
\usepackage[english]{babel}

\RequirePackage[colorlinks,citecolor=blue,urlcolor=blue]{hyperref}

\RequirePackage{hypernat}

\startlocaldefs

\newcommand{\EXCLUDE}[1]{}
\newcommand{\be}{\begin{equation}}
\newcommand{\ee}{\end{equation}}
\newcommand{\bea}{\begin{eqnarray}}
\newcommand{\no}{\nonumber}
\newcommand{\non}{\nonumber}
\newcommand{\eea}{\end{eqnarray}}
\newcommand{\sP}{\mathsf{P}}  
\newcommand{\pr}[1]{\mathsf{P}\left( #1 \right)}
\newcommand{\sE}{\mathsf{E}} 
\newcommand{\EXP}[1]{\mathsf{E}\!\left(#1\right) }
\newcommand{\COV}[1]{\mathsf{Cov}\left( #1 \right)}

\newcommand{\1}[1]{\mathsf{1}\!\left[\,#1\,\right] }

\newtheorem{thm}{Theorem}[section]
\newtheorem{cor}[thm]{Corollary}
\newtheorem{lem}[thm]{Lemma}
\newtheorem{prop}[thm]{Proposition}
\newtheorem{remark}[thm]{Remark}
\newenvironment{rem}{\begin{remark}\rm}{\end{remark}}

\newtheorem{example}[thm]{\bf Example}
\newenvironment{exe}{\begin{example}\rm}{\end{example}}

\def\rar{\rightarrow}
\def\la{\langle}
\def\ra{\rangle}

\newcommand{\al}{\alpha}

\newcommand{\lam}{\lambda}
\newcommand{\Lam}{\Lambda}

\newcommand{\sg}{\sigma}

\def\Ph{\Phi}

\def\mR{\mathbb{R}}
\def\mZ{\mathbb{Z}}

\def\mN{\mathbb{N}}

\def\mE{\mathbb{E}}
\def\mL{\mathbb{L}}
\def\mM{\mathbb{M}}

\newcommand{\ur}{{\underline r}}
\newcommand{\ovr}{{\overline r}}

\newcommand{\md}{\text{d}}

\newcommand{\cX}{{\mathcal X}}
\newcommand{\cN}{{\mathcal N}}

\def\1{\mathbf{1}}

\endlocaldefs

\begin{document}

\begin{frontmatter}

\title{Clustering, percolation and directionally convex ordering of
point processes}
\runtitle{Clustering, percolation and $dcx$ ordering}

\begin{aug}

\author{\fnms{Bart{\l}omiej}
  \snm{B{\l}aszczyszyn}\corref{}\thanksref{inria,ens}\ead[label=e1]{Bartek.Blaszczyszyn@ens.fr}} and
\author{\fnms{D.} \snm{Yogeshwaran}\thanksref{ens,inria}\ead[label=e2]{yogesh@di.ens.fr}}
\address{INRIA/ENS\\
23 av. d'Italie\\
CS 81321\\
75214 Paris Cedex 13\\
France\\
\printead{e1}\\
\phantom{E-mail:\ }\printead*{e2}}
\affiliation{INRIA\thanksmark{inria} and ENS Paris\thanksmark{ens}}

\runauthor{B. B{\l}aszczyszyn and D. Yogeshwaran}

\end{aug}

\begin{abstract}\
Heuristics indicate that point processes exhibiting
clustering of points have larger critical radius $r_c$ for
the percolation of their continuum percolation models  than spatially
homogeneous point processes.  
It has already been shown, and we reaffirm it in this paper, that 
the $dcx$ ordering of point processes is suitable to compare their
clustering tendencies. Hence, it was tempting to conjecture that~$r_c$
is increasing in $dcx$ order. 
Some numerical evidences support this conjecture 
for a special class of point processes, called perturbed lattices,
which are ``toy models'' for determinantal and permanental point
processes. However the conjecture is not true in full generality,
since one can  construct a Cox point process with degenerate critical
radius~$r_c=0$, that is $dcx$ larger than a
given homogeneous Poisson point process, for which this radius is
known to be strictly positive. Nevertheless, the aforementioned
monotonicity in $dcx$ order can be proved,
for a nonstandard critical radius $\ovr_c$ (larger than $r_c$), 
related to the Peierls argument.
Moreover, we show the reverse monotonicity for another 
nonstandard critical radius $\ur_c$
(smaller than $r_c$). This gives uniform lower and upper bounds on
$r_c$ for all point processes $dcx$ smaller than some given point process.
Moreover, we show that point processes $dcx$ smaller than a 
homogeneous Poisson point process
admit uniformly non-degenerate lower and upper bounds on~$r_c$.
In fact, all the above results hold under weaker assumptions on the
ordering of void probabilities or factorial moment measures
only. Examples of point processes comparable to Poisson point
processes in this weaker sense include determinantal and permanental
point processes with trace-class integral kernels. 
More generally, we show that point processes $dcx$ smaller than a 
homogeneous Poisson point 
processes exhibit phase transitions in certain percolation
models based on the level-sets of additive shot-noise fields of these point
process.  Examples of such models are $k$-percolation and 
SINR-percolation.
\end{abstract}
\end{frontmatter}
\clearpage 
\section{Introduction}
\label{sec:intro}
\paragraph{Heuristic}
Consider a point process $\Phi$ in the $d$-dimensional Euclidean
space $\mR^d$.  For a given ``radius'' $r \ge 0$, let us join by an
edge any two points of $\Phi$, which are at most at a distance of $2r$ from each
other. Existence of an infinite component in the resulting graph is
called {\em percolation} of the continuum model based on
$\Phi$. Clustering of $\Phi$ roughly means that the points of $\Phi$
lie in clusters (groups) with the clusters being well spaced out. When
trying to find the minimal $r$ for which the continuum  model based on~$\Phi$
percolates, we observe that points lying in the same cluster of $\Phi$ will be
connected by edges for some smaller $r$ but points in different
clusters need a relatively higher $r$ for having edges between
them. Moreover, percolation cannot be achieved without edges between
some points of different clusters. It seems to be 
evident that spreading points
from clusters of $\Phi$ ``more homogeneously'' in the space would
result in a decrease of the radius $r$ for which the percolation takes
place. This is a heuristic explanation why clustering in a point
process $\Phi$ should increase the {\em critical radius}
$r_c=r_c(\Phi)$ for the percolation
of the continuum percolation model on~$\Phi$, called also the
Gilbert's disk graph or the Boolean model with fixed spherical
grains. 

\paragraph{Clustering and  $dcx$ order}
To make a formal conjecture out of the above heuristic,
one needs to adopt a tool to compare clustering properties of
different point processes. In this regard, we use {\em directionally convex} ($dcx$) order in this article. The $dcx$ order of random vectors is an integral order generated by twice differentiable functions with all their second order partial derivatives being non-negative.\footnote{We remark here that $dcx$ order was initially developed for random vectors  (\cite{MeesterSh93,MeesterSh99,Shaked90})
partially in conjunction  with Ross-type conjectures,
which predicted that queues with a variable input perform
worse (cf~\cite{Miyoshi04a}).  Much earlier to these works, a
comparative study of queues using the supermodular order and motivated by
neuron-firing models can be found in~\cite{Huffer84a}.}
Its extension  to point processes consists in comparison
of vectors of number of points in every possible finite collection of bounded Borel subsets of the space. Our choice has its roots in \cite{snorder}, where one shows various results as well as examples indicating that the $dcx$ order on point processes implies ordering of several well-known clustering characteristics in spatial statistics such as Ripley's K-function and second moment densities. Namely, a point process that is larger in the $dcx$ order exhibits more clustering,
while having the same mean number of points in any given set.

\paragraph{Conjecture}
The above discussion tempts one to  conjecture that $r_c$ is increasing with respect to the $dcx$ ordering of the underlying point processes; i.e.,  
$\Phi_1 \leq_{dcx} \Phi_2$ implies $r_c(\Phi_1) \leq r_c(\Phi_2)$.
Since the critical radius $r_c$ for the percolation does not
seem to have any explicit representation as a $dcx$ function of
vectors of number of points in a finite collection of bounded Borel
subsets, we are not able to compare straightforwardly $r_c$  of $dcx$
ordered point processes. At the same time, numerical evidences
gathered for a certain  class of point processes, called perturbed lattice point
processes, were  supportive of this conjecture. But as it turns out,
the conjecture is not true in full generality and we will present a
counter-example.    

\paragraph{Counterexample}
More specifically,  for a given Poisson point process, we 
 will construct a $dcx$ larger than it Poisson-Poisson cluster point process
(a special case of doubly-stochastic Poisson, called also Cox, point process),  
whose critical radius is null (hence smaller than that of the given Poisson point process, which is known to be positive). In this Poisson-Poisson cluster  point process points concentrate  on some carefully chosen 
larger-scale structure, which itself has good percolation properties. 
In this case, the points concentrate in clusters, however we cannot say that  clusters are well spaced out. Hence,  this example does not 
contradict our initial heuristic explanation of  why clustering in a
point process should increase the critical radius for the percolation.
It reveals rather, that $dcx$ ordering, while being able to compare
the clustering tendency of point processes, is not capable of 
comparing macroscopic structures of clusters.

\paragraph{Comparison results} 
What we are able to do with $dcx$ order
is to compare some  {\em nonstandard critical radii} $\ovr_c$ and $\ur_c$ related, respectively, to the finiteness of the expected number of void circuits around the origin and asymptotic of the expected number of long occupied paths from the origin in suitable discrete approximations of the continuum model. These new critical radii sandwich the ``true'' one, $\ur_c(\Phi) \le r_c(\Phi) \le \ovr_c(\Phi)$. We show that if  $\Phi_1\le_{dcx}\Phi_2$, then as
suggested by the heuristic, $\ovr_c(\Phi_1)\le\ovr_c(\Phi_2)$.
However, we obtain the {\em reversed inequality} for the other nonstandard
critical radius: $\ur_c(\Phi_1)\ge\ur_c(\Phi_2)$.
This reversed inequality can be also explained heuristically, by noting that
whenever there is at least one path of some given length in a point process that
clusters more, there might be actually so many such paths, that the
inequality for the expected numbers of paths are reversed. However, as we have
mentioned above, 
the fact is that there are {\em examples of $dcx$ ordered point processes,
for which the inequality is  reversed for the critical radius
  $r_c$ itself}, namely examples of 
$\Phi_1\le_{dcx} \Phi_2$ with $r_c(\Phi_1)> r_c(\Phi_2)$. 

\paragraph{Non-trivial phase transitions}
A more positive interpretation of the aforementioned results is as follows.   
Combining the two opposite inequalities, we obtain as  a corollary
that if $\Phi_1\le_{dcx}\Phi_2$ then the usual critical radius of the
$dcx$ smaller point process can be sandwiched between the two nonstandard
critical radii of the larger one,
$$ \ur_c(\Phi_2) \leq \ur_c(\Phi_1) \leq r_c(\Phi_1) \leq \ovr_c(\Phi_1) \leq \ovr_c(\Phi_2).
$$
In more loose terms, the above result can be rephrased as follows:
if a point process $\Phi_2$ 
exhibits the following {\em double phase transition} 
$0<\ur_c(\Phi_2)\le\ovr_c(\Phi_2)<\infty$,  stronger than the
usual one ($0<r_c(\Phi_2)<\infty$), then its two nonstandard critical radii act as
non-degenerate bounds on the critical radius $r_c(\Phi)$, uniformly
valid  over all $\Phi\le_{dcx}\Phi_2$, thus in particular implying 
the usual phase transition $0<r_c(\Phi)<\infty$ for all such $\Phi$. 

One would naturally like to verify the aforementioned double phase
transitions for the homogeneous Poisson point process.
A direct verification eludes us at the moment. However, using a slightly
modified approach, we are able to obtain the same result (uniform
bounds on the critical radius $r_c$) for  all
homogeneous point processes that are $dcx$ smaller than the Poisson
point process of a given  intensity; we call them {\em homogeneous
  sub-Poisson point processes}. 

In fact, all the above results regarding comparison of 
$\ur_c$ and $\ovr_c$, as well as the non-trivial phase-transition 
for homogeneous sub-Poisson points processes,
can be proved under a weaker assumption than the $dcx$ assumption.
Namely, it is enough to assume that the void probabilities and factorial moment
measures of the respective point processes are ordered.
However, $dxc$ ordering needs to be
 assumed to show non-trivial phase transitions
for other models based on level-sets of
additive or extremal shot-noise fields.

\paragraph{Examples}
In this article, we also provide new examples of point processes
comparable in $dcx$ order.
In particular, families of point processes monotone in $dcx$ order are
constructed using  some random {\em perturbations
of lattices}. Numerical experiments performed for these processes
suggest  monotonicity of $r_c$ with respect
to $dcx$ order within this class. 

Recently, {\em determinantal} and {\em permanental point processes}
have been 
attracting a lot of attention. They are known as examples of point
processes that cluster, respectively, less and more than the Poisson
point process of the same intensity. Assuming trace-class integral kernels,
it is relatively easy to show that these
point processes,  
have void probabilities and moment measures comparable to those of
the Poisson point process. Thus, in particular, determinantal point processes
admit uniform bounds on the critical radius~$r_c$ discussed
above. Moreover, we show that determinantal and permanental point processes
are $dcx$ comparable to the Poisson point process on mutually disjoint  {\em simultaneously observable sets}.
Our proof of this latter result is based on
the known representation of the distribution of these point processes on
such sets, as well as on  Lemma~\ref{lem:MeesterSh-like}
regarding $dcx$ ordering of random vectors, which we
believe is new, albeit similar to~\cite[Lem\-ma~2.17]{MeesterSh93}.
Interestingly enough, our perturbed lattice examples
admit a very similar representation, but this time on all
sets, and it is the same Lemma~\ref{lem:MeesterSh-like} that allows us
to prove their $dcx$ ordering.

\paragraph{Related work}
Let us now make some remarks on other comparison
studies in continuum percolation. Most of the results  regard
comparison of different models driven by the same (usually Poisson)
point process. In~\cite{Jonasson01}, it was shown that  the critical intensity for
percolation of the Boolean model on the plane is minimized when the shape of
the typical grain is a triangle and maximized when it is a centrally symmetric set. Similar result
was proved in~\cite{Roy02} using more probabilistic arguments
for the case when the shapes are taken over the set of all
polygons. This idea was also used for comparison of percolation models
with different shapes in three dimensions. It is known for many
discrete graphs that bond percolation is strictly easier than site
percolation.
A similar result as well as strict
inequalities for spread-out connections in random connection model has
been proved in~\cite{Franc_etal05,Franc_etal10}.
The underlying point process in all the above studies has been a Poisson
point process and it is here that our study differs from them.

Critical radius of the continuum percolation model on
the hexagonal lattice perturbed by
the Brownian motion is studied in a recent
pre-print~\cite{BenjaminiStauffer2011}.%
This is an example of our perturbed lattice and as such it is a
$dcx$ sub-Poisson point process.%
\footnote{More precisely, at any time $t$ of the evolution of the
Brownian motion, it is $dcx$ smaller than  a non-homogeneous
Poisson point process of some intensity which depends on~$t$,
and converges to the homogeneous one for $t\to\infty$.}
It is shown that for a short enough time of the evolution of the Brownian
motion the critical radius is not larger than that of the
non-perturbed lattice, and for large times, the critical radius asymptotically
converges to  that of the homogeneous Poisson point process. This latter result
is shown by some coupling  in the sense of set inclusion of point processes.
Many other inequalities in percolation theory depend on
such coupling arguments (cf. e.g.~\cite{Liggett97}),
which  for obvious reasons are not suited to comparison of point
processes with the same mean measures.
For studies of this type, convex
orders from the theory of stochastic ordering turn out to be quite
useful. Our general goal in this article is to show the utility of
these tools for comparison of 
properties of continuum percolation models.

\paragraph{Summary of results}
Let us be more specific regarding the contributions of this paper
and quickly summarize our results.
\begin{list}{\labelitemi}{\leftmargin=2em}
\item  Point processes {\em smaller in $dcx$ order have smaller void
  probabilities} as well as imply
{\em smaller void probabilities of their Boolean models} (Propositions \ref{Prop:voids_pp} and \ref{prop:ord_cap_fnl}).
\item {\em For $\Phi_1\le_{dcx}\Phi_2,$ we have
$\ovr_c(\Phi_1)\le\ovr_c(\Phi_2)$ and $\ur_c(\Phi_1)\ge
\ur_c(\Phi_2)$}, where $\ur_c$ and
$\ovr_c$ are certain nonstandard critical radii 
bounding the (standard) critical radius $r_c$, $\ur_c\le r_c\le
\ovr_c$, for the percolation of 
the Boolean model with spherical grains
(Corollary \ref{cor:ord_crit_rad_racs} and Proposition \ref{prop:lower_crit_rad}). 
\item {\em For any stationary point process 
having smaller void probabilities and 
factorial moment measures than the Poisson point process
of the same intensity, $0<r_c<\infty$}.
Moreover, {\em the critical
radii $r_c$ are uniformly bounded away of 0 and of $\infty$ over all such
point processes  of a given intensity} (Propositions~\ref{thm:sinr_poisson_perc} and \ref{thm:perc_sinr_sub-Poisson}).
\item The above result is extended to the {\em $k-percolation$ 
of the Boolean model} (percolation of the subset
covered by at least $k$ of its grains) {\em under the assumption that
 the point process is sub-Poisson}
(Proposition~\ref{thm:k_perc_sub-Poisson_pp}). 
\item We prove  the {\em existence of the percolation regime in 
certain Signal-to-interferen\-ce-and-noise ratio coverage models
driven by homogeneous sub-Poisson
point processes}  (Propositions~\ref{thm:sinr_poisson_perc}
and~\ref{thm:perc_sinr_sub-Poisson}).
These are examples of germ-grain models with grains jointly 
dependent on certain shot-noise fields;
their percolation has been previously studied assuming a Poisson point process of germs.
\item We prove that {\em perturbed lattices with convexly ordered 
point replication kernels  are $dcx$ ordered}
(Proposition~\ref{p.pert-lattice}). This provides 
examples of {\em $dcx$-monotone families  of point processes
comparable to a given Poisson process} (Example
\ref{ex:perturbed.lattice}). 
\item We give an {\em example of a Cox point process with degenerate
    critical radius, $r_c=0$, which is $dcx$ larger than a given 
homogeneous Poisson process} (Example \ref{ex.PoPoClust}). 
\item We show that {\em determinantal and permanental point processes
 with trace-class integral kernels  have void probabilities and
 factorial moment measures ordered with  
    respect to  Poisson point processes of the same mean measure.
     Moreover, we show that they are
    $dcx$-ordered with respect to such Poisson point processes on
    mutually disjoint simultaneously observable sets.} (Examples \ref{ex.det} and~\ref{ex.perm}).
\end{list}

\paragraph{Paper organization} The necessary notions and notations 
are introduced in Section~\ref{s.Notation}. In Section~\ref{sec:prelim}, we argue that $dcx$ ordering is suitable for the comparison of clustering properties of point processes. In Section~\ref{s.Boolean}, we state and prove our
comparison results for two nonstandard critical radii $\ur_c$ and
$\ovr_c$.
In Section~\ref{sec:sub-poisson}, we use  generic exponential
estimates to obtain non-degenerate bounds for the percolation of
various percolation models   driven by homogeneous sub-Poisson
point processes. New examples of $dcx$ ordered  point processes,
including the counter-example to the conjecture, 
are provided in Section~\ref{sec:ex}.  
Lemma~\ref{lem:MeesterSh-like} used for showing $dcx$ ordering of
perturbed lattices and determinantal and permanental point processes
is proved in the Appendix.

\section{Notions and notation}
\label{s.Notation}
\subsection{Point processes}
Let $\texttt{B}^d$ be the Borel $\sg$-algebra and $\texttt{B}_b^d$ be
the $\sg$-ring of {\em bounded (i.e., of compact
  closure) Borel subsets} (bBs) in the
$d$-dimensional Euclidean space $\mR^d$. Let $\mM^d =
\mathbb{M}(\mR^d)$ be the space of non-negative Radon (i.e., finite on
bounded sets) measures on
$\mR^d$. The Borel $\sg$-algebra $\mathcal{M}^d$ is generated by the
mappings $\mu \mapsto \mu (B)$ for all $B$ bBs. A random measure
$\Lambda$ is a random element in $(\mM^d,\mathcal{M}^d)$ i.e, a
measurable map from a probability space $(\Omega,\mathcal{F},\sP)$ to
$(\mM^d,\mathcal{M}^d)$. We shall call a random measure $\Phi$ a {\em point
process} if $\Phi \in \bar{\mN}^d$, the subset of counting measures in
$\mM^d$. The corresponding $\sg$-algebra is denoted by
$\mathcal{N}^d$. Further, we shall say that a point process (pp) $\Phi$ is
simple if a.s. $\Phi(\{x\}) \leq 1$ for all $x \in \mR^d$. Throughout,
we shall use $\Lambda$ for an arbitrary random measure and $\Phi$ for
a pp. This is the standard framweork for random measures
(see \cite{Kallenberg83}).

As always, a pp or a random measure on $\mR^d$ is said to be {\em stationary} if its distribution is invariant with respect to translation by
vectors in~$\mR^d$.

\subsection{Boolean model}
\label{sss.Boolean_model}
Given  (the distribution of) a random closed set $G$
 (racs, see~\cite{Matheron75} for the standard measure-theoretic framework), and a pp $\Phi$, by a {\em
  Boolean model} with the pp of {\em germs} $\Phi$  and the
{\em typical grain} $G$, we
call the random set $C(\Phi,G) =\bigcup_{X_i\in\Phi} \{X_i+G_i\}$, where
$x+A=\{x+A: a\in A\}, a\in\mR^d, A\subset\mR^d$ and given $\Phi$,
$\{G_i\}$ is a sequence of i.i.d. racs, distributed as $G$.

A commonly made technical assumption about the distribution of $\Phi$
and $G$ is that for any compact $K\subset\mR^d$, the expected number of germs $X_i$ such that $(X_i+G_i)\cap K\not=\emptyset$ is finite.
This assumption, called for short ``local finiteness of the Boolean
model'' guarantees in particular that $C(\Phi,G)$ is a racs.
All the Boolean models considered in this article will be assumed to have the local finiteness property.

Denote by $B_x(r)$, the ball of radius $r$ centered at $x$.
We shall call $G$ a fixed grain if there exists a closed set $B$ such
that $G = B$ a.s.. In the case $B = B_O(r)$, where $O$ is the origin of $\mR^d$
and $r \ge 0$ constant, we shall denote the Boolean model by $C(\Phi,r)$. Note that $C(\Phi,r)$ is a racs
and so is $C(\Phi,G)$ for any $G$ a.s. bounded.

\subsection{Directionally convex ordering}
Let us quickly introduce the theory of directionally convex
ordering. We refer the reader to \cite[Section 3.12]{Muller02} for a
more detailed introduction.

The order $\leq$ on $\mR^k$  will  denote the component-wise partial
order, i.e., $(x_1,\ldots,x_k)\allowbreak \leq (y_1,\ldots,y_k)$ if
$x_i \leq y_i$ for every $i$. For $p,q,x,y\in\mR^k$ we shall
abbreviate $p \le \min \{x,y\}$ to $ p \le \{x,y\}$ and  $\max \{x,y\} \le q$ to
$\{x,y\} \leq q$.

One introduces the following families
of  {\em Lebesgue-measurable} functions on $\mR^k$:
A function $f:\mR^k \rar \mR$ is said to be
{\em directionally convex}~($dcx$) if for every $x,y,p,q \in \mR^k$ such
that $p \leq \{x,y\} \leq q$ and $x+y = p+q$, we have that
$$f(x) + f(y) \leq f(p) + f(q)\,.$$
A function $f$ is said to be {\em directionally concave}~($dcv$) if the inequality in the last equation is reversed.
A  function is said to be {\em supermodular (sm)} if the above inequality
holds for all $x,y\in\mR^k$ with $p=\min \{x,y\}$ and $q=\max \{x,y\}$.
A Lebesgue-measurable function is $dcx$ if and
only if it is $sm$ and coordinate-wise convex.
A {\em convex} function on $\mR$  will be denoted by $cx$.
We abbreviate {\em increasing} and $dcx$ by $idcx$, {\em decreasing}
and $dcx$ by $ddcx$ and similarly for $dcv$ functions.

Let $\mathfrak{F}$ denote some class of  Lebesgue-measurable
functions from $\mR^k$ to $\mR$ with the dimension $k$ being understood from the context. In the remaining part of the article, we will mainly consider $\mathfrak{F}$ to be one among the class of $dcx, idcx, idcv, dcv, ddcv, ddcx$ functions. 
Unless mentioned, when we state $\sE(f(X))$ for $f \in \mathfrak{F}$ and $X$ a random vector, we
assume that the expectation exists.
Suppose $X$ and $Y$ are real-valued random vectors of the
same dimension. Then  {\em $X$ is said to be less than $Y$ in
$\mathfrak{F}$ order} if
$\sE(f(X)) \leq \sE(f(Y))$ for all $f \in \mathfrak{F}$
such that both the expectations are finite. We shall
denote it as $X \leq_{\mathfrak{F}} Y$.
This property clearly regards only the distributions of $X$ and $Y$,
and hence sometimes we will say that the law of $X$ is less in
$\mathfrak{F}$ order than that of $Y$.

Suppose $\{X(s)\}_{s \in S}$ and $\{Y(s)\}_{s \in S}$ are
real-valued random fields, where  $S$ is an arbitrary index set.
We say that $\{X(s)\} \leq_{\mathfrak{F}}
\{Y(s)\}$ if for every $k \geq 1$ and $s_1,\ldots,s_k\in S$,
$(X(s_1),\ldots,X(s_k)) \leq_{\mathfrak{F}}(Y(s_1),\ldots,Y(s_k))$.

A random measure $\Lambda$ on $\mR^d$
can be viewed as the random
field $\{\Lambda(B)\}_{B \in \texttt{B}_b^d}$. With the
aforementioned definition of $dcx$ ordering for random fields, for two
random measures on $\mR^d$, one
says that  $\Lambda_1(\cdot) \leq_{dcx} \Lambda_2(\cdot)$, if for any $B_1, \ldots, B_k$ bBs in $\mR^k$,
\begin{equation}
\label{defn:dcx_rm}
(\Lambda_1(B_1),\ldots,\Lambda_1(B_k)) \leq_{dcx} (\Lambda_2(B_1),\ldots,\Lambda_2(B_k)).
\end{equation}
The definition is similar for other orders, i.e., when $\mathfrak{F}$
is the class of  $idcx, idcv,\allowbreak ddcx$ or $ddcv$ functions.
It was shown in~\cite{snorder} that it is enough to verify the above
condition for $B_i$  mutually disjoint.

In order to avoid technical difficulties, we will
consider here only random measures (and pp) whose  {\em mean measures} $\sE(\Lambda(\cdot))$
are Radon (finite on bounded sets). For such random measures, $dcx$ order is a
transitive order~\footnote{\label{fn:f-Ox}Due to the fact that
each $dcx$ function
can be monotonically approximated by
$dcx$ functions $f_i(\cdot)$ which satisfy $f_i(x)=O(||x||_\infty)$
at infinity, where $||x||_\infty$ is the $L_\infty$ norm on
the Euclidean space; cf.~\cite[Theorem~3.12.7]{Muller02}.}.
Note also that
$\Lambda_1(\cdot) \leq_{dcx} \Lambda_2(\cdot)$ implies the {\em equality
of their  mean measures}: $\sE(\Lambda_1(\cdot)) =\sE(\Lambda_2(\cdot))$.
For more details on $dcx$ ordering
of pp and random measures, see~\cite{snorder}.

\section{Clustering and $dcx$ ordering of point processes}
\label{sec:prelim}
In this section, we will present some basic results on $dcx$ ordering of
pp that will allow us to see this order as a tool to compare clustering of
pp. It will be shown that pp smaller in $dcx$ order clusters less. These results
involve comparison of some usual statistical descriptors of spatial
(non-)homogeneity as well as void probabilities.
Examples of comparable pp will be provided in
Section~\ref{ss:sim_ex} and more extensively in Section~\ref{sec:ex}.

\subsection{Statistical descriptors of spatial (non-)homogeneity}
Looking at Figure~\ref{f.Lattice1}, it is  intuitively obvious that
some pp cluster less than others.
However, to the best of our knowledge, there
has been no mathematical formalization of such a statement
based on the theory of stochastic ordering.
In what follows, we present a few reasons to believe that  {\em pp that are smaller in $dcx$ order exhibit less clustering}.

\begin{figure*}[!t]
\begin{center}
\begin{minipage}[b]{0.25\linewidth}
\includegraphics[width=1.\linewidth]{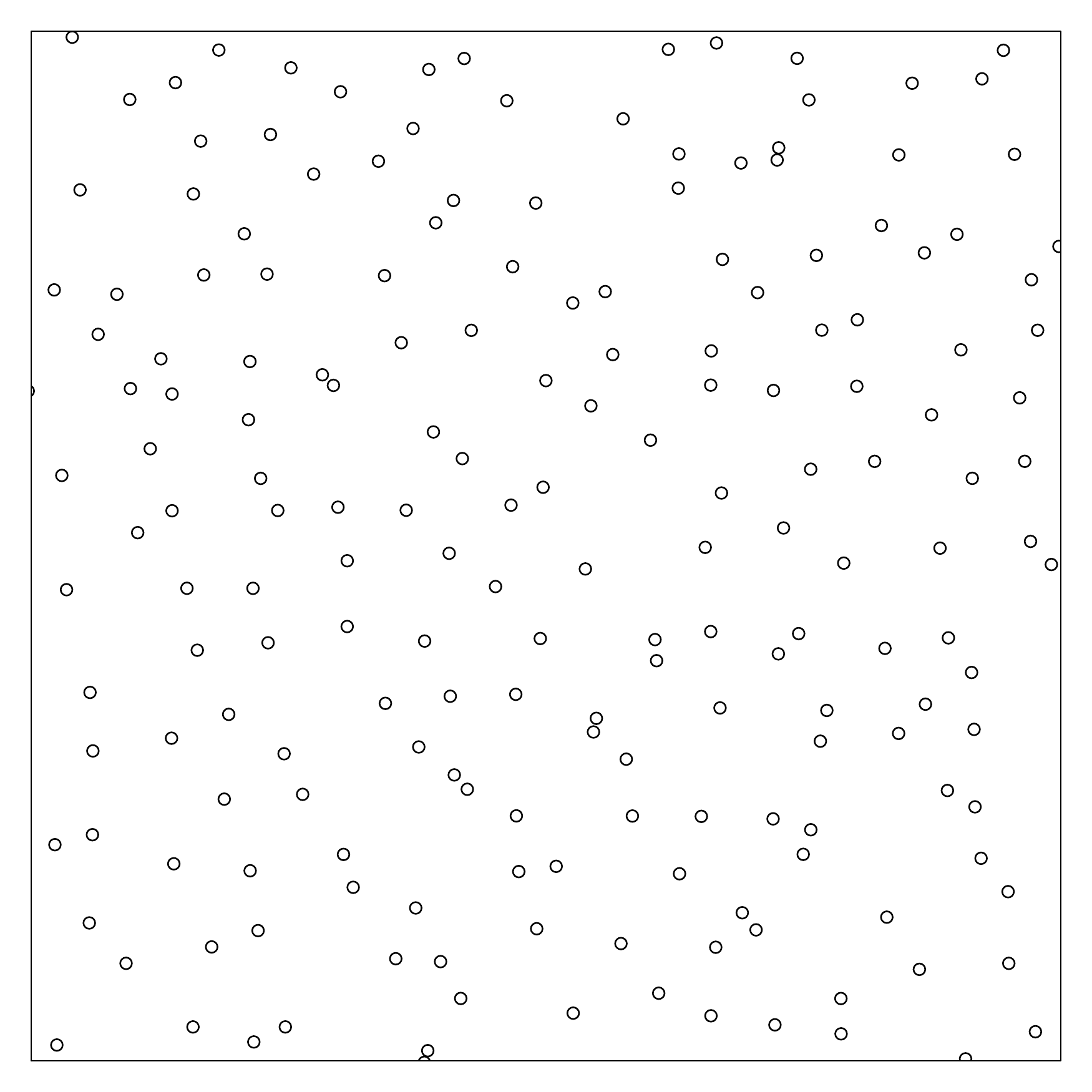}\\
\centerline{simple perturbed lattice}
\end{minipage}
\hspace{2em}
\begin{minipage}[b]{0.25\linewidth}
\includegraphics[width=1.\linewidth]{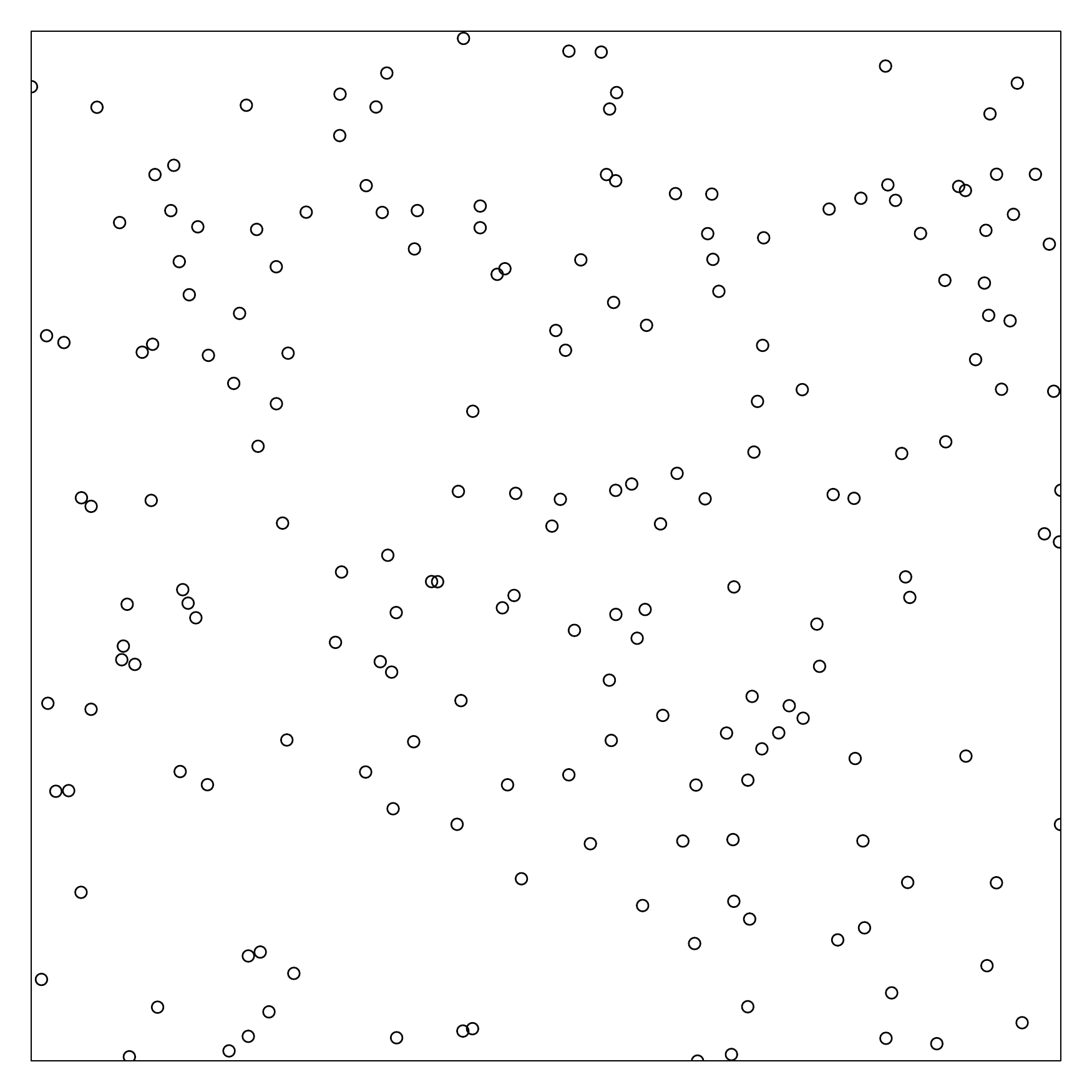}\\
\centerline{Poisson point process}
\end{minipage}
\hspace{2em}
\begin{minipage}[b]{0.25\linewidth}
\includegraphics[width=1.\linewidth]{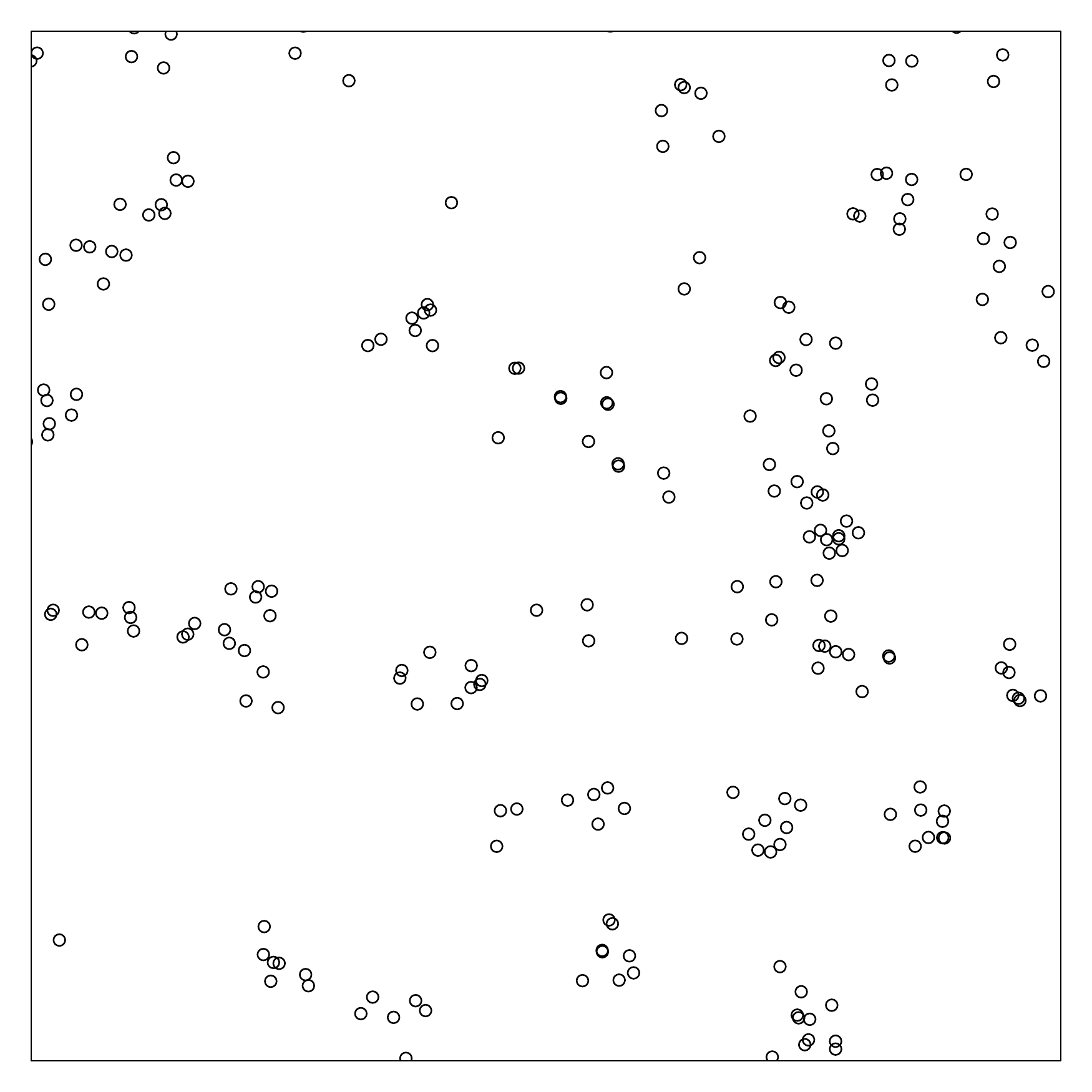}\\
\centerline{Cox point process}
\end{minipage}
\end{center}
\vspace{-2ex}
\caption{\label{f.Lattice1} From left to right : patterns sampled from a simple  i.i.d.  perturbed lattice, Poisson pp and  a doubly stochastic Poisson (Cox) pp, all with the same mean number of points per unit of  surface. For more details see Example~\ref{ex:perturbed.lattice} in  Section~\ref{sec:pl}.}
  \end{figure*}

Roughly speaking, a set of points in $\mR^d$ is ``spatially
homogeneous'' if approximately the same number of points occurs in any
circular region of a given area. A  set of points ``clusters'' if it lacks spatial homogeneity. There exist some statistical descriptors of spatial homogeneity or clustering of pp. We will show that $dcx$  ordering  is consistent with these statistical descriptors. We restrict
ourselves to the stationary setting.

One of the most popular functions for the statistical analysis of
spatial homogeneity is the {\em Ripley's $K$~function} defined for stationary pp (cf~\cite{Stoyan95}). Assume that $\Phi$ is a stationary pp on $\mR^d$ with finite intensity $\lambda=\sE(\Phi([0,1]^d))$. Then
$$K(r):=\frac1{\lambda\|B\|}\sE\Bigl(\sum_{X_i\in\Phi\cap B}
(\Phi(B_{X_i}(r))-1)\Bigr)\,,$$ where
$\|B\|$ denotes the Lebesgue measure of a bBs $B$.
Due to stationarity, the definition does not depend on the choice of $B$.
The following observation was made in~\cite{snorder}.
\begin{prop}\label{prop:Ripley}
Consider two stationary pp $\Phi_j$, $j=1,2$,  with the same finite intensity
and denote by $K_j(r)$ their Ripley's $K$~functions.
If $\Phi_1\leq_{dcx}\Phi_2$ then $K_1(\cdot)\le K_2(\cdot)$.
\end{prop}

Another  useful characteristic for measuring clustering effect in point
processes is the {\em pair correlation function}. It is related to the probability of finding the center of a particle at a given distance from the center of another particle and can be defined on $\mR^{2d}$ as
$g(x,y):=\frac{\rho^{(2)}(x,y)}{\rho^{(1)}(x)\rho^{(1)}(y)},$
where $\rho^{(k)}$ is the {\em $k\,$th  joint intensity}; i.e.,
a function (if it exists) $\rho^{(k)} :(\mR^{d})^{k} \to [0,\infty)$ such that
for any $k$ mutually disjoint subsets
$B_1,\ldots,B_k$, we have that
$\sE{\prod_{i=1}^k \Phi(B_i)} = \int_{B_1\times\dots\times B_k}
\rho^{(k)}(x_1,\ldots,x_k) \md x_1 \ldots \md x_k$
(see
\cite{Stoyan95}).

\medskip
We now make a digression to
joint intensities, which will be used later on.
Recall that the measure  $\alpha^{k}(\cdot)$ defined
by $\alpha^{k}(B_1\times\ldots\times B_k)
=\sE{\prod_{i=1}^k \Phi(B_i)}$ for all (not
necessarily disjoint) bBs $B_i$  ($i=1,\ldots, k$)
is called {\em $k\,$th order moment
measure} of $\Phi$. For simple pp,
the truncation of the measure $\alpha^k(\cdot)$ to the subset
$\{(x_1,\ldots,x_k)\in (\mR^{d})^k: x_i\not=x_j,\; \text{for\;} i\not=j\}$
is equal to the {\em $k\,$th order factorial moment
measure} $\alpha^{(k)}(\cdot)$. Hence $\rho^{(k)}$ (if it exists) is the
density of $\alpha^{(k)}(\cdot)$ for simple pp. Consequently,  the
joint intensities characterize the distribution of a simple pp.
The above facts remain true even when the densities $\rho^k$
are considered with respect to  $\prod_{i=1}^k\mu(dx_k)$ for an arbitrary Radon
measure $\mu$ on $\mR^d$.  We will assume this generality
in Examples~\ref{ex.det} and~\ref{ex.perm}.

The following result was also proved
in~\cite{snorder};~\footnote{$\sigma$-finiteness condition is missing
  there; see~\cite[Prop.~4.2.4]{Yogesh_thesis} for the correction.}.

\begin{prop}
\label{Prop:pcf}
Let $\Phi_1, \Phi_2$ be two pp on $\mR^d$ with $\sigma$-finite
$k\,$th moment measures $\alpha_j^{k}$, $j=1,2$. If $\Phi_1
\leq_{idcx} \Phi_2$ then
 $\alpha_1^{k}(A)\le \alpha_2^{k}(A)$ for all bBs $A\subset (\mR^{d})^k$.
Consequently
$\rho^{(k)}_1(x_1,\ldots,\allowbreak x_k)
\leq  \rho^{(k)}_1(x_1,\ldots,\allowbreak x_k)$
(whenever they exist) for
Lebesgue a.e. $(x_1\ldots,x_k)$. Moreover,
if $\Phi_1 \leq_{dcx} \Phi_2$, then
$g_1(x,y) \leq g_2(x,y)$ for Lebesgue a.e. $(x,y)$.
\end{prop}

Recall that for a (stationary) Poisson pp $\Phi_\lambda$ of
intensity $\lambda$ with $0<\lambda<\infty$, we have that $K(r)\equiv1$ and
$\rho(x,y)\equiv 1$. Poisson pp  is often considered as a
reference model of spatial homogeneity to which other pp
are compared. In Sections \ref{ss.sub-suuper-Poisson}
and \ref{sec:ex}, we will consider pp smaller and larger than $\Phi_\lambda$ in $dcx$ ordering.

\subsection{Void probabilities}
One of the fundamental observations to our approach
is that $dcx$ ordering allows us to compare the void
probabilities of pp $\nu(B)=\pr{\Phi\cap B=\emptyset}$ for all bBs $B$.
\begin{prop}
\label{Prop:voids_pp}
Denote by $\nu_1(\cdot),\nu_2(\cdot)$ the void probabilities of pp
$\Phi_1$ and $\Phi_2$ on $\mR^d$ respectively.
If  $\Phi_1 \leq_{ddcx} \Phi_2$ then
$\nu_1(B) \le \nu_2(B)$
for all bBs $B\subset\mR^d$.
\end{prop}
\begin{proof}
This follows  directly from the definition of $dcx$
ordering of pp, expressing  $\nu_j(B)=\sE(f(\Phi_j(B)))$, $j=1,2$, with the
function $f(x)=\max(0,1-x)$ that is decreasing and convex (so
$ddcx$ in one dimension).
\end{proof}

The latter result can be interpreted as follows:
for two pp that are $ddcx$ ordered, the smaller one has less chance to create a particular hole
(absence of points in a given region). In the case of $dcx$ ordered pp
this holds true even if both pp have the same
expected number of points in any given set. The above result
can be strengthened to the comparison of void probabilities  of Boolean models
having their germ pp $dcx$ ordered. This observation will be the  starting point for our investigation of the connections between percolation and
directionally convex ordering of pp in Section~\ref{s.Boolean}.

\begin{prop}
\label{prop:ord_cap_fnl}
Let  $C(\Phi_j,G)$, $j=1,2$ be two Boolean models with pp of germs
$\Phi_j$, $j=1,2$, respectively, and common distribution of the typical
grain $G$. Assume that $\Phi_i$ are simple and have locally finite moment
measures. If $\Phi_1\le_{dcx}\Phi_2$ then
$\pr{C(\Phi_1,G) \cap B=\emptyset}\le\pr{C(\Phi_2,G) \cap
  B=\emptyset}$
for all bBs $B\subset\mR^d$.
Moreover, if $G$ is a fixed compact grain then the same result holds,
provided $\nu_1(B)\le\nu_2(B)$  for all bBs $B\subset\mR^d$, where $\nu_i$ is
the void probability of $\Phi_i$.
\end{prop}
\begin{proof}
For $x\in\mR^d$, $F$ racs and $B\in\texttt{B}_b^d$, define a
function  $h((x,F),B) = \1[B\cap (x + F)\not=\emptyset]$ and
an {\em extremal
shot-noise} field on $\texttt{B}_b^d$ defined by
$S_i(B):=\allowbreak\sup_{(X^i_j,G^i_j)\in\tilde\Phi_i}h((X^i_j,G^i_j),B)$,
where
$\tilde\Phi_i=\{(X_j^i,G_j^i)\}$ are the pp of germs {\em
independently  marked} by the grains defining the corresponding Boolean
models. By the local finiteness of the two Boolean models,
$\pr{C(\Phi_i,G)\cap B=\emptyset}=\pr{S_i(B)\le0}$
for all bBs $B\subset\mR^d$.
By~\cite[Proposition~3.2, point
2]{snorder}, $\Phi_1\le_{dcx}\Phi_2$ implies
$\tilde\Phi_1\le_{dcx}\tilde\Phi_2$
(\footnote{The assumption on local-finiteness of mean measures of
  $\Phi_i$ is missing in the statement of this result
  in~\cite{snorder}. Also the $idcx/ddcx/idcv$ ordering is mentioned there
  but not proved.}).
Moreover, by Proposition~4.1 in the aforementioned article,
$\pr{S_1(B)\le 0}\allowbreak\le\pr{S_2(B)\le 0}$. This completes the
proof of the first statement.

The proof of the result for the Boolean model with
a fixed compact grain follows immediately from
the equivalence:  $(x+G)\cap B=\emptyset$ if and only if
$x\not\in B+\check G$, where $\check G=\{-y:y\in G\}$.
\end{proof}

\subsection{Sub- and super-Poisson point processes}
\label{ss.sub-suuper-Poisson}
We now concentrate on comparison of pp to the Poisson pp of same
mean measure. To this end, we will call a pp
{\em$dcx$ sub-Poisson} (respectively {\em
 $dcx$  super-Poisson}) if it is smaller (larger) in $dcx$ order than the
Poisson pp (necessarily of the same mean measure).
For simplicity, we will just refer to them as sub-Poisson or super-Poisson pp
omitting the word $dcx$. Examples of such pp
are given in Section~\ref{sec:ex}. In particular,
a rich class of such pp called the perturbed lattice pp will be provided later. Also it was observed in~\cite{snorder} that Poisson-Poisson cluster pp, L\'{e}vy based Cox pp, Ising-Poisson cluster pp etc... are super-Poisson pp.

We will also consider some weaker notions of sub-  or
super-Poisson pp, for which only moment measures or void
probabilities can be compared.
Recall from Propositions~\ref{Prop:pcf} and~\ref{Prop:voids_pp}
that $dcx$ order implies ordering of factorial moment measures and void probabilities
$\nu(\cdot)$. Recall also that a Poisson pp
can be characterized as having void probabilities of the form
$\nu(B)=\exp(-\alpha(B))$. Bearing this in mind, we say that a pp $\Phi$
is {\em weakly sub-Poisson in the sense of void probabilities}
({$\nu$-weakly sub-Poisson} for short) if
\begin{equation}\label{e.nu-weakly-sub-Poisson} \pr{\Phi(B)=0}\le
e^{-\EXP{\Phi(B)}}
\end{equation}
for all Borel sets $B \subset \mR^d$.  When the inequality
in~(\ref{e.nu-weakly-sub-Poisson}) is reversed, we will say that $\Phi$
is {\em weakly super-Poisson in the sense of void probabilities}
({$\nu$-weakly super-Poisson}).

In our study of the percolation
properties of pp, it will be enough to be able to compare moment
measures (off diagonals) of pp. In this regard, we say that a pp
$\Phi$ is {\em weakly sub-Poisson in the sense of moment measures}
({$\alpha$-weakly sub-Poisson} for short) if
\begin{equation}\label{e.alpha-weakly-sub-Poisson}
\EXP{\prod_{i=1}^k\Phi(B_i)}\le \prod_{i=1}^k\EXP{\Phi(B_i)}\,
\end{equation} for all mutually disjoint bBs $B_i\subset\mR^d$.  When
the inequality in~(\ref{e.alpha-weakly-sub-Poisson}) is reversed, we
will say that $\Phi$ is {\em weakly super-Poisson in the sense of
moment measures} ({\em $\alpha$-weakly super-Poisson}).

Finally, we will say that $\Phi$ is {\em weakly sub-Poisson}
if $\Phi$ is $\alpha$-weakly sub-Poisson and $\nu$-weakly sub-Poisson.
Similarly, we define {\em weakly super-Poisson} pp.

\subsection{Simple Examples}
\label{ss:sim_ex}

We give here some examples of weakly super-Poisson pp that follow immediately from earlier results. More detailed examples of sub- and super-Poisson pp shall be presented in Section \ref{sec:ex}.

It can be seen that any {\em associated}
pp~\footnote{i.e, pp $\Phi$ satisfying
\begin{equation}
\label{eqn:association}
\COV{f(\Phi(B_1),\ldots,\Phi(B_k))g(\Phi(B_1),\ldots,\Phi(B_k))} \geq 0
\end{equation}
for any finite collection of bBs $B_1,\ldots,B_k\subset\mR^d$ and
$f,g$ continuous and increasing functions taking values in $[0,1]$; ~\cite{BurtonWaymire1985}.
This property is also called positive correlations, or the FKG property.}
is $\alpha$-weakly super-Poisson. In fact, it is easy to see using
the Jensen's inequality that the associated pp have moment measures
{\em everywhere} (not only off the diagonals) larger than those of
the corresponding Poisson pp~\footnote{Indeed, due to the symmetry of the moment
measures $\alpha^k(\cdot)$, it is enough to verify the inequality on
rectangles $A=B_1^{k_1}\times\dots\times B_n^{k_n}$ for any mutually
disjoint bBs $B_i\subset\mR^d$, any $k_1+\ldots+k_n=k$, $k_i\in\mZ$,
$k,k_i\ge 1$ with $\sum_{i=1}^nk_i=k$.  Recall that for the Poisson pp
of mean measure $\alpha(\cdot),$ we have
$\alpha^{k}(A)= 
\Bigl(\sum_{j=1}^{k_i}{k_i\choose
j}\Bigl(\alpha(B_{k_i})\Bigr)^{j}\Bigr)
$.}.
From~\cite[Th.~5.2]{BurtonWaymire1985}, we know that any {\em Poisson
center cluster pp} is associated.  This is a generalization of our
perturbation~(\ref{defn:perturbed_pp}) of a Poisson pp $\Phi$
(cf. Example~\ref{ex.perturbation_of_Poisson}) having form
$\Phi^{cluster}=\sum_{X\in\Phi}\{X+\Phi_X\}$ with $\Phi_X$ being
arbitrary i.i.d. (cluster) point measures.  Other examples of
associated pp given in~\cite{BurtonWaymire1985} are Cox pp with intensity measures being associated.

It is easy to see by the Jensen's inequality that {\em all Cox pp are $\nu$-weakly super-Poisson} and hence Cox pp with associated intensity measures are weakly super-Poisson.

In Section~\ref{sec:det.perm}, we will show that determinantal and permanental
pp are weakly sub-Poisson and weakly super-Poisson respectively.
Moreover, their  $dcx$ comparison to Poisson pp is possible on mutually
disjoint, {\em simultaneously observable} sets.

\section{Percolation of Boolean models and the $dcx$ ordering of the  underlying point processes} 
\label{s.Boolean}

By {\em percolation} of a Boolean model $C=C(\Phi,G)$
(see Section~\ref{sss.Boolean_model}),
we refer to the existence of an topologically
unbounded connected subset (giant component) of $C$.
We have given heuristic arguments in the Introduction as to
why more clustering in the pp $\Phi$ should make percolation of the corresponding Boolean model $C(\Phi,G)$ less likely. 
Since we have shown in the previous section that $dcx$ order is suitable for
comparing the clustering tendencies of pp, it is tempting to conjecture
that pp smaller in $dcx$ order percolate better. 
Some numerical evidences supporting
this conjecture for a certain class of pp,  called perturbed lattice pp,
will be presented in Section~\ref{ss.NumerialPercol}.
However, as we have also  mentioned in the Introduction, 
this conjecture is not true in full generality; a counterexample
will be presented in Section~\ref{ss.super-percol}.
Nevertheless, some general comparison results regarding nonstandard critical radii for percolation and $dcx$ order of pp shall be proved in this section. 

Define  a {\em critical  radius} for $C(\Phi,G)$ as
\begin{equation}\label{e:rc}
r_c=r_c(\Phi,G):=
\inf \left\{r>0 : \pr{C(\Phi,G^{(r)})\;\text{percolates}} >0 \right\}\,;
\end{equation}
i.e., the smallest $r$ for which
$C(\Phi,G^{(r)})$  percolates with a positive probability,
where $G^{(r)} := \bigcup_{x \in G}B_x(r),$  is the
$r$-th parallel set of $G$.

\smallskip
We say that the \emph{origin percolates} in $C(\Phi,G)$ if $O$
belongs to an infinite component of $C(\Phi,G)$. If $\Phi$ is stationary
then it is easy to show that for any $r>0$,
$C(\Phi,G^{(r)})$ percolates with a positive probability if
and only if the origin percolates with a positive probability. Further, ergodic pp percolate with probability that is either zero or one. 

The critical radius $r_c$ is similar to the critical radius defined
for percolation in various continuum percolation models.
For example, taking  $G=\{O\}$ we obtain the critical radius $r$ for
the percolation of the classical continuum percolation model
(the Boolean model with spherical grains $G=B_O(r)$).

In what follows, we will define two other
critical radii for the percolation of $C(\Phi,G)$. They will
be an upper and a lower bound of $r_c$ respectively.
We will show that the ordering of void probabilities
and moment measures of the pp germs $\Phi$ allows comparison of these
two new critical radii. In this regard, we introduce some more notations.

\subsection{Auxiliary discrete models}
\label{ss.AuxModels}
Though we focus on the percolation of Boolean models (continuum percolation models),
but as is the wont in the subject we shall extensively use discrete
percolation models as approximations.
For $r > 0, x \in \mR^d$, define the following subsets of $\mR^d$. Let
$Q^r := (-\frac{1}{2r},\frac{1}{2r}]^d$,  $Q^r(x) := x + Q^r$,
$Q_r :=  (-r,r]^d$ and $Q_r(x) := x + Q_r$.
We will consider the following  three families of discrete graphs
parametrized by  $n \in \mN$.
\begin{itemize}
\item
$\mL^{*d}_n=(\mZ^d_n, \mE^{*d}_n)$ is the usual  {\em close-packed lattice graph} scaled down by
the factor $1/n$. It has $\mZ^d_n  =\frac{1}{n}\mZ^d$, where $\mZ$ is
the set of integers, as the set of
vertices  and the set of edges $\mE^{*d}_n := \{ \langle z_i,z_j \rangle \in
(\mZ^d_n)^2 :  Q^{\frac{n}{2}}(z_i) \cap Q^{\frac{n}{2}}(z_j) \neq \emptyset
\}$.
\item $\mL^{*d}(r) = (r\mZ^d,\mE^{*d}(r))$ is a similar
close-packed graph on the scaled-up lattice $r\mZ^d$; the edge-set is
$\mE^{*d}(r) := \{ \langle z_i,z_j \rangle \in (r\mZ^d)^2 : Q_r(z_i) \cap
Q_r(z_j) \neq \emptyset \}$.

\item Finally, for a given compact set $K \subset \mR^d$,
the graph $\mL^{*d}_n(K)=(\mZ^d_n,\mE^{*d}_n(K))$ has vertex-set $\mZ^d_n $
and edge-set $\mE^{*d}_n(K) := \{\langle z_i,z_j \rangle : \{Q^n(z_i) + K\}
\cap \{Q^n(z_j) + K\} \neq \emptyset \}$,
where we recall that $A + B = \{a + b : a\in A, b\in B\}$.
\end{itemize}

A path in $\mL^{*d}_n(K)$ from the origin $O\in\mR^d$ to $\partial
Q_m$ is $\gamma = \langle z_1,\ldots,z_{k+1}\rangle \subset \mZ^d_n$
such that for $1 \leq i \leq k$, $z_i \in Q_m$, $\langle z_i,z_{i+1}
\rangle \in \mE^d_n(K)$
and  $O\in \{Q^n(z_1) + K\}$, $\{Q^n(z_{k+1}) +
K\} \cap \partial Q_m \neq \emptyset$. Let
$\Upsilon^n_m(K)$ be the set of all such paths from the origin to
$\partial Q_m$ in $\mL^{*d}_n(K)$.

A {\em contour} in $\mL^{*d}_n$ is a minimal collection of vertices such
that any infinite path in $\mL^{*d}_n$
from the origin has to contain one of these
vertices (the minimality condition implies that the removal of any vertex
from the collection will lead to existence of an infinite path from
the origin without any intersection with the remaining vertices in the
collection). Let $\Gamma_n$ be the set of all contours around the origin in
$\mL^{*d}_n$. For any subset of points $\gamma\subset\mR^d$, in particular for paths
$\gamma \in \Upsilon^n_m(K),\, \Gamma_n$, we define  $Q_{\gamma} = \bigcup_{z \in \gamma} Q^n(z)$.

Throughout the article, we will define several auxiliary site percolation models
on the above graphs by randomly declaring some of their vertices (called also
sites) open.  As usual,
we will say that a given discrete
site percolation model percolates if the corresponding sub-graph consisting of all open sites contains an infinite component. In particular, note that
the existence of only a finite number of
 paths in $ \Gamma_n$ (i.e., contours around the origin in
$\mL^{*d}_n$)  implies
the percolation of the corresponding site percolation model on $\mL^{*d}_n$.
If this number of paths in $\Gamma_n$  has a finite mean
(which is thus a sufficient condition for the percolation)
then we will say that the site percolation model  on $\mL^{*d}_n$
{\em percolates through the Peierls argument} (cf.~\cite[ pp.~17--18]{Grimmett99}).

\subsection{Ordering of void probabilities and percolation}
\label{ss.void-percolation}
Define a new critical radius
\begin{equation}\label{e:urc}
\ovr_c=\ovr_c(\Phi,G):= \inf \Bigl\{r>0 : \text{for all}\;  n\ge1, \sum_{\gamma \in
\Gamma_n}\pr{C(\Phi,G^{(r)}) \cap Q_{\gamma} = \emptyset} <
  \infty\Bigr\}\,.
\end{equation}
It might be seen as the critical radius corresponding
to the phase transition when the discrete
model $\mL^{*d}_n=(\mZ^d_n, \mE^{*d}_n)$,  approximating
$C(\Phi,G^{(r)})$ with an arbitrary precision, starts percolating through
the Peierls argument.
As a consequence, $\ovr_c(\Xi)$ is an upper bound for the actual
critical radii, as we show now.
\begin{lem}
\label{lem:ovcrit_rad_ub}
Let  $C(\Phi,G)$, be a Boolean model with the pp of germs
$\Phi$ and $G$ as the distribution of the typical
grain. Then, $\ovr_c(\Phi,G) \geq r_c(\Phi,G)$.
\end{lem}
\begin{proof}
The proof as indicated above shall use Peierls argument on a suitable
discrete approximation of the model on $\mL^{*d}_n$.
Define a random field $\cX_r := \{X_r(z)\}_{z \in \mZ^d_n}$ such that
$X_r(z) = \1[C(\Phi,G^{(r)}) \cap Q^n(z) \neq \emptyset]$. By Peierls
argument for $r > \ovr_c$, for all $n$, $\cX_r$ percolates in $\mL^{*d}_n$ a.s.. Thus for all $n$, $C(\Phi,G^{(r+\frac{\sqrt{d}}{n})})$ percolates
a.s.
and so $r+\frac{\sqrt{d}}{n} > r_c$. Hence, $r \geq
r_c$.
\end{proof}
Here is the main observation of this section.
It follows directly from the definition of $\ovr_c$ and Proposition~\ref{prop:ord_cap_fnl}.

\begin{cor}
\label{cor:ord_crit_rad_racs}
Let  $C_j=C(\Phi_j,G)$, $j=1,2$ be two Boolean models with pp of germs
$\Phi_j$, $j=1,2$, respectively, and $G$ as the common distribution of the typical grain. Assume that $\Phi_i$ are simple and have locally finite moment
measures. If $\Phi_1\le_{dcx}\Phi_2$ then $\ovr_c(\Phi_1,G)\le\ovr_c(\Phi_2,G)$.
Moreover, if $G$ is a fixed compact grain then the same result holds
provided $\nu_1(B)\le\nu_2(B)$  for all bBs $B\subset\mR^d$, where $\nu_i$ is
the void probability of $\Phi_i$.
\end{cor}

\begin{rem}
The critical radii $r_c=r_c(\Xi)$ and  $\ovr_c=\ovr_c(\Xi)$ can be
defined for an arbitrary racs $\Xi$
by replacing $C(\Phi,G^{(r)})$ by~$\Xi^{(r)}$ in~(\ref{e:urc}) and~(\ref{e:rc}),
respectively.  The result of Lemma~\ref{lem:ovcrit_rad_ub} remains
true (cf. the proof). Moreover, if two racs $\Xi_1,\Xi_2$ are ordered with
respect to their void probabilities; i.e., $\pr{\Xi_1\cap
  K=\emptyset}\le \pr{\Xi_2\cap K=\emptyset}$ for any compact set $K$,
then $\ovr_c(\Xi_1)\le \ovr_c(\Xi_2)$.
\end{rem}

\subsection{Comparison of moment measures and percolation}
\label{ss.moment-percolation}
We will define now another critical radius for $C(\Phi,G)$ by counting
some paths on it. In this regard,
define {\em an open path $\gamma$ from the origin to
$\partial Q_m$ in $C(\Phi,G)$} (recall the notation from  Section~\ref{ss.AuxModels})
as an (ordered) subset $\langle
X_1,\ldots,X_{k+1}\rangle$ of distinct points of $\Phi$ such that for all
$1 \leq i < k$, $\{X_i\}_{ i \geq 1} \subset Q_m $, $\{X_{i+1} +
G_{i+1}\} \cap \{X_i + G_i\} \neq \emptyset$ and $\{X_{k+1} +
G_{k+1}\} \cap \partial Q_m \neq \emptyset$ , $O \in \{X_1 +
G_1\}$. Let $N_m(\Phi,G)$ be the number of such open paths from the origin
to $\partial Q_m$.

Define the  following new   critical  radius : 
\begin{equation}\label{e:uurc}
\ur_c=\ur_c(\Phi,G):=\inf \left\{r>0 : \liminf_{m\to\infty}
\EXP{N_m(\Phi,G^{(r)})}
 > 0 \right\}\,.
\end{equation}
It corresponds to phase transitions in the Boolean model when the
{\em expected} number of open paths from the origin to the boundary
of an arbitrary large box becomes positive. It is easy to see that
$\ur_c$ is a lower bound for $r_c$.

\begin{lem}
\label{lem:ucrit_rad_ub}
Assume a stationary pp $\Phi$ of germs and let
the typical grain  $G$ be an almost surely connected set.
Then  $\ur_c(\Phi,G)\leq r_c(\Phi,G)$.
\end{lem}
\begin{proof}
By the stationarity of $\Phi$, for any
 $r>0$, $C(\Phi,G^{(r)})$ percolates with positive probability if
and only if  the origin percolates in $C(\Phi,G^{(r)})$ with positive probability.
Moreover, by the connectivity of grains, this latter statement is
equivalent to the statement that $\pr{N_m(\Phi,G^{(r)}) \geq 1} > 0$ for
all $m\ge1$. Consequently, since the above
sequence of events is decreasing in $m$,
$$r_c(\Phi,G) = \inf \Bigl\{r>0 : \liminf_{m\to\infty}
\pr{N_m(\Phi,G^{(r)}) \geq 1} > 0 \Bigr\}\,.$$
Now, the proof follows by the Markov's inequality.
\end{proof}

Here is the main result of this section.
\begin{prop}
\label{prop:lower_crit_rad}
Let  $C_j=C(\Phi_j,G)$, $j=1,2$ be two Boolean models with the same fixed
(deterministic) grain $G$ and simple pp of germs
$\Phi_j$, $j=1,2,$ with $\sigma$-finite
$k\,$th moment measures $\alpha_j^{k}$, respectively, for all $k\ge1$.
If $\alpha_1^{k}(B_1\times\dots\times B_k)\le
\alpha_2^{k}(B_1\times\dots\times B_k)$
for all $k\ge1$ and all mutually disjoint $B_i$, $i=1,\ldots,k$,
then $\ur_c(\Phi_1,G)\ge \ur_c(\Phi_2,G)$.
In particular, the above inequality is true if
$\Phi_1\le_{dcx}\Phi_2$.
\end{prop}
Note the reversal of the inequality in the above Proposition for the
critical radii $\ur_c$ with respect to  
the result of Corollary~\ref{cor:ord_crit_rad_racs} regarding $\ovr_c$.
A similar result can also be proved for the critical radius
defined with respect to the expected number  $N^*_n$,
of paths of length $n$ in the Boolean model, as well as for the
expected number of crossings of a rectangle~\footnote{Similarly
to open paths from the origin to
$\partial Q_m$, one can define an {\em open path on the germs of $\Phi$
crossing the rectangle
 $ [0,m] \times
[0,3m] \times \ldots \times [0,3m]$ across the shortest side} and
define yet another critical radius $r_s$ as the smallest $r$  for which such a path
exists with positive probability for an arbitrarily large $m$.
If all grains are {\em path-connected}
then the existence of such a path on the germs of $\Phi$
is equivalent to the existence of a crossing defined as a
continuous curve in $C(\Phi,G)$. This consistency
allows us to relate this new critical radius $r_s$ with
the critical intensity defined in~\cite[(3.20)]{MeeRoy96}.
Hence, by \cite[Proposition 3.5]{MeeRoy96}, $r_s= r_c$ for Boolean models
with an homogeneous Poisson pp of germs and spherical grains of bounded radius (possibly random). Moreover, replacing the probability of the existence
of such crossing by the expected number of crossings, we can define a lower
bound $\ur_s\le r_s$ in the same way as we have defined $\ur_c\le
r_c$.}.

\begin{proof}[Proof of Proposition \ref{prop:lower_crit_rad}]
Define $r_m^* := \inf \Bigl\{r : \EXP{N_m(\Phi_2,G^{(r)})} = \infty
\Bigr\}$ for any $m \geq 1$. (Recall that $N_m(\Phi_i,G^{(r)})$ is the
number of open paths from the origin to
$\partial Q_m$ in $C(\Phi_i,G^{(r)}$)). We will first prove that the
condition
\begin{equation}\label{e.Mm-inequality}
\EXP{N_m(\Phi_1,G^{(r)})} \leq
\EXP{N_m(\Phi_2,G^{(r)})}\,,\quad\text{for all $m\ge1$ and $r \neq r_m^*$}\,
\end{equation}
implies $\ur_c(\Phi_1,G) \geq \ur_c(\Phi_2,G)$.

Indeed, assume~(\ref{e.Mm-inequality}) and suppose that
$\ur_c(\Phi_1,G)\allowbreak < \ur_c(\Phi_2,G)$. Then choose $r \not\in
\{r^*_m\}_{m \in \mN}$ such that $\ur_c(\Phi_1,G) < r <
\ur_c(\Phi_2,G)$ and note that (\ref{e.Mm-inequality}) implies $0 < \liminf_m \EXP{N_m(\Phi_1,G^{(r)})} \leq
\liminf_m \EXP{N_m(\Phi_2,G^{(r)})}$. This contradicts $r <
\ur_c(\Phi_2,G)$ and hence $\ur_c(\Phi_1,G) \geq \ur_c(\Phi_2,G)$. 

Note that (\ref{e.Mm-inequality}) is true for $r > r_m^*$. Thus, in order to prove~(\ref{e.Mm-inequality}), we will work with a fixed $m\ge1$ and $r < r_m^*$ for the remaining part  of the proof.

Let $\Phi_1$ and $\Phi_2$ be the given independent pp of germs
and $G$ be the fixed (deterministic) grain. Let $R$ be such
$G\subset B_O(R)$, where $B_O(R)$ is the ball centered at the origin
of radius $R$. The first observation is that in order to study the paths from the origin to $\partial Q_m$ on $C(\Phi_i,G^{(s)})$, $i=1,2$, and on the discrete graph $\mL^{*d}_n(G^{(s)})$,
$n\ge1$,  with $0\le s\le r+\sqrt d/2,$ it is
enough to consider these models only in $Q_{m'}$ with $m'=m+R+r+\sqrt
d$.

Let $n_0$ be the  smallest $n\ge1$ such that for all $z\in Q_{m'}\cap\mZ^d_n$
either $(\Phi_1+\Phi_2)(Q^n(z)) \leq 1$ or $(\Phi_1+\Phi_2)(Q^n(z)) =
2$ and $\Phi_1$ and $\Phi_2$ share a common atom in $Q^n(z)$.
The random variable $n_0$ is finite since both $\Phi_1$ and $\Phi_2$ are simple pp.

Let us now remark the following two useful relations between Boolean models
$C(\Phi_i,G^{(s)})$, $i=1,2$, and their discrete graph approximations
$\mL^{*d}_n(G^{(s)})$, $n\ge1$.
\begin{itemize}
\item For any $n\ge1$, 
\begin{equation}\label{e.Nm<Mmn*}
N_m(\Phi_i,G^{(r)})\1(n \ge n_0)  \leq
N_m^n(\Phi_i,G^{(r)})\1(n\ge n_0)\,,
\end{equation}
where for a compact set $K$, $N_m^n(\Phi_i,K) = \sum_{\gamma \in
  \Upsilon^n_m(K)} \prod_{z \in \gamma}\Phi_i(Q^n(z))$,
and  (recall from Section~\ref{ss.AuxModels}) $\Upsilon^n_m(K)$
denotes the set of open paths from
the origin to $\partial Q_m$ in the graph $\mL^{*d}_n(K)$.
This follows from the fact that on the event $\{n\ge n_0\}$,
for any $z\in\mZ^d_n$, if $\Phi_i(Q^n(z))>0$ then $\Phi_i$ has at most
one point in $Q^n(z)$. Moreover, for any
two distinct points  $X_1,X_2\in\Phi_i$, $X_1\not=X_2$, necessarily
belonging to two different boxes $Q^n(z_1)$, $Q^n(z_2)$, respectively,
with $z_1\not=z_2$, which  satisfy $\{X_1+G^{(r)}\}\cap
\{X_2+G^{(r)}\}\not=0$, we have that $\{Q^n(z_1)+G^{(r)}\}\cap
\{Q^n(z_2)+G^{(r)}\}\not=\emptyset$.
\item For any $n\ge1$ and $s(n)=r+\sqrt d/(2n)$,
\begin{equation}\label{e.Nsn}
N_m^n(\Phi_i,G^{(r)}) \leq
N_m(\Phi_i,G^{(s(n))})\,.
\end{equation}
This is so because for any $X\in\Phi_i\cap Q^n(z)$ for some
$z\in\mZ^d_n$, we have $\{Q^n(z)+G^{(r)}\}\subset \{X+G^{(s(n))}\}$
and, moreover, in $N_m^n(\Phi_i,G^{(r)})$ we are counting only
the paths of $C(\Phi_i,G^{(s(n))})$ from $O$ to $\partial Q_m$, which
do not have edges between points in the same box $Q^n(z)$ for some $z \in \mZ^d_n$.
\end{itemize}
The following
observation regarding the discrete models is crucial to our proof.
\begin{eqnarray}
\EXP{N_m^n(\Phi_1,G^{(r)})} & = & \sum_{\gamma \in \Upsilon_m^n(G^{(r)})} \EXP{N_{\gamma}(\Phi_1)} \non \\
& \leq & \sum_{\gamma =\langle 0,z_0,z_1,\ldots,z_k\rangle \in  \Upsilon_m^n(G^{(r)})}\EXP{\prod_{i=1}^k \Phi_2(Q^n(z_i))} \non \\
\label{eqn:ineq_no_paths_disc}& =& \EXP{N_m^n(\Phi_2,G^{(r)})} \,
\end{eqnarray}
where the inequality is due to our assumption on the moment measures.
Using~(\ref{e.Nm<Mmn*}) and   Fatou's lemma we get
\begin{eqnarray}
\liminf_{n\to\infty}\EXP{N_m^n(\Phi_1,G^{(r)})}&\ge&
\EXP{\liminf_{n\to\infty}N_m(\Phi_1,G^{(r)})\1(n \ge n_0)\}} \no \\
&=&\EXP{N_m(\Phi_1,G^{(r)})}\, \label{eqn:liminf_EN_m}.
\end{eqnarray}
For $r < r^*_m,$ choose $n_1$ such that $s(n_1)< r^*_m$.
Note that for $n \ge n_1$, again by~(\ref{e.Nsn})
we have
$$N_m^n(\Phi_2,G^{(r)})\le N_m(\Phi_2,G^{(s(n))})\le
N_m(\Phi_2,G^{(s(n_1))})$$
and, by the definition of $ r^*_m$,
$\EXP{N_m(\Phi_2,G^{(s(n_1))})}<\infty$. Hence, using reverse Fatou's
lemma, (\ref{e.Nsn}), (\ref{eqn:ineq_no_paths_disc}) and (\ref{eqn:liminf_EN_m}), we get that
\begin{eqnarray*}
\EXP{N_m(\Phi_1,G^{(r)})} & \leq &  \limsup_{n\to\infty}\EXP{N_m^n(\Phi_1,G^{(r)})} \\
& \leq & \limsup_{n\to\infty}\EXP{N_m^n(\Phi_2,G^{(r)})} \\ &\leq&
\EXP{\limsup_{n\to\infty}N_m^n(\Phi_2,G^{(r)})}\\
&\le& \EXP{\limsup_{n\to\infty} N_m(\Phi_2,G^{(s(n))})}\\
& = &\EXP{N_m(\Phi_2,G^{(r)})}\,,
\end{eqnarray*}
where the last equality follows from the fact that
$\bigcap_{n=1}^\infty C(\Phi_2,G^{(s(n))})=C(\Phi_2,G^{r})$, since
Boolean model is a closed set.
Thus we have shown~(\ref{e.Mm-inequality}), which concludes the proof.
\end{proof}
\begin{rem}\label{rem:sandwich}
Combining the results of Corollary~\ref{cor:ord_crit_rad_racs} and
Proposition~\ref{prop:lower_crit_rad} regarding percolation of the two
Boolean models $C(\Phi_i,G)$, with a fixed connected grain $G$
and  simple pp of germs  $\Phi_i$ of locally
finite moment measures,
we obtain the following sandwich
inequality  for the critical radius $r_c(\Phi_1,G)$ of the Boolean model
$C(\Phi_1,G)$
$$\ur_c(\Phi_2,G) \leq \ur_c(\Phi_1,G) \leq r_c(\Phi_1,G) \leq \ovr_c(\Phi_1,G) \leq \ovr_c(\Phi_2,G)\,,$$
provided $\Phi_1$ is stationary and has smaller void probabilities and
factorial moment measures than $\Phi_2$.
Assume moreover that $\ur_c(\Phi_2,G)>0$ and $\ovr_c(\Phi_2,G)<\infty$.
Then the above sandwich inequality implies that $C(\Phi_1,G)$
exhibits a non-trivial phase transition relative to the percolation
radius; i.e.,  $0<r_c(\Phi_1,G)<\infty$.
\end{rem}

\section{Non-trivial phase transition for percolation models on
 sub-Poisson point processes}
\label{sec:sub-poisson}

Following  Remark~\ref{rem:sandwich},
with $\Phi_2$ being a homogeneous Poisson pp
$\Phi_\lambda,$ one could prove the existence of a non-trivial phase
transition in the Boolean model driven by a sub-Poisson pp
(or weakly sub-Poisson pp in the case of fixed grains) $\Phi_1,$ by
showing that $\ur_c(\Phi_\lambda,G)>0$ and
$\ovr_c(\Phi_\lambda,G)<\infty$. Even for a Poisson pp, a direct verification of these latter conditions
evades us. However, the
existence of a nontrivial phase transition for the percolation of Boolean model
and other models driven by sub-Poisson pp can be shown and that shall be the goal of this section.

We will be particularly interested in  percolation models on level-sets of additive shot-noise fields.
The rough idea is as
follows: level-crossing
probabilities for these models can be bounded using Laplace transform
of the underlying pp. For sub-Poisson pp (pp that are $dcx$ smaller than
Poisson pp), this can further be bounded by the Laplace transform of
the corresponding Poisson pp, which has a closed-form expression.
For  'nice' response functions of the shot-noise,
these expressions are amenable enough to deduce the asymptotic bounds
on the expected number of closed contours around the origin or the
expected number of open paths of a given length from the origin
and thus, using standard arguments, deduce percolation or
non-percolation of a suitable discrete approximation of the model.
In what follows we shall carry out this program for $k$-percolation
in the Boolean model and percolation in the SINR model.
For the similar study of word percolation
see~\cite[Section~6.3.3]{Yogesh_thesis}. 

\subsection{Bounds in discrete models}
\label{sec:bds_disc_models}

We shall start with a generic bound on a discrete model which
shall be used to prove  bounds in the continuum models.
Denote by  $V_{\Phi}(x) := \sum_{X \in \Phi}\ell(x,X)$ the (additive)
shot-noise field generated by a pp $\Phi$
and a non-negative response function $\ell(\cdot,\cdot)$ defined on
$\mR^d\times\mR^d$.
 Define the corresponding lower and upper level sets of this
 shot-noise field on  the lattice $r\mZ^d$ by
$\mZ^d_r(V_{\Phi},\leq h) := \{z \in r\mZ^d : V_{\Phi}(z) \leq h \}$
and $\mZ^d_r(V_{\Phi},\geq h) := \{z \in r\mZ^d : V_{\Phi}(z) \geq h
\}$. We will be interested in percolation of
$\mZ^d_r(V_{\Phi},\leq h)$ and $\mZ^d_r(V_{\Phi},\geq
h)$ understood in the sense of site-percolation of the close-packed
lattice $\mL^{*d}(r)$ (cf Section~\ref{ss.AuxModels}).

\begin{rem}\label{rem:perc-standard-arguments}
Recall that the number of  contours surrounding
the origin  in $\mL^{*d}(r)$
~\footnote{A contour  surrounding
the origin in $\mL^{*d}(r)$ is a
minimal collections of vertices of $\mL^{*d}(r)$ such
that any infinite path on this graph
from the origin has to contain one of these
vertices.}
is at most $n(3^d-2)^{n-1}$.
Hence, in order to prove percolation of a
given model using Peierls argument, it is enough to show that
the corresponding probability of having $n$ distinct sites simultaneously closed is
smaller than $\rho^{n}$ for some $0\le\rho<(3^d-2)^{-1}$ for $n$ large enough.
Similarly, since the number of paths of length $n$ starting from the
origin is at most $(3^d-1)^n$, in order to disprove percolation of a
given model it is enough to show that
the corresponding probability of having $n$ distinct sites simultaneously open is
smaller than $\rho^{n}$ for some $0\le\rho<(3^d-1)^{-1}$ for $n$ large
enough.%
\footnote{The bounds $n(3^d-2)^{n-1}$ and $(3^d-1)^n$ are not tight;
  we use them for simplicity of exposition. For more about the former bound, refer \cite{Lebowitz98,Balister07}.}
\end{rem}

The following result allows us to derive the afore-mentioned bounds.
We restrict ourselves to the stationary case.
\begin{lem}
\label{lem:gen_bds}
Let $\Phi$ be a stationary pp and $V_{\Phi}(\cdot)$,
$\mZ^d_r(V_{\Phi},\leq h)$, $\mZ^d_r(V_{\Phi},\geq h)$ be as
defined above. Let $\Phi_\lambda$ be the homogeneous Poisson pp with
intensity $\lambda$ on $\mR^d$.
If $\Phi \leq_{idcx} \Phi_{\lam}$ then for any $s > 0$,
\begin{equation}
\label{lem:bd_uls}
\pr{V_{\Phi}(z_i) \geq h , 1 \leq i \leq n} \leq e^{-snh} \exp\left\{\lam \int_{\mR^d}(e^{s\sum_{i=1}^n\ell(x,z_i)}-1) \md x \right\}.
\end{equation}
If $\Phi \leq_{ddcx} \Phi_{\lam}$ then for any $s > 0$,
\begin{equation}
\label{lem:bd_lls}
\pr{V_{\Phi}(z_i) \leq h , 1 \leq i \leq n} \leq e^{snh} \exp\left\{\lam \int_{\mR^d}(e^{-s\sum_{i=1}^n\ell(x,z_i)}-1) \md x \right\}.
\end{equation}
\end{lem}

\begin{proof}
In order to prove the first  statement, observe by Chernoff's
inequality that for any $s > 0,$
\begin{eqnarray*}
\pr{V_{\Phi}(z_i) \geq h , 1 \leq i \leq n} & \leq & \pr{\sum_{i=1}^nV_{\Phi}(z_i) \geq nh } \\
& \leq & e^{-snh}\EXP{\exp\left\{s\sum_{i=1}^nV_{\Phi}(z_i)\right\}}
\\
& \leq &
e^{-snh}\EXP{\exp\left\{s\sum_{i=1}^nV_{\Phi_{\lam}}(z_i)\right\}}%
\\
& = & e^{-snh}\exp\left\{- \lam \int_{\mR^d}(1 -
e^{s\sum_{i=1}^n\ell(x,z_i)}) \md x\right\}\,, 
\end{eqnarray*}
where the third inequality follows from the ordering
$\Phi\le_{idcx}\Phi_\lambda$ and the equality by the known
representation of the Laplace transform of a functional of Poisson
pp (cf~\cite[eqn.~9.4.17 p.~60]{DVJII2007}).

The proof of the second statement follows along the same lines
by noting that for any random variable $X$ and any  $a \in \mR, s > 0$,
$\pr{X \leq a} = \pr{e^{-sX} \geq e^{-sa}} \leq e^{sa} \EXP{e^{-sX}}$.
\end{proof}

\subsection{$k$-percolation in Boolean model}
\label{sec:k_perc}
By $k$-percolation in a Boolean model, we understand percolation of the subset of
the space covered by at least $k$ grains of the Boolean model.
The aim of this section is to show that for sub-Poisson pp (i.e, pp
that are $dcx$-smaller than Poisson pp), the critical
intensity for $k$-percolation of the Boolean model is non-degenerate. The result for $k=1$
(i.e., the usual percolation) holds under a weaker
assumption of ordering of void probabilities and factorial moment measures.

Given a  pp of germs $\Phi$  and the distribution of the
 typical grain  $G$, extending the definition of the Boolean model,
we define the coverage field $V_{\Phi}(x) :=
 \sum_{X_i \in \Phi} \1[x \in X_i + G_i]$,  where
given $\Phi$, $\{G_i\}$ is a sequence of i.i.d. racs, distributed as
$G$. The $k$-covered set is defined as
$C_k(\Phi,G) := \{x : V_{\Phi}(x) \geq k \}$.
Note that $C_1(\Phi,G)=C(\Phi,G)$ is the Boolean model considered in
Section~\ref{s.Boolean}.
For $k\ge 1$, define the {\em critical radius for $k$-percolation} as
$$ r^k_c(\Phi,G) := \inf \{r : \pr{\mbox{$C_k(\Phi,G^{(r)})$
    percolates}} > 0 \}\,,$$
where, as before, percolation means existence of an unbounded connected subset.
Clearly, $r_c^1(\Phi,G) = r_c(\Phi,G) \leq r_c^k(\Phi,G)$.
Again,  we shall abbreviate $C_k(\Phi,B_O(r))$ by $C_k(\Phi,r)$ and $r^k_c(\Phi,\{O\})$ by $r^k_c(\Phi)$.
\begin{prop}
\label{thm:k_perc_sub-Poisson_pp}
Let $\Ph$ be a stationary pp. For $k \geq 1, \lam > 0$, there exist constants $c(\lam)$ and
$c(\lam,k)$ (not depending on the distribution of $\Phi$) such that
$0 < c(\lam) \leq r_c^1(\Phi)$ provided  $\Phi \leq_{idcx}
\Phi_{\lam}$ 
and $r_c^k(\Phi) \leq c(\lam,k) < \infty$ provided
$\Phi \leq_{ddcx} \Phi_{\lam}$.
Consequently, for $\Phi \leq_{dcx} \Phi_{\lam}$ combining both the
above statements,
we have that
$$ 0 < c(\lam) \leq r_c^1(\Phi) \leq r_c^k(\Phi) \leq c(\lam,k) < \infty. $$
\end{prop}
\begin{rem}
\begin{enumerate}
\item More simply, the theorem gives an upper and lower bound for the critical radius of a sub-Poisson pp dependent only on its mean measure (as this determines the $\lam$ in $\Phi_{\lam}$) and not on the finer structure.

\item The following extensions follow by obvious coupling arguments. For a racs $G \subset B_O(R)$ $a.s.$ with $R \geq 0 $, we have that $r_c(\Phi,G) \geq r_c(\Phi,B_O(R)) \geq (c(\lam) - R)^+$ and for a racs $G \supset B_O(R)$ $a.s.$ with $R \geq 0$, $r^k_c(\Phi,G) \leq r^k_c(\Phi,B_O(R)) \leq (c(\lam,k) - R)^+$.

\end{enumerate}

\end{rem}

\begin{proof}[Proof of Proposition~\ref{thm:k_perc_sub-Poisson_pp}]
In order to prove the first statement, let $\Phi \leq_{idcx}
\Phi_{\lam}$ and $r > 0$. Consider the close
packed lattice $\mL^{*d}(2r)$. Define the response function $l_r(x,y)
:= \1[x \in Q_r(y)]$ and the corresponding shot-noise field
$V^r_{\Phi}(z)$ on $\mL^{*d}(2r)$. Note that if $C(\Phi,r)$ percolates
then $\mZ^d_{2r}(V^r_{\Phi},\geq 1)$ percolates as well.
We shall now show that there exists a $r >0$ such that
$\mZ^d_{2r}(V^r_{\Phi},\geq 1)$ does not percolate.
For any $n$ and $z_i \in r\mZ^d, 1 \leq i \leq n$, $\sum_{i=1}^n l_r(x,z_i) =  1$ iff $x \in \bigcup_{i=1}^n Q_r(z_i)$ and else $0$. Thus, from Lemma \ref{lem:gen_bds}, we have that
\begin{eqnarray}
\pr{V^r_{\Phi}(z_i) \geq 1 , 1 \leq i \leq n} & \leq & e^{-sn}
\exp\left\{\lam \int_{\mR^d}(e^{s\sum_{i=1}^nl_r(x,z_i)}-1) \md
  x\right\}\, ,
  \non \\
& = & e^{-sn} \exp\left\{\lam \|\bigcup_{i=1}^n Q_r(z_i)\| (e^s-1)
\right\} \md x\, , \non \\
\label{eqn:gen_bd_non_perc} & = & (\exp\{-(s+ (1-e^s)\lam (2r)^d) \})^n.
\end{eqnarray}
Choosing $s$ large enough that $e^{-s} < (3^d -1)^{-1}$ and then by
continuity of $(s+ (1-e^s)\lam (2r)^d)$ in $r$, we can choose a
$c(\lam) > 0$ such that for all $r < c(\lam)$, $\exp\{-(s+ (1-e^s)\lam
(2r)^d))\} < (3^d -1)^{-1}$. Now, using the standard argument
involving the expected number of open paths (cf Remark~\ref{rem:perc-standard-arguments}),
we can show non-percolation of $\mZ^d_{2r}(V^r_{\Phi},\geq 1)$ for $r
< c(\lam)$. Hence for all $r < c(\lam)$,  $C(\Phi,r)$ does not
percolate and so $c(\lam) \leq r_c(\Phi)$.

For the second statement, let $\Phi \leq_{ddcx} \Phi_{\lam}$.
Consider the close packed lattice
$\mL^{*d}(\frac{r}{\sqrt{d}})$. Define the response function $l_r(x,y)
:= \1[x \in Q_{\frac{r}{2\sqrt{d}}}(y)]$ and the corresponding
additive shot-noise field $V^r_{\Phi}(z)$ on
$\mL^{*d}(\frac{r}{\sqrt{d}})$. Note that $C_k(\Phi,r)$ percolates if
$\mZ^d_{\frac{r}{\sqrt{d}}}(V^r_{\Phi},\geq \lceil k/2 \rceil)$
percolates,
where $\lceil a \rceil=\min\{z\in\mZ: z\ge a\}$.
We shall now show that there exists a $r < \infty$ such that $\mZ^d_{\frac{r}{\sqrt{d}}}(V^r_{\Phi},\geq \lceil k/2 \rceil)$ percolates.
For any $n$ and $z_i, 1 \leq i \leq n$, from Lemma \ref{lem:gen_bds}, we have that
\begin{eqnarray}
\lefteqn{\pr{V^r_{\Phi}(z_i) \leq \lceil k/2 \rceil - 1 , 1 \leq i
    \leq n}}\non \\
 & \leq & e^{sn(\lceil k/2 \rceil - 1)} \exp\left\{\lam \int_{\mR^d}(e^{-s\sum_{i=1}^nl_r(x,z_i)}-1)\, \md x \right\}\ \non \\
& = & e^{sn(\lceil k/2 \rceil - 1)} \exp\left\{\lam \|\bigcup_{i=1}^n Q_{\frac{r}{2\sqrt{d}}}(z_i)\| (e^{-s}-1)  \,\md x\right\} \non \\
\label{eqn:gen_bd_perc} & = & (\exp\{-((1-e^{-s})\lam (\frac{r}{\sqrt{d}})^d - s(\lceil k/2 \rceil - 1)) \})^n.
\end{eqnarray}
For any $s$, there exists $c(\lam,k,s) < \infty$ such that for all $r
> c(\lam,k,s)$, the last term in the above equation is strictly less
than $(3^d-1)^{-n}$. Thus one can use
the standard argument involving the expected number of closed contours
around the origin
(cf Remark~\ref{rem:perc-standard-arguments})
to show that $\mZ^d_{\frac{r}{\sqrt{d}}}(V^r_{\Phi},\geq \lceil k/2
\rceil )$ percolates.
 Further defining $c(\lam,k) := \inf_{s > 0} c(\lam,k,s)$, we have that $C_k(\Phi,r)$ percolates for all
$r > c(\lam,k)$. Thus $r^k_c(\Phi) \leq c(\lam,k)$.
\end{proof}
For $k=1$; i.e., for the usual percolation in Boolean model,
we can avoid the usage of exponential estimates of
Lemma~\ref{lem:gen_bds}  and work with void probabilities and factorial
moment measures only, as in Section~\ref{s.Boolean}.
The gain is two-fold: we extend the result to weakly sub-Poisson pp
(cf. Section~\ref{ss.sub-suuper-Poisson})
and moreover, improve the bounds on the critical radius.
\begin{prop}
\label{prop:Peierls_arg_bd_pp}
Let $\Phi$ be a stationary pp of intensity $\lambda$ and $\nu$-weakly
sub-Poisson (i.e., it has void probabilities smaller than those of
$\Phi_\lambda$).
Then $r_c(\Phi) \leq \tilde{C}(\lam) := \sqrt{d}\left(\frac{\log (3^d-2)}{\lam}\right)^{1/d} \leq c(\lam,1) < \infty$.
\end{prop}
\begin{proof}
As in the second part of the proof of Theorem
\ref{thm:k_perc_sub-Poisson_pp}, consider the close packed lattice
$\mL^{*d}(\frac{r}{\sqrt{d}})$. Define the response function $l_r(x,y)
:= \1[x \in Q_{\frac{r}{2\sqrt{d}}}(y)]$ and the corresponding
extremal shot-noise field $U^r_{\Phi}(z):=\sup_{X\in\Phi}l_r(z,X)$ on
$\mL^{*d}(\frac{r}{\sqrt{d}})$. Now,
note that $C(\Phi,r)$ percolates if
$\{z:U^r_{\Phi}(z) \geq 1\}$ percolates on $\mL^{*d}(\frac{r}{\sqrt{d}})$. We shall
now show that this holds true for $r > \tilde{c}(\lam)$.
Using the ordering of void probabilities we have
\begin{eqnarray}
\pr{U^r_{\Phi}(z_i) = 0 , 1 \leq i \leq n} & = & \pr{\Phi \cap \bigcup_{i=1}^n Q_{\frac{r}{2\sqrt{d}}}(z_i) = \emptyset} \non \\
& \leq & \pr{\Phi_{\lam} \cap \bigcup_{i=1}^n Q_{\frac{r}{2\sqrt{d}}}(z_i) = \emptyset } \non \\
\label{eqn:Peierl_bd_perc} & = & \left(\exp\left\{-\lam
    (\frac{r}{\sqrt{d}})^d \right\}\right)^n\,.
\end{eqnarray}
Clearly, for $r > \tilde{c}(\lam),$ the exponential term above is
less than $(3^d-2)^{-1}$ and thus   $\{z:U^r_{\Phi}(z) \geq 1\}$
percolates by Peierls argument (cf
Remark~\ref{rem:perc-standard-arguments}).
 It is easy to see that for any $s >
0$,  $\exp\{-\lam (\frac{r}{\sqrt{d}})^d\} \leq \exp\{-(1-e^{-s})\lam
(\frac{r}{\sqrt{d}})^d\}$ and hence $\tilde{c}(\lam) \leq c(\lam,1)$.
\end{proof}

\begin{prop}
\label{prop:non_perc_joint_int}
Let $\Phi$ be a stationary, $\alpha$-weakly
sub-Poisson pp of intensity $\lambda$  (i.e., it has all factorial moment measures smaller than those of
$\Phi_\lambda$).
Then $r_c(\Phi)\geq \frac12\frac{1}{(\lambda(3^d-1))^{1/d}} > 0$.
\end{prop}
\begin{proof}

We shall use the same method as in the first part of Theorem
\ref{thm:k_perc_sub-Poisson_pp} but just that we will bound the level
crossing probabilities by using the factorial moment measures.
As in Theorem \ref{thm:k_perc_sub-Poisson_pp}, consider the close
packed lattice $\mL^{*d}(2r)$, the response function $l_r(x,y) := \1[x
\in Q_r(y)]$ and the corresponding shot-noise field $V^r_{\Phi}(z)$ on
$\mL^{*d}(2r)$. We know that $C(\Phi,r)$ percolates only if
$\mZ^d_{2r}(V^r_{\Phi},\geq 1)$ percolates. Let us disprove the latter
for $r<
\frac12\frac{1}{(\lambda(3^d-1))^{1/d}}$.
\begin{eqnarray*}
\pr{V^r_{\Phi}(z_i) \geq 1 , 1 \leq i \leq n} & = & \pr{\Phi(Q_r(z_i))
  \geq 1 ,  1 \leq i \leq n}  \\
& \leq & \EXP{\prod_{i=1}^n \Phi(Q_r(z_i)) }\\
& \leq &\EXP{\prod_{i=1}^n \Phi_{\lam}(Q_r(z_i))} \\
& = & 
(\lam (2r)^d)^n
\end{eqnarray*}
We can see that for $r<
\frac12\frac{1}{(\lambda(3^d-1))^{1/d}}$ the set $\mZ^d_{2r}(V^r_{\Phi},\geq 1)$
does not percolate (cf Remark~\ref{rem:perc-standard-arguments}).
This disproves percolation in $C(\Phi,r)$.
\end{proof}

\begin{cor}
\label{cor:phase-transition}
Combining the results of Propositions~\ref{prop:non_perc_joint_int}
and~\ref{prop:Peierls_arg_bd_pp} we have 
$0<\frac12\frac{1}{(\lambda(3^d-1))^{1/d}}\le r_c(\Phi)\le
\sqrt{d}\left(\frac{\log (3^d-2)}{
\lambda}\right)^{1/d}<\infty$
for all  weakly sub-Poisson pp. Examples of such pp are
determinantal pp  with the trace-class integral kernels 
(see Section~\ref{sec:det.perm}).
\end{cor}

\subsection{Percolation in SINR graphs}
\label{sec:perc_SINR}
Study of percolation in the Boolean model $C(\Phi,r)$ was proposed in~\cite{Gilbert61}
to address the feasibility of multi-hop communications in large
``ad-hock'' networks, where full connectivity is typically hard to maintain. 
The Signal-to-interference-and-noise ratio (SINR)
model~(see~\cite{sinr_coverage,Dousse_etal06}~\footnote{The name 
{\em shot-noise germ-grain process} was also suggested by D.~Stoyan in 
his private communication to BB.})
 is more adequate than the Boolean model
in the context of  wireless communication networks as it  allows
one to take into account the {\em interference} intrinsically related to wireless communications.
For more motivation to study SINR model, refer \cite{subpoisson} and the references therein.

We begin with a formal introduction of the SINR graph model.
In this subsection, we shall work only in $\mR^2$. The parameters of the model are non-negative numbers $P$(signal power), $N$(environmental noise), $\gamma$, $T$(SINR threshold) and an attenuation function $\ell :\mR^2_+ \to \mR_+$  satisfying the following assumptions:
$\ell(x,y)  =  l(|x-y|)$ for some continuous function $l : \mR_+ \to
\mR_+$, strictly decreasing on its support, with  $l(0) \geq TN/P$, $l(\cdot) \leq  1$,
and  $\int_0^{\infty} x l(x) \md x  <  \infty$.
These are exactly the assumptions made in~\cite{Dousse_etal06} and we
refer to this paper for a discussion on their validity.

Given a pp $\Phi$, the {\em interference} generated due to
the pp at a location $x$ is defined as the following
shot-noise field
\label{eqn:interference}
$I_{\Phi}(x) := \sum_{X \in \Phi\setminus\{x\}} l(|X-x|)$.
Define the SINR value as follows :
\begin{equation}
\label{eqn:sinr_defn}
SINR(x,y,\Phi,\gamma) :=  \frac{Pl(|x-y|)}{N + \gamma P I_{\Phi\setminus\{x\}}(y)}.
\end{equation}

Let $\Phi_B$ and $\Phi_I$ be two pp. Let $P,N,T > 0$ and $\gamma \geq 0$. The SINR graph is defined as $G(\Phi_B,\Phi_I,\gamma) := (\Phi_B,E(\Phi_B,\Phi_I,\gamma))$ where $E(\Phi_B,\Phi_I,\gamma) := \{ \la X,Y \ra~\in~\Phi_B^2 : SINR(Y,X,\Phi_I,\gamma) > T, SINR(X,Y,\Phi_I,\gamma) > T \}$. The SNR graph(i.e, the graph without interference, $\gamma = 0$) is defined as $G(\Phi_B) := (\Phi_B,E(\Phi_B))$ where $E(\Phi_B) := \{ \la X,Y \ra~\in~\Phi_B^2 : SINR(X,Y,\Phi_B,0) > T \}$.

Observe that the SNR graph $G(\Phi)$ is same as the graph $C(\Phi,r_l)$ with $2r_l = l^{-1}(\frac{TN}{P})$. Also, when $\Phi_I = \emptyset$, we shall omit it from the parameters of the SINR graph. Recall that percolation in the above graphs is existence of an infinite connected component in the graph-theoretic sense.

\subsubsection{Poissonian back-bone nodes}
\label{sec:poisson_sinr}

Firstly, we consider the case when the backbone nodes ($\Phi_B$) form
a Poisson pp and in the next section, we shall relax this
assumption. When $\Phi_B = \Phi_{\lam}$, the Poisson pp of intensity
$\lambda$, we shall use $G(\lam,\Phi_I,\gamma)$ and $G(\lam)$ to
denote the SINR and SNR graphs respectively. Denote by
$\lam_c(r):=\lambda (r_c(\Phi_\lambda)/r)^2$
the {\em critical intensity} for percolation of the Boolean model $C(\Phi_\lam,r)$. The
following result guarantees the existence of a $\gamma > 0$ such that
for any sub-Poisson pp $\Phi = \Phi_I$, $G(\lam,\Phi,\gamma)$ will
percolate provided $G(\lam)$ percolates i.e, the SINR graph percolates
for small interference values when the corresponding SNR graph
percolates.
\begin{prop}
\label{thm:sinr_poisson_perc}
Let $\lam > \lam_c(r_l)$ and $\Phi \leq_{idcx} \Phi_{\mu}$ for some
$\mu > 0$. Then there exists a $\gamma > 0$ such that $G(\lam,\Phi,\gamma)$ percolates.
\end{prop}
Note that we have not assumed the independence of $\Phi$ and
$\Phi_{\lam}$. In particular, $\Phi$ could be $\Phi_{\lam} \cup
\Phi_0$ where $\Phi_0$ is an independent sub-Poisson pp. The case
$\Phi_0 = \emptyset$ was proved in \cite{Dousse_etal06}. Our proof
follows their idea of coupling the continuum model with a discrete
model. As in
\cite{Dousse_etal06}, it is clear that for  $N \equiv 0$, the
above result holds with $\lam_c(r_l) = 0$.
\begin{proof}[Sketch of the proof of Proposition~\ref{thm:sinr_poisson_perc}]
Our proof follows the arguments  given in \cite{Dousse_etal06} and here,
we will only give a sketch of the proof. The details can be found in
in~\cite[Section~6.3.4]{Yogesh_thesis}.

Assuming $\lambda>\lambda_c(\rho_l)$, one observes first
that the graph $G(\lambda)$ also percolates
with any slightly larger constant noise
$N'=N+\delta'$, for some $\delta'>0$.
Essential to the proof of the result
is to show that the
level-set $\{x:I_{\Phi_I}(x)\le M\}$ of the interference
field  percolates (contains an infinite connected component) for sufficiently
large $M$. Suppose that it is true.
Then taking $\gamma=\delta'/M$ one has
percolation of the level-set $\{y:\gamma I_{\Phi_I}(y)\le \delta'\}$.
The main difficulty consists in showing that
$G(\lambda)$ with noise $N'=N+\delta'$
percolates {\em within} an infinite connected component of
$\{y:I_{\Phi_I}(y)\le \delta'\}$. This was done in~\cite{Dousse_etal06},
by mapping both models $G(\lambda)$ and the level-set of the
interference field to a discrete lattice
and showing that both discrete approximations not only percolate
but actually satisfy a stronger condition, related to the Peierls argument.
We follow exactly the same steps and the only fact that we have to
prove,  regarding  the interference, is that there exists a constant
$\epsilon<1$ such that for arbitrary $n\ge1$ and
arbitrary choice of locations $x_1,\ldots,x_n$ one has
$\pr{I_{\Phi_I}(x_i)> M, \, i=1,\ldots, n}\le \epsilon^n$.
In this regard, we use the first statement of Lemma ~\ref{lem:gen_bds}
to prove, exactly as in~\cite{Dousse_etal06}, that for sufficiently small $s$
it is not larger than $K^n$ for some
constant $K$  which depends on $\lam$ but not on $M$.
This completes the proof.
\end{proof}

\subsubsection{Non-Poissonian back-bone nodes}
\label{sec:non_poisson_sinr}

We shall now consider the case when the backbone nodes are formed by a
sub-Poisson pp. In this case, we can give a weaker result, namely that
with an increased signal power (i.e, possibly much greater than the critical
power), the SINR graph will percolate for small interference
parameter $\gamma > 0$.
\begin{prop}
\label{thm:perc_sinr_sub-Poisson}
Let $\Phi$ be a stationary, $\nu$-weakly sub-Poisson pp and $\Phi_I
\leq_{idcx} \Phi_{\mu}$ for some $\mu > 0$ and also assume that $l(x)
> 0$ for all $x \in \mR_+$. Then there exist $P, \gamma > 0$ such
that $G(\Phi,\Phi_I,\gamma)$ percolates.
\end{prop}
As in Theorem \ref{thm:sinr_poisson_perc}, we have not assumed the independence of $\Phi_I$ and $\Phi$. For example, $\Phi_I = \Phi \cup \Phi_0$ where $\Phi$ and $\Phi_0$ are independent sub-Poisson pp.  Let us also justify the assumption of unbounded support for $l(.)$. Suppose that $r = \sup \{x : l(x) > 0 \} < \infty$. Then if $C(\Phi,r)$ is sub-critical, $G(\Phi,\Phi_I,\gamma)$ will be sub-critical for any $\Phi_I,P,\gamma$.
\begin{proof}[Sketch of the proof of Proposition~\ref{thm:perc_sinr_sub-Poisson}]
In this scenario, increased power is equivalent to increased radius in
the Boolean model corresponding to SNR model. From this
observation, it follows from Proposition~\ref{prop:Peierls_arg_bd_pp} that with
possibly increased power the associated SNR model percolates. Then, we
use the approach from the proof of
Proposition~\ref{thm:sinr_poisson_perc} to obtain
a $\gamma > 0$ such that the SINR network percolates as well.
The details can be found in \cite[Section~6.3.4]{Yogesh_thesis}.
\end{proof}
For further  discussion on $dcx$ ordering in the context of
communication networks see~\cite{subpoisson}.

\section{Examples of ordered point processes}
\label{sec:ex}
\hfill\newline
In this section, we will give examples of pp for which
some formal comparison of clustering is possible.
In particular, in Section~\ref{sec:pl}, 
we will give examples of families of perturbed lattice
pp which are proved to be monotone with respect to the $dcx$ order.
In Section~\ref{ss.NumerialPercol}, we will show numerical evidences
supporting the conjecture that  
within this class of pp the critical radius $r_c$ is monotone with
respect to the $dcx$ order on the underlying pp.
It is known that these pp can be considered as toy  models for
determinantal and permanental pp,  which we show in
Section~\ref{sec:det.perm} to be, respectively,
weakly sub- and super-Poisson. 
Poisson-Poisson cluster pp are known to be $dcx$
larger than Poisson pp. An example of such a pp - with
$r_c = 0$ - will  be given in Section~\ref{ss.super-percol}.
This invalidates the conjecture on the monotonicity of 
$r_c$ with respect to the $dcx$ order of pp, in full generality.

\subsection{Perturbed point processes}
\label{sec:pl}
Let $\Phi$ be a pp on $\mR^d$ and $\cN(\cdot,\cdot)$,
$\cX(\cdot,\cdot)$ be two probability kernels from $\mR^d$ to
non-negative integers $\mZ^+$ and $\mR^d$, respectively.
Consider the following {\em independently marked}
version of the pp $\Phi$, $\tilde\Phi^{pert}=\{(X,N_X,{\mathbf Y}_X) \}_{X\in\Phi}$
where given $\Phi$:
\begin{itemize}
\item $N_X$,  $X\in\Phi$ are independent, non-negative
integer-valued random variables with distribution
$\pr{N_X\in\cdot\,|\,\Phi}=\cN(X,\cdot)$,
\item  ${\mathbf Y}_{X}=(Y_{iX} : i =1,2,\ldots)$, $X\in\Phi$ are
  independent vectors of  i.i.d.  elements of  $\mR^d$, with
  $Y_{iX}$'s having the conditional distribution $\pr{Y_{iX}\in\cdot\,|\,\Phi}=\cX(X,\cdot)$,
\item the random elements
$N_X,{\mathbf Y_X}$ are independent for all $X\in\Phi$.
\end{itemize}
Consider the following subset of $\mR^d$
\begin{equation}
\label{defn:perturbed_pp}
\Phi^{pert}= \bigcup_{X \in \Phi} \bigcup_{i=1}^{N_X} \{ X + Y_{iX}\}\,,
\end{equation}
where the inner sum is interpreted as $\emptyset$ when $N_X=0$.
The set $\Phi^{pert}$ can  (and will) be considered as a pp
on $\mR^d$ provided it is locally finite. In what follows, in
accordance with our general assumption for this article, we will assume that the mean measure of $\Phi^{pert}$ is locally finite (Radon measure)
\begin{equation}
\label{e:perturbation-Radon}
\int_{\mR^d} n(x)\cX(x,{A-x})\,  \al(dx)< \infty, \quad
\text{for all bBs~$B\subset\mR^d$},
\end{equation}
where  $\al(\cdot) = \EXP{\Phi(\cdot)}$ stands
for the mean measure of the pp $\Phi$ and $n(x)=\sum_{k=1}^\infty
k\cN(x,\{k\})$ is the mean value of the distribution $\cN(x,\cdot)$.

The pp $\Phi^{pert}$ can be seen as
{\em independently replicating and
translating points from the pp $\Phi$}, with the  number of
replications of the point  $X\in\Phi$ having distribution
$\cN(X,\cdot)$ and the independent  translations of these replicas
from $X$  by vectors having distribution $\cX(X,\cdot)$. For this
reason, we call $\Phi^{pert}$ a {\em perturbation} of~$\Phi$
driven by  the {\em replication kernel} $\cN$ and the {\em translation kernel}
$\cX$.

An important observation for us is that the {\em operation of perturbation of $\Phi$ is $dcx$ monotone with respect to the replication kernel} in the following sense.

\begin{prop}\label{p.pert-lattice}
Consider a pp $\Phi$ with Radon mean measure $\al(\cdot)$
and its two perturbations
$\Phi_j^{pert}$ $j=1,2$ satisfying
condition~(\ref{e:perturbation-Radon}), having the same
translation kernel $\cX$ and possibly different replication kernels
$\cN_j$, $j=1,2$, respectively.
If $\cN(x,\cdot)\leq_{cx}\cN(x,\cdot)$ (convex
ordering of the conditional distributions of the number of replicas)
for $\al$-almost all $x\in\mR^d,$ then
$\Phi^{pert}_1\leq _{dcx} \Phi^{pert}_2$.
\end{prop}
\begin{proof}
We will consider some particular coupling of the two perturbations
$\Phi^{pert}_j$,  $j=1,2$.
Given  $\Phi$ and ${\mathbf Y}_X=(Y_{iX}:i=1,\ldots)$ for
each  $X\in\Phi$, let $\Phi_{jX} = \bigcup_{i=1}^{N^j_{X}} \{ X + Y_{iX} \}$,
where $N^j_X$ has distribution $\cN_j(X,\cdot)$, $j=1,2$,
respectively.
Thus  $\Phi^{pert}_j=\sum_{X\in\Phi}\Phi_{jX},$ $j=1,2$
are the two considered perturbations.
Note that given $\Phi$, $\Phi^{pert}_j$ can be seen as
independent superpositions of $\Phi_{jX}$ for $X \in \Phi.$ Hence,
by \cite[Proposition 3.2(4)]{snorder} (superposition preserves $dcx$
order) and \cite[Theorem 3.12.8]{Muller02} (weak and $L_1$ convergence
jointly preserve $dcx$ order),
it is enough to show that conditioned on $\Phi$, $\Phi_{1X} \leq_{dcx}
\Phi_{2X}$ for every $X \in\Phi$.
In this regard, given $\Phi$, consider $X\in\Phi$ and let
$B_1,\ldots,B_k$ be mutually disjoint bounded Borel subsets and $f : \mR^k
\to \mR$, a $dcx$ function. Define a real valued function
$g:\mZ \to \mR$, as
$$g(n):= \EXP{f\Bigl(\text{sgn}(n)\sum_{i=1}^{|n|}(\1[Y_{iX} \in B_1 - X],\ldots,
\1[Y_{iX} \in B_k - X])\Bigr) \bigg| \Phi}\,,$$
where $\text{sgn}(n)= \frac{n}{|n|}$ for $n \neq 0$ and $\text{sgn}(0) = 0.$ By Lemma~\ref{lem:MeesterSh-like}, $g(\cdot)$ is a convex function
on~$\mZ$ and by Lemma~\ref{lem:discrete-convex}
it can be extended to a convex function $\tilde g(\cdot)$ on~$\mR$.
Moreover,
$\EXP{\tilde g(N^j_{X}) | \Phi} =\EXP{g(N^j_{X}) | \Phi}\allowbreak =
\sE\Bigl(f(\Phi_{jX}(B_1),\ldots,\Phi_{jX}(B_k)) | \Phi\Bigr)$
for $j=1,2$.
Thus, the result follows from the assumption $N_{X}^1\le_{cx}N_{X}^2$.
\end{proof}

\begin{rem}
The above proof remains valid for an extension
of the perturbation model in which
the distribution $\cX(X,\cdot)$ of the translations $Y_{iX}$
depends not only on the location of the point $X\in\Phi$ but also on the
entire configuration $\Phi$; $\cX(X,\cdot)=\cX(X,\Phi, \cdot)$,
provided condition~(\ref{e:perturbation-Radon}) is replaced by finiteness of
$\int_{\mM^d}\int_{\mR^d}n(x)\cX(x,\phi,A-x)\,C(d(x,\phi))$,
where $C(d(x,\phi))$ is the Campbell measure of $\Phi$.
\end{rem}

Perturbed pp, by Proposition~\ref{p.pert-lattice}, provide
many examples of pp comparable in $dcx$ order.
We will be particularly interested in the following two special cases for the
choice of the initial pp $\Phi$.

\begin{exe}\label{ex.perturbation_of_Poisson}
{\em Perturbation of a Poisson pp.\;}
 Let $\Phi$ be a (possibly
  inhomogeneous) Poisson pp of mean measure $\alpha(dx)$ on $\mR^d$.
Let
$\cN(x,\cdot)=\varepsilon_1={\1(1\in\cdot)}$ be the Dirac measure
on $\mZ^+$ concentrated at~1 for all $x\in\mR^d$ and assume
 an  arbitrary translation kernel $\cX$ satisfying
$\alpha^{pert}(A)=\int_{\mR^d}\cX(x,\allowbreak
A-x)\,\alpha(dx)<\infty$ for all
bBs $A$. Then  by the displacement theorem for Poisson pp,
$\Phi^{pert}$ is also a Poisson pp with mean measure $\alpha^{pert}(dx)$.
Assume {\em any} replication kernel $\cN_2(x,\cdot)$,
with mean number of replications $n_2(x)=\allowbreak\sum_{k=1}^\infty
k\cN_2(x,\{k\})=1$ for all $x\in\mR^d$. Then, by the Jensen's inequality and
Proposition~\ref{p.pert-lattice}, one obtains a super-Poisson pp
$\Phi_2^{pert}$. In the special case, when $\cN_2(x,\cdot)$ is
the Poisson distribution with  mean~1 for all $x\in\mR^d$,
$\Phi_2^{pert}$ is a Poisson-Poisson cluster pp which is
a special case of a Cox (doubly stochastic Poisson) pp with
(random) intensity measure  $\Lambda(A) = \sum_{X \in
  \Phi}\cX(x,A-x)$. The fact that it is super-Poisson was already observed
in~\cite{snorder}. Note that for a general distribution of  $\Phi$,
its perturbation $\Phi_2^{pert}$ is also a Cox pp of the intensity
$\Lambda$ given above.  Other Poisson-Poisson cluster processes, 
with $\cN_2(x,\cdot)$ not necessarily of mean~1, will be considered in
Section~\ref{ss.super-percol}. 
\end{exe}

\begin{exe} \label{ex:perturbed.lattice}
{\em Perturbation of a deterministic lattice.\;}
Assuming a deterministic lattice $\Phi$ (e.g. $\Phi=\mZ^d$)
gives rise to the perturbed lattice pp of the type considered
in~\cite{Sodin04}. Surprisingly enough, starting
from such a $\Phi$, one can also construct a Poisson pp and
both super- and sub-Poisson perturbed pp. In this regard, assume for simplicity that $\Phi=\mZ^d$, and the translation kernel
$\cX(x,\cdot)$ is uniform on the unit cube $[0,1)^d$.
Let $\cN(x,\cdot)$ be the Poisson distribution with mean~$\lambda$ ($Poi(\lambda)$). It is easy to see that such a perturbation $\Phi^{pert}$ of the lattice $\mZ^d$ gives rise to a homogeneous Poisson pp with intensity~$\lambda$.
\begin{description}
\item[{\rm \em Sub-Poisson perturbed lattices.}]
Assuming for  $\cN_1$ some distribution convexly ($cx$)  smaller than $Poi(\lambda),$
one obtains a sub-Poisson perturbed lattice pp.
Examples are {\em hyper-geometric} $H\!Geo(n,m,k)$,
$m,k\le n$, $km/n=\lambda$
and {\em binomial}  $Bin(n,\lambda/n)$, $\lambda\le n$ distributions%
~\footnote{$Bin(n,p)$ has probability mass function
  $p_{Bin(n,p)}(i)={n\choose i}p^i(1-p)^{n-i}$ ($i=0,\ldots,n$).
$H\!Geo(n,m,k)$ has probability
  mass function $p_{H\!Geo(n,m,k)}(i)={m\choose i}{n-m\choose
    k-i}/{n\choose k}$ ($\max(k-n+m,0)\le i\le m$).},
which can be ordered as follows:
$$H\!Geo(n,m,\lambda n/m)\le_{cx} Bin(m,\lambda/m)\le_{cx}
Bin(r,\lambda/r)\le_{cx} Poi(\lambda)$$
for $\lambda\le m\le \min(r,n)$;
cf.~\cite{Whitt1985}%
\footnote{One shows the
logarithmic concavity of the ratio of the respective probability mass
functions, which  implies increasing convex order and, consequently,
$cx$ provided the distributions have the same  means.}.
Specifically, taking  $\cN_1(x,\cdot)$ to be Binomial
$Bin(n,\lambda/n)$ for
$n\ge\lambda$, one obtains
a $dcx$ monotone increasing family of sub-Poisson pp.
Taking $\lambda=n=1$ (equivalent to $\cN(x,\cdot)=\varepsilon_1$),
one obtains a {\em simple perturbed lattice} that is $dcx$ smaller
than the Poisson pp of intensity~1.
A sample realization of this latter process (with $\Phi$ being the unit
hexagonal lattice on the plane rather than the square lattice)
is shown in Figure~\ref{f.Lattice1}.

\item[{\rm \em Super-Poisson perturbed lattices.}]
Assuming for  $\cN_2$ some distribution convexly larger than $Poi(\lambda),$
one obtains a super-Poisson perturbed lattice. Examples are
{\em negative binomial} ${N\!Bin}(r,p)$ distribution
with $rp/(1-p)=\lambda$, {\em geometric} $Geo(p)$ distribution with
$1/p-1=\lambda$;~%
\footnote{$p_{Geo(p)}(i)=p^i(1-p)^{n-i}$,  $p_{N\!Bin(r,p)}(i)=
{r+i-1\choose i}p^i(1-p)^r$.}
with
\begin{eqnarray*}
Poi(\lambda)&\le_{cx}& N\!Bin(r_2,\lambda/(r_2+\lambda))
\le_{cx}N\!Bin(r_1,\lambda/(r_1+\lambda))\\
&\le_{cx}& Geo(1/(1+\lambda))
\le_{cx}\sum_{j}\lambda_j\;Geo(p_j)
\end{eqnarray*}
with $r_1\le r_2$,
$0\le \lambda_j\le 1$, $\sum_j\lambda_j=1$ and
$\sum_j\lambda_j/p_j=\lambda+1$,
where the largest distribution above is a mixture
of geometric distributions having mean~$\lambda$;
cf.~\cite{Whitt1985}.
Specifically, taking  $\cN_2(x,\cdot)$ to be negative binomial
$N\!Bin(n,\lambda/(n+\lambda))$ for
$n=1,\ldots$ one obtains
a $dcx$ monotone decreasing  family of super-Poisson pp.
Recall that $N\!Bin(r,p)$
is  a mixture of $Poi(x)$  with parameter $x$ distributed
as a gamma distribution with scale parameter $p/(1-p)$
and shape parameter $r$.
\end{description}
\end{exe}

\begin{rem}
From~\cite[Lemma~2.18]{MeesterSh93}, we know that any {\em  mixture of Poisson distributions} having mean $\lambda$ is $cx$ larger than $Poi(\lambda)$.
Thus, the super-Poisson perturbed lattice with such a replication kernel
(translation kernel being the uniform distribution on the unit cube) again gives rise to a Cox pp. A special case of such a Cox pp
with $\cN_2$ being a mixture of two Poisson distributions
was considered in~\cite{snorder} (and called Ising-Poisson cluster pp). However, the proof of the fact that it is super-Poisson was
based on the observation that the (random) density of this Cox pp is
a conditionally increasing field. This argument can be
extended to the case when the replication  marks $N_X$, $X\in\Phi$
constitute a field that is {\em 1-monotonic} (~\cite[Ch.~2]{Grimmett2006}).
Due to space constraints, we refer to~\cite[Ch.~5]{Yogesh_thesis} for the details. A sample realization of the Cox pp obtained by the analogous perturbation of the hexagonal lattice on the plane with $\cN_2$ being Bernoulli $Bin(1,1/5)$ mixture of $Poi(5i)$ for  $i=0,1$ is shown on  Figure~\ref{f.Lattice1}.
\end{rem}

Our interest in  sub-Poisson perturbed lattices stems from their relations to zeros of Gaussian analytic functions (GAFs) (see~\cite{Peres05,Sodin04}). More precisely, \cite{Sodin06} shows that zeros of GAFs have the same distribution as the pp $\bigcup_{z \in \mZ^d}\{z+X_z\}$ for a $\mZ^2$-shift invariant sequence $\{X_z\}_{z \in \mZ^2}$. Simulations and second-moment properties (\cite{Peres05}) indicate that the zero set of GAFs exhibit less clustering (more spatial homogeneity) than a Poisson pp. The above example when seen in the light of the above-mentioned
papers, asks the question that  we are currently not able to answer, whether zeros of GAF are sub-Poisson.

We have seen above two examples of Cox pp (being some perturbations of a Poisson pp or a lattice) that are super-Poisson. A general class of Cox pp called the L\'{e}vy-based Cox pp were proved to be super-Poisson in~\cite{snorder}.

\subsubsection{Numerical comparison of percolation for
perturbed lattices}
\label{ss.NumerialPercol}
\begin{figure}[!t]
\begin{minipage}[b]{0.8\linewidth}
\vspace{-10ex}
\includegraphics[width=1.\linewidth]{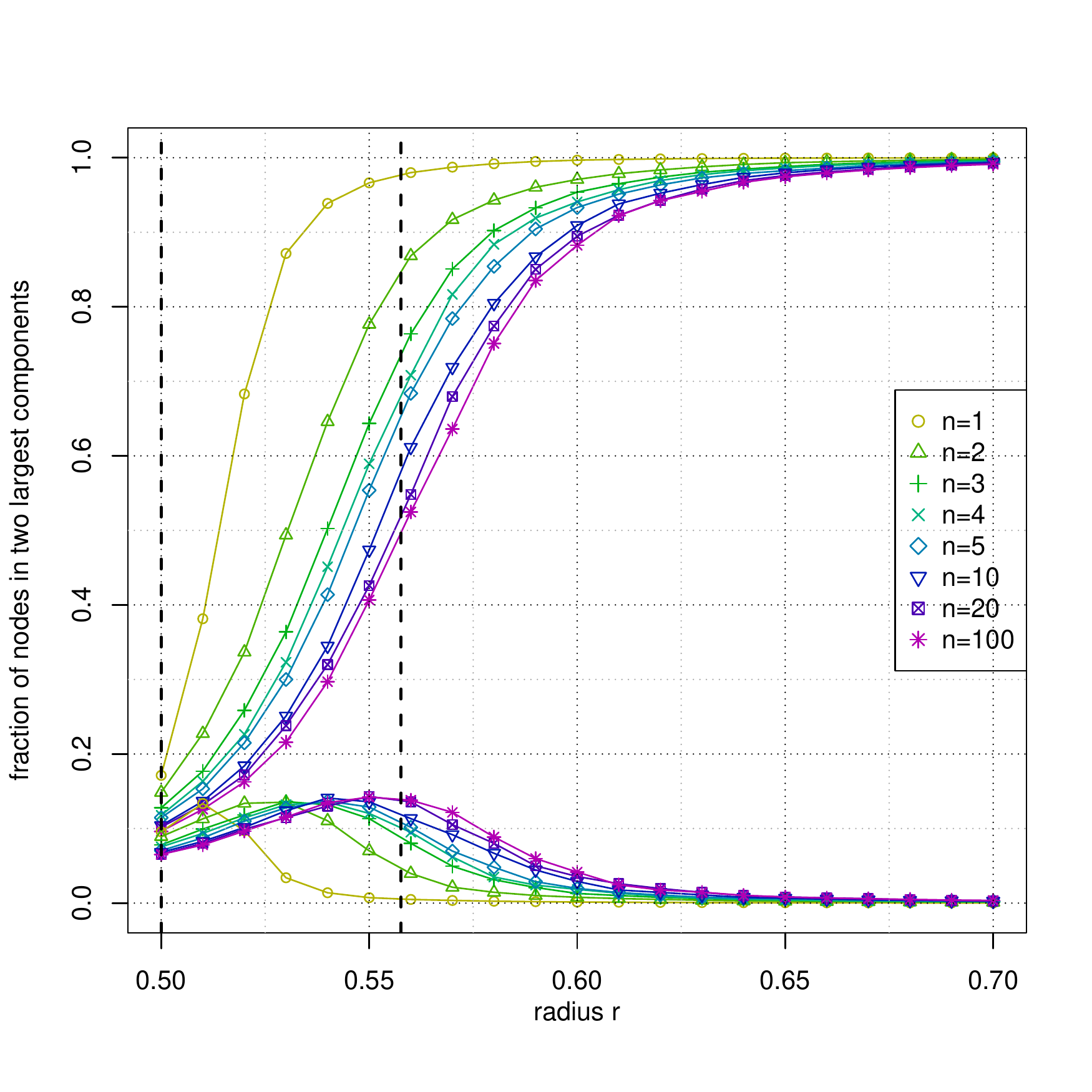}
\end{minipage}
\vspace{-6ex}
\caption{\label{f.TwoComponents.subP}
Mean fractions of nodes in the two largest components
of the sub-Poisson Boolean model  $C(\Phi^{pert}_{Bin}(n),r)$ as a function of $r$;
see Section~\ref{ss.NumerialPercol}. The underlying pp
$\Phi^{pert}_{Bin}(n)$ are $dcx$
increasing in $n$ to the Poisson pp $\Phi_\lambda$
of intensity $\lambda=2/(\sqrt3)=1.154701$.
The right dashed vertical line corresponds
to the radius $r=0.5576495$ which is believed to be close 
to the critical radius $r_c(\Phi_\lambda)$
for the Poisson pp.
The left  dashed vertical line correspond to the
critical radius $r_c(\Phi)=0.5$ for the non-perturbed lattice.}
\end{figure}

\begin{figure}[!t]
\begin{center}
\hbox{
\hbox{}\hspace{0.01\linewidth}
\begin{minipage}[b]{0.22\linewidth}
\includegraphics[width=1\linewidth]{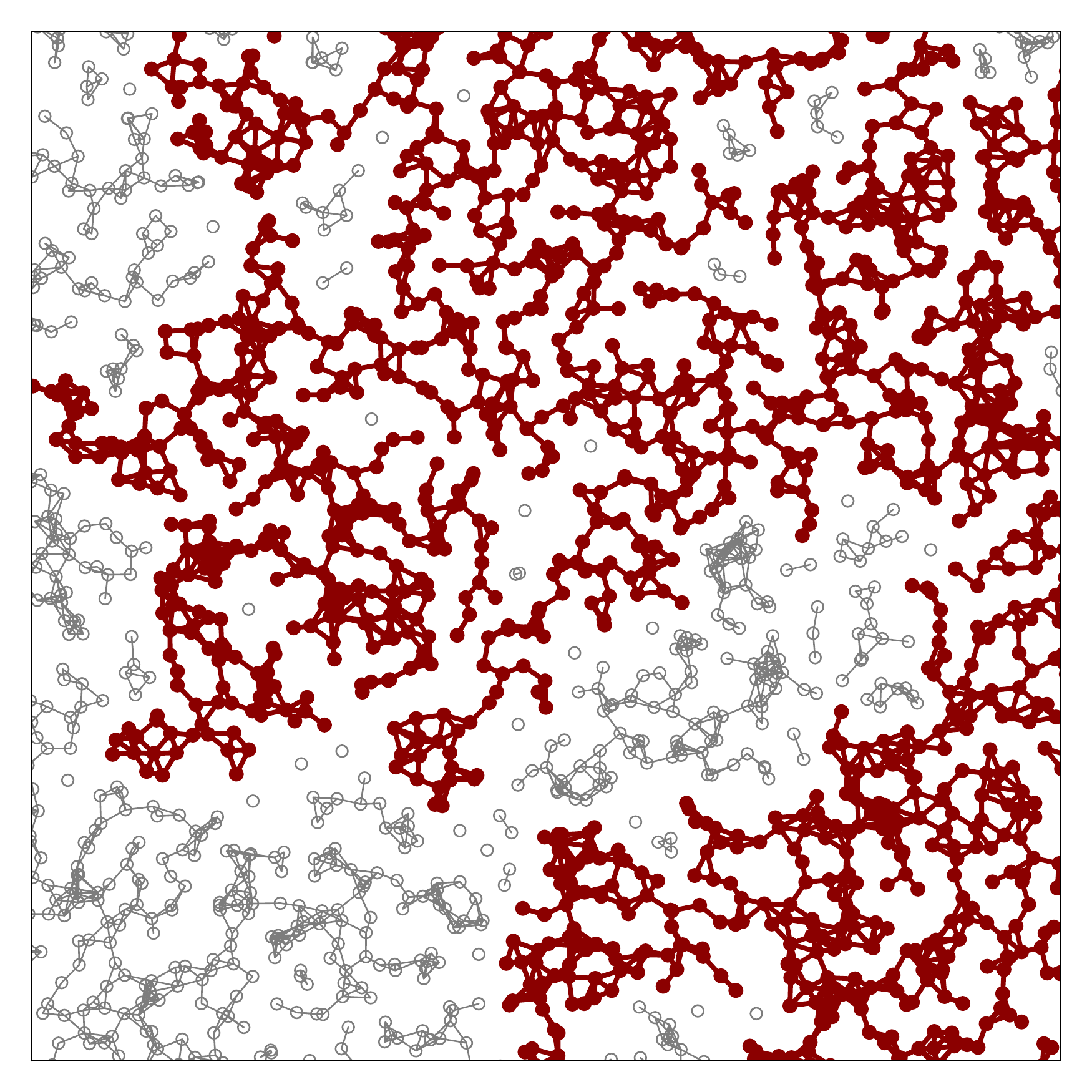}\\[-0.6\linewidth]
\hbox{}\hspace{-0.16\linewidth}
\includegraphics[width=0.70\linewidth]{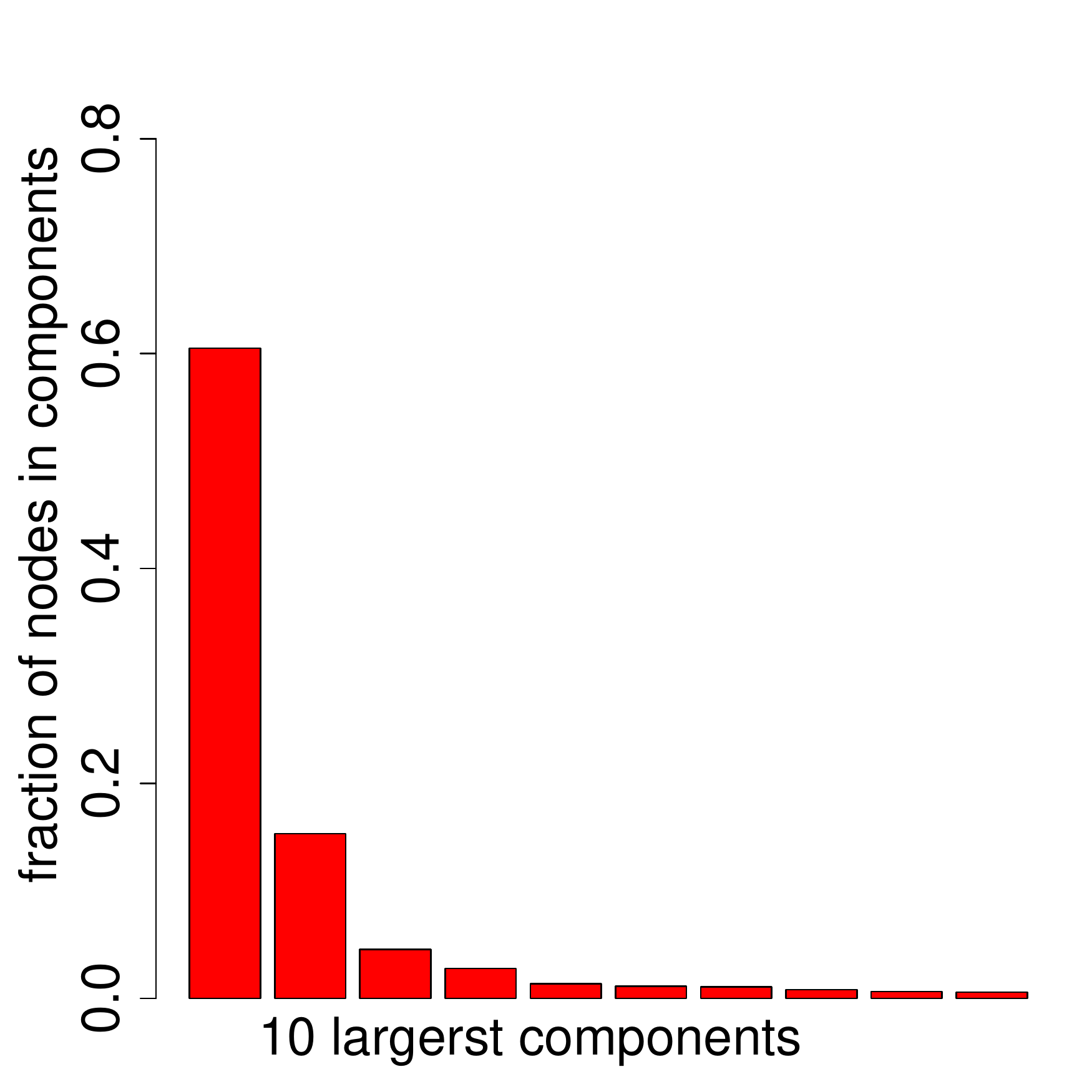}\\
\centerline{$n=5$}
\end{minipage}
\hspace{0.02\linewidth}
\begin{minipage}[b]{0.22\linewidth}
\includegraphics[width=1\linewidth]{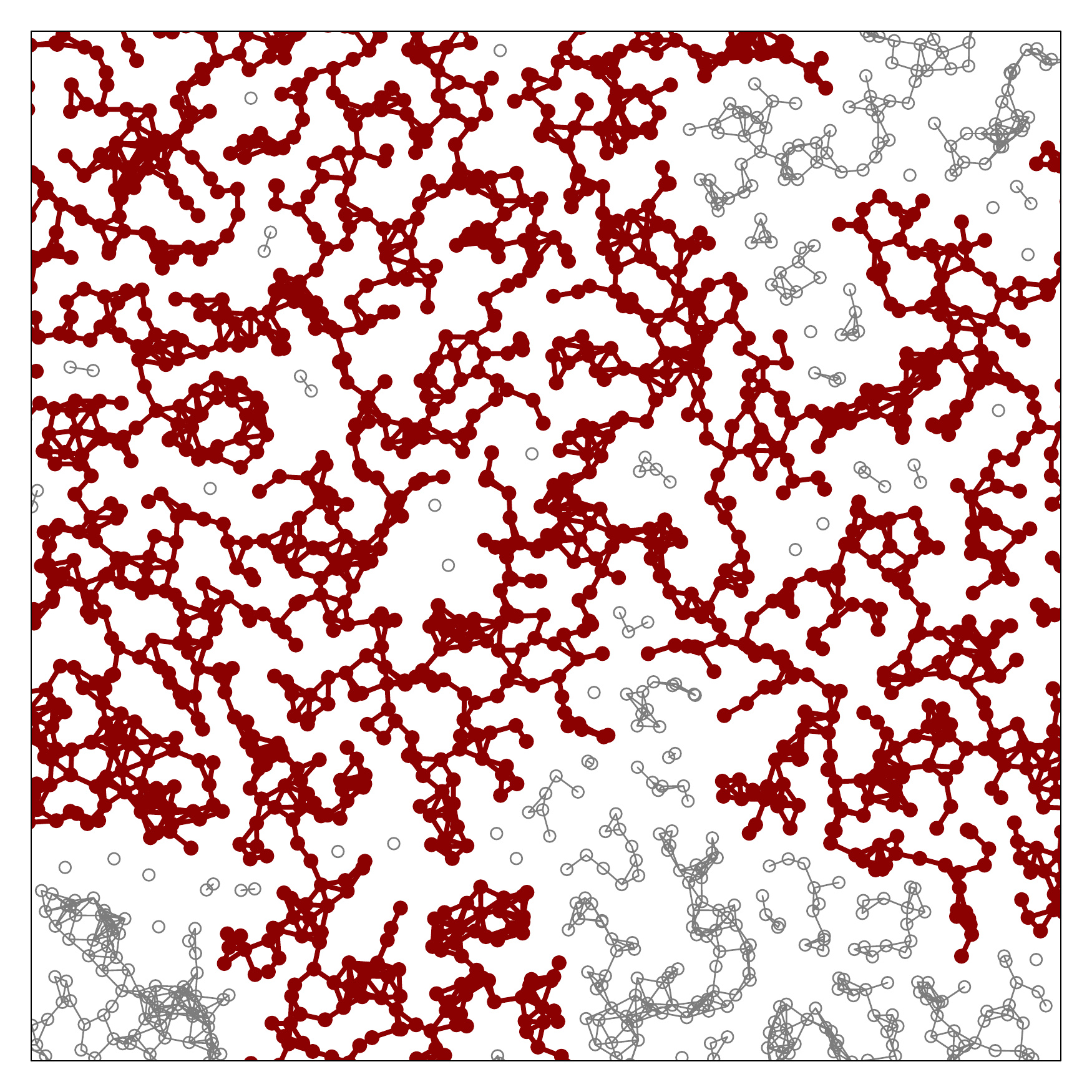}\\[-0.6\linewidth]
\hbox{}\hspace{-0.16\linewidth}
\includegraphics[width=0.70\linewidth]{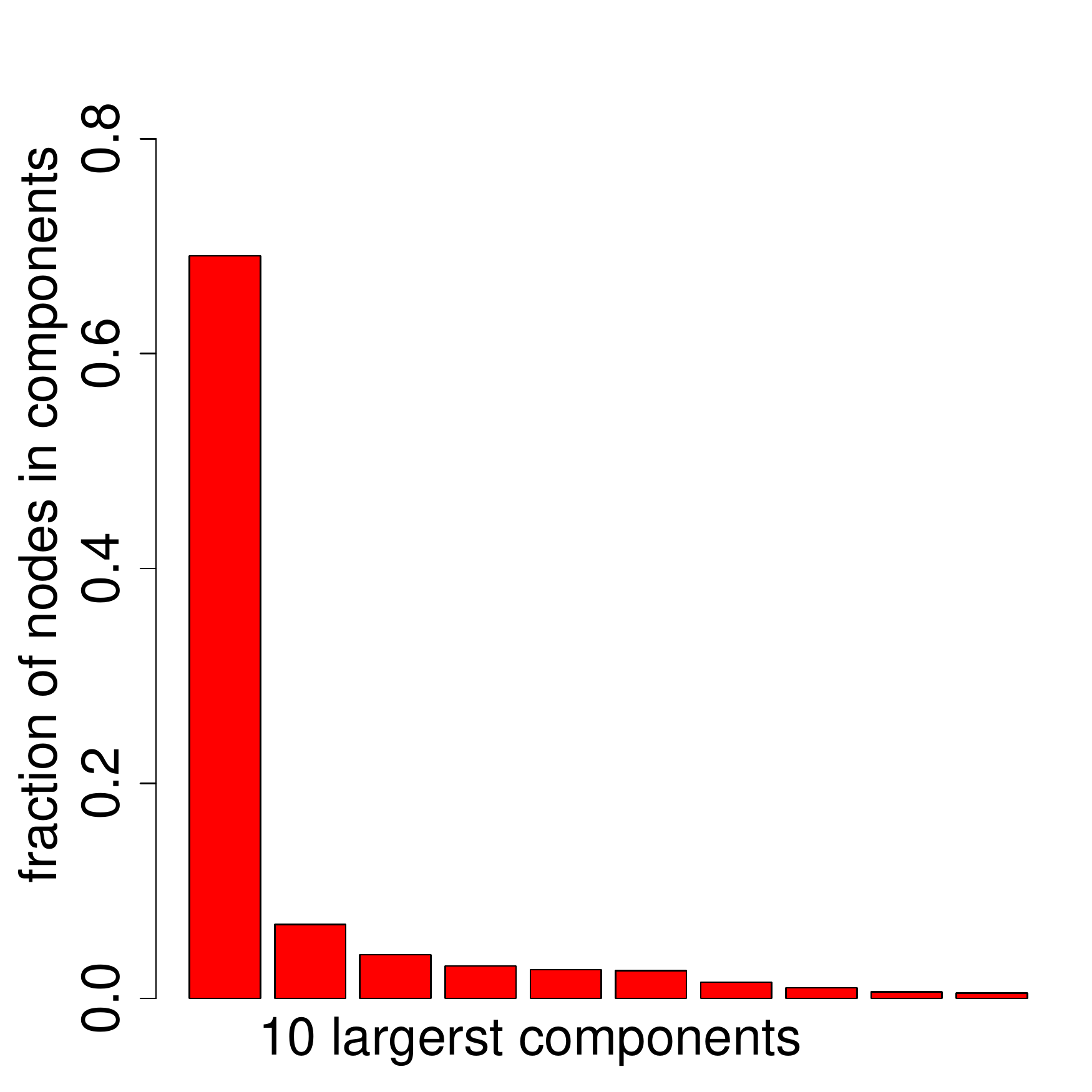}\\
\centerline{$n=6$}
\end{minipage}
\hspace{0.02\linewidth}
\begin{minipage}[b]{0.22\linewidth}
\includegraphics[width=1\linewidth]{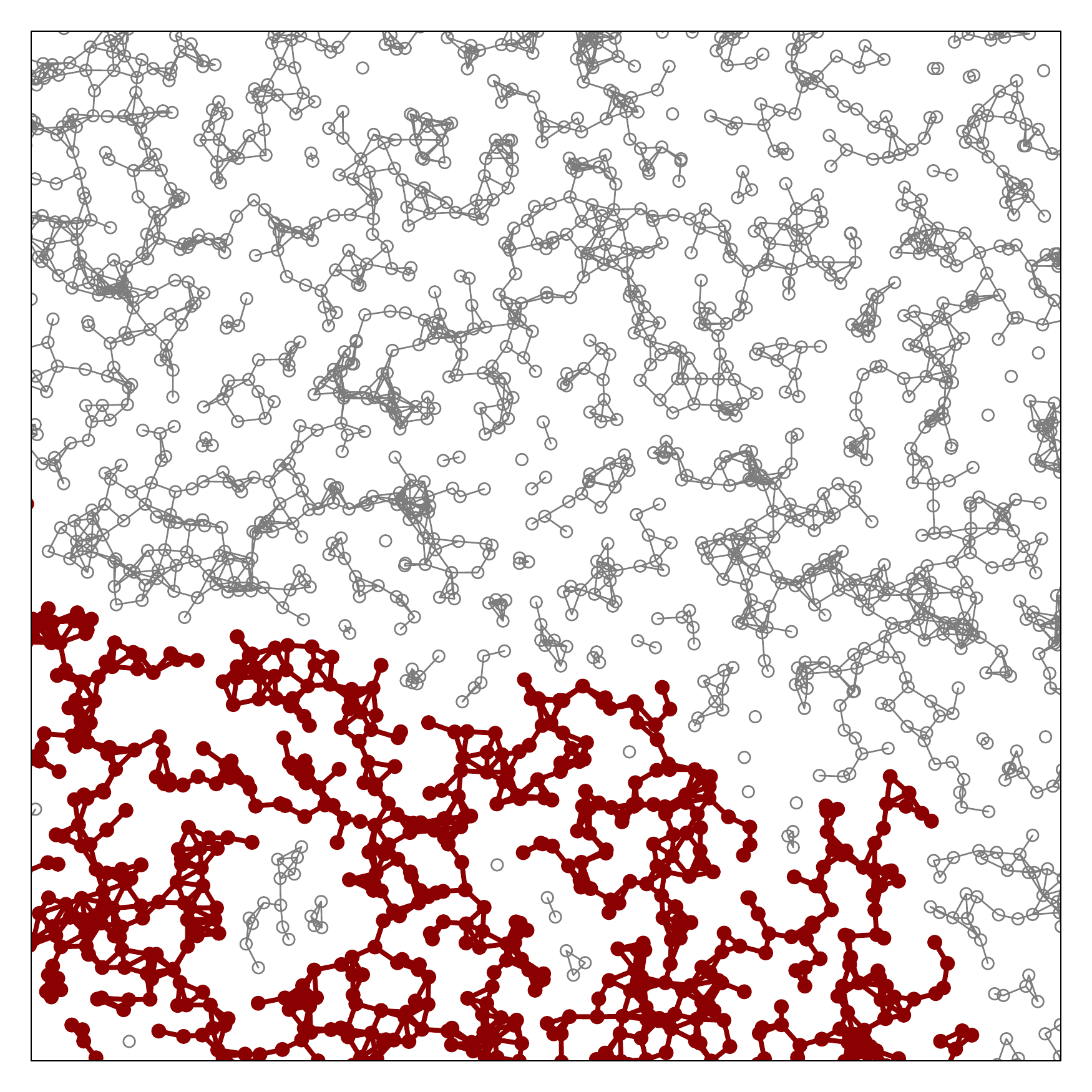}\\[-0.6\linewidth]
\hbox{}\hspace{-0.16\linewidth}
\includegraphics[width=0.70\linewidth]{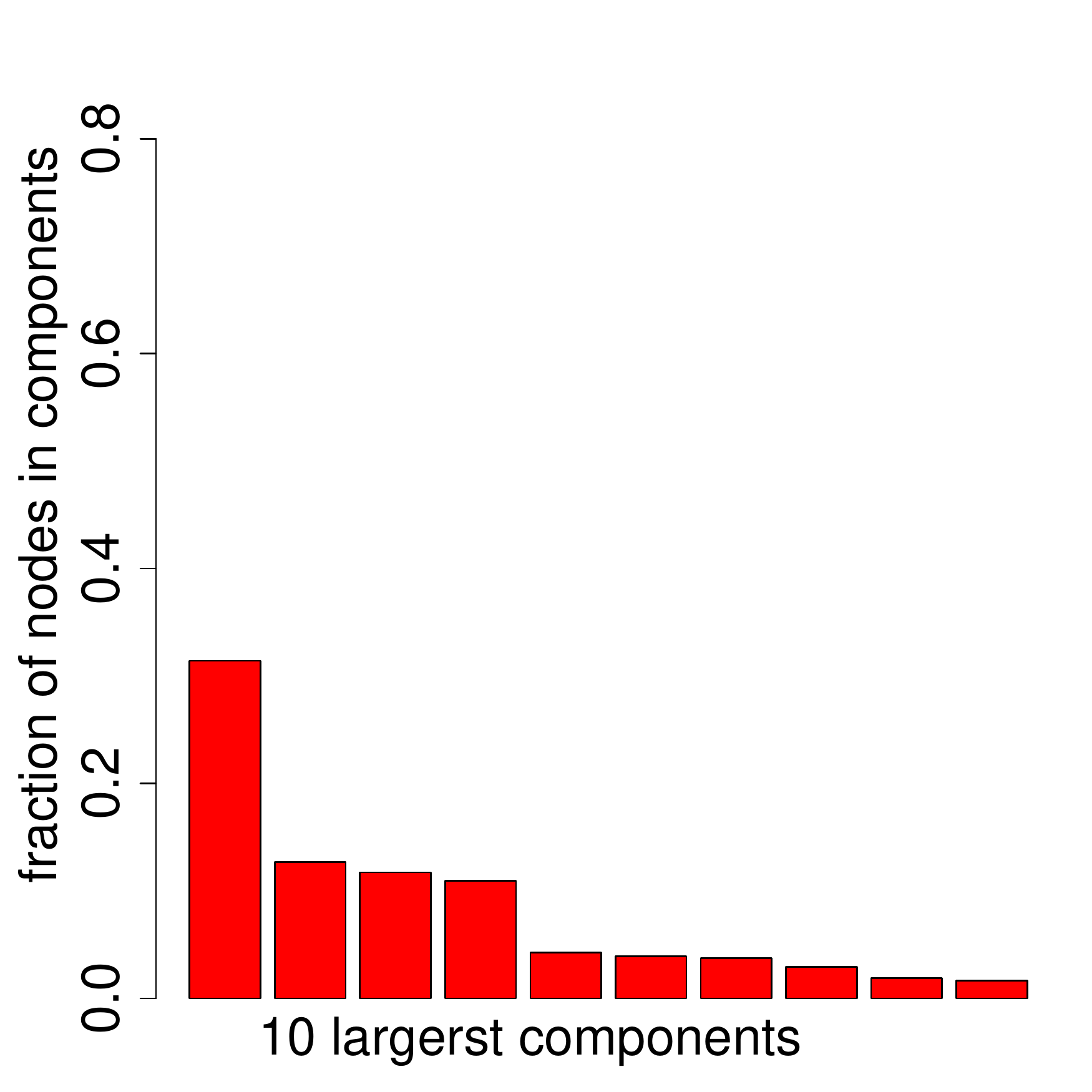}\\
\centerline{$n=7$}
\end{minipage}
\hspace{0.02\linewidth}
\begin{minipage}[b]{0.22\linewidth}
\includegraphics[width=1\linewidth]{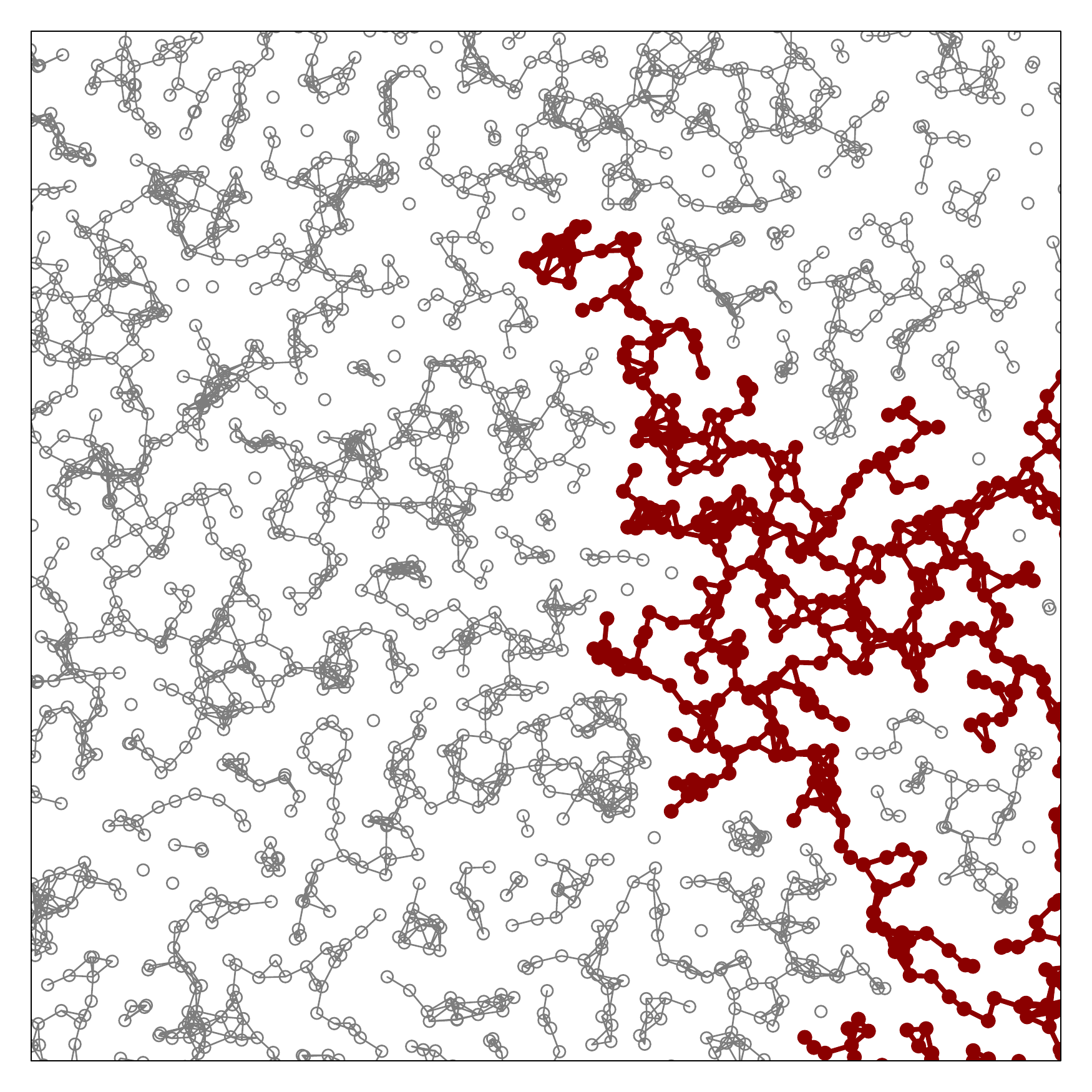}\\[-0.6\linewidth]
\hbox{}\hspace{-0.16\linewidth}
\includegraphics[width=0.70\linewidth]{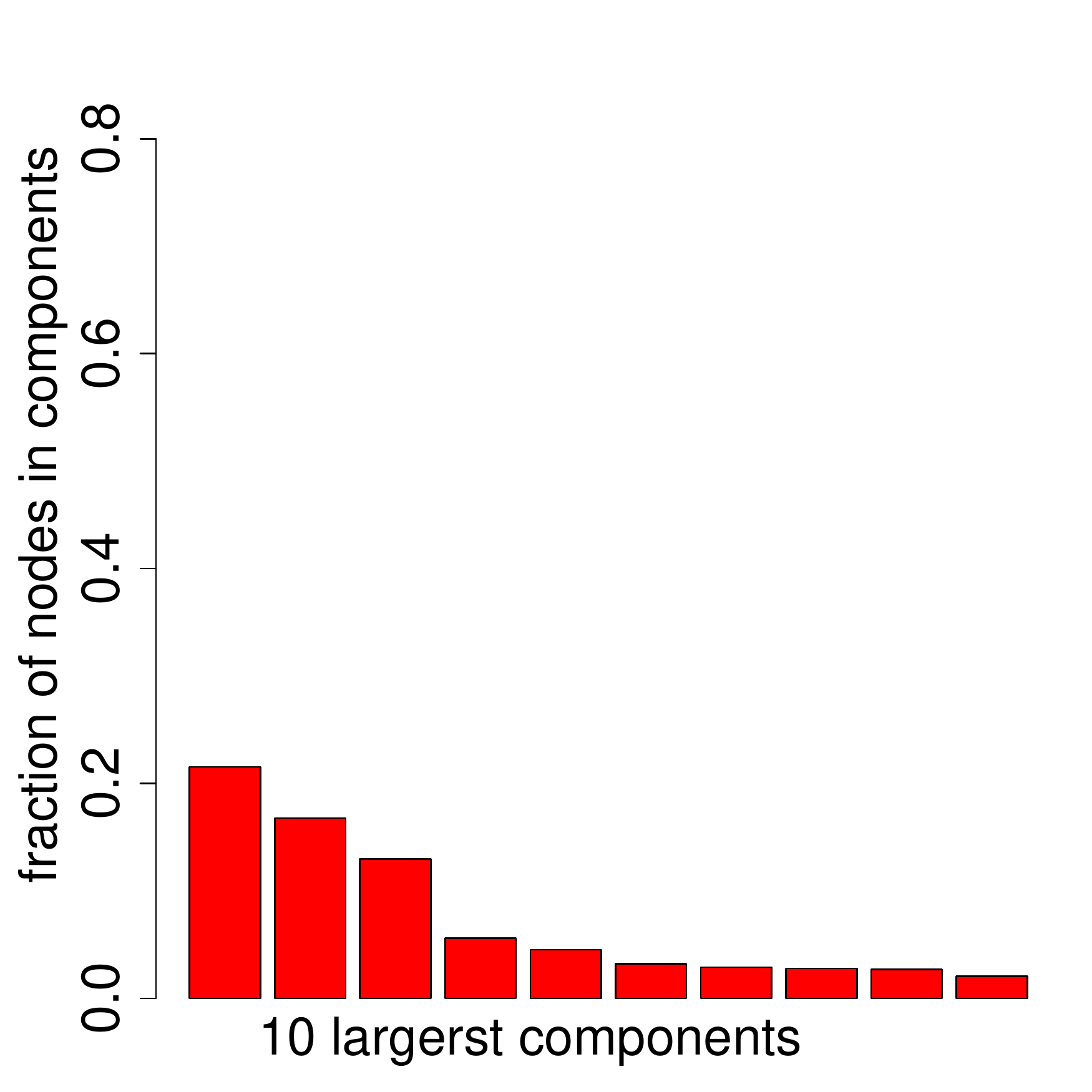}\\
\centerline{$n=8$}
\end{minipage}}
\end{center}
\vspace{-6ex}
\caption{\label{f.PhaseTran_in_n}
Illustration of the
phase transition in the clustering parameter $n$ for the
sub-Poisson Boolean model $C(\Phi^{pert}_{Bin}(n),r)$ with $r=0.55$;
see Section~\ref{ss.NumerialPercol}.
The largest component in the simulation window is highlighted.
Bar-plots show the fraction of nodes in ten largest components.
}
\end{figure}

\begin{figure}[!t]
\begin{center}
\begin{minipage}[b]{0.8\linewidth}
\vspace{-10ex}
\includegraphics[width=1.\linewidth]{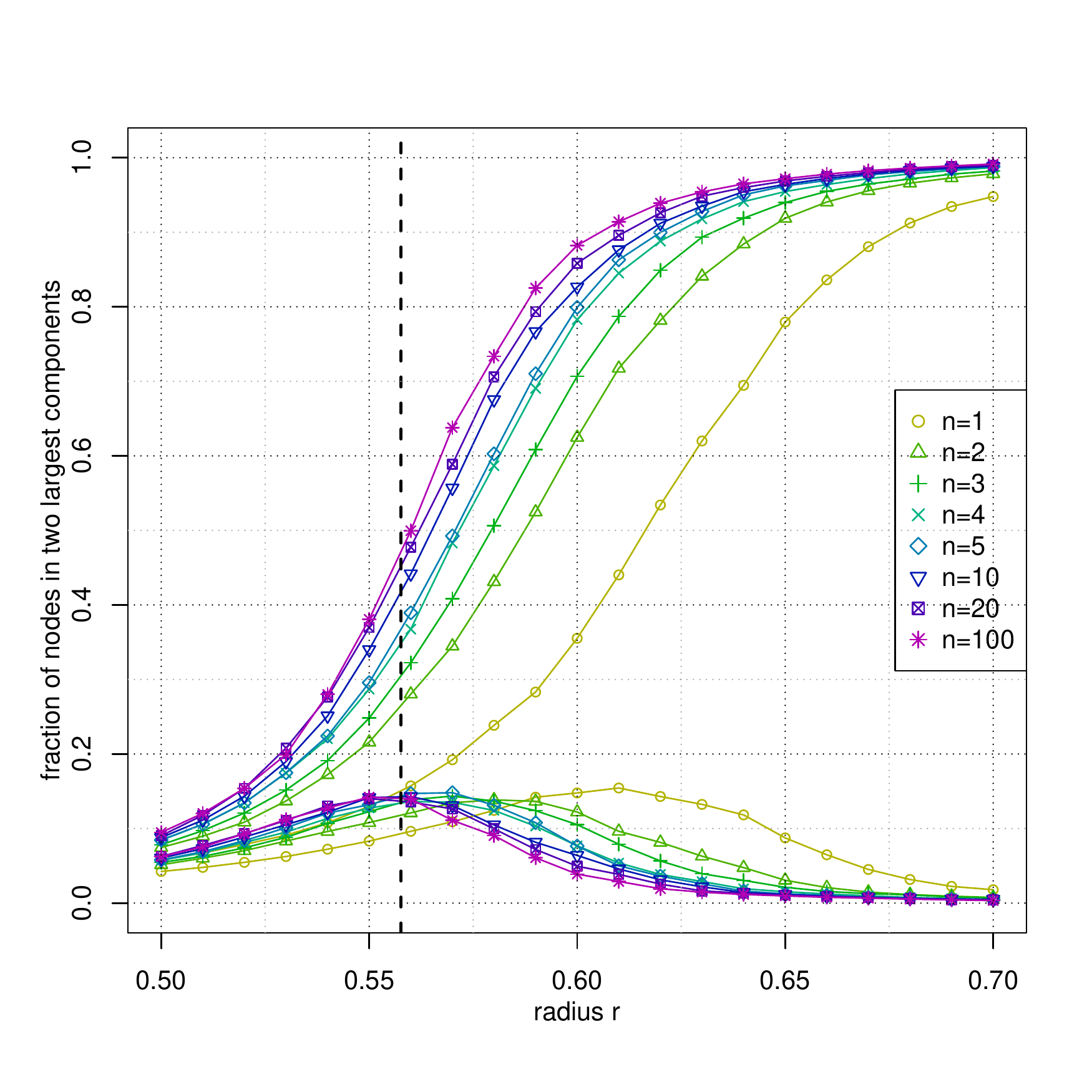}
\end{minipage}
\end{center}
\vspace{-6ex}
\caption{\label{f.TwoComponents.superP}
Mean fractions of nodes in two largest components
of  super-Poisson Boolean model  $C(\Phi^{pert}_{N\!Bin}(n),r)$ in function of $r$;
see Section~\ref{ss.NumerialPercol}. The underlying pp
$\Phi^{pert}_{N\!Bin}(n)$ are $dcx$
decreasing in $n$ to Poisson pp $\Phi_\lambda$
of intensity $\lambda=2/(\sqrt3)=1.154701$.
}
  \end{figure}

\begin{figure}[!t]
\begin{center}
\begin{minipage}[b]{1\linewidth}
\vspace{-10ex}
\includegraphics[width=1.\linewidth]{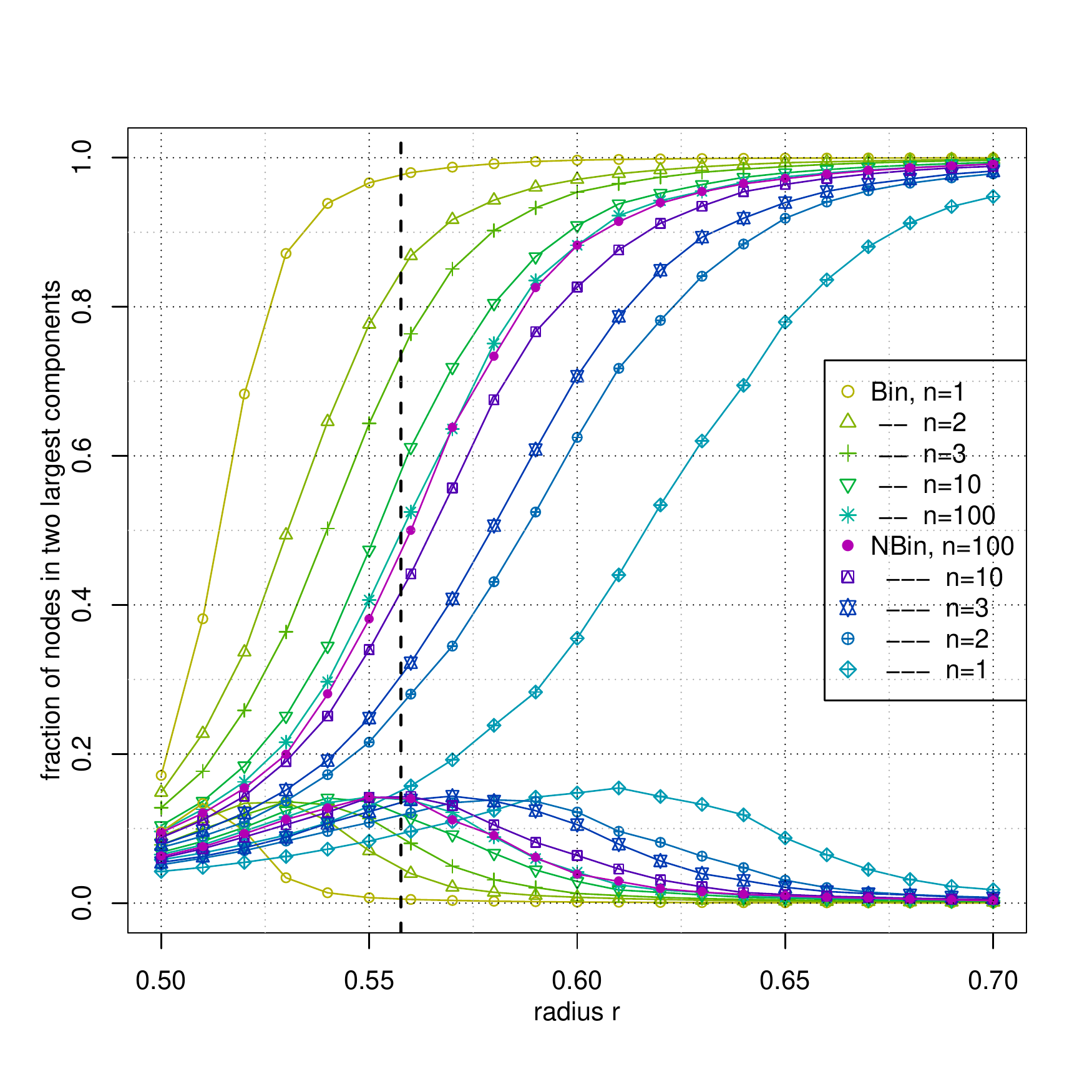}
\end{minipage}
\end{center}
\vspace{-6ex}
\caption{\label{f.TwoComponents.subsuperP}
{From} sub- to super-Poisson. Figures~\ref{f.TwoComponents.subP}
and~\ref{f.TwoComponents.superP} together.}
  \end{figure}

Consider Boolean model $C(\Phi,r)$ with spherical grains of
radius~$r$. Recall that
the critical radius $r_c=r_c(\Phi)$
is the smallest radius $r$ such that  $C(\Phi,r)$  percolates with
positive probability. In what follows, we will show some example of
simulation results  supporting the hypothesis that
$\Phi_1\le_{dcx}\Phi_2$ implies $r_c(\Phi_1)\le r_c(\Phi_2)$.

Consider two families of $dcx$ ordered pp $\Phi^{pert}$ constructed
as perturbations of the hexagonal lattice $\Phi$
(of inter-node distance~1) with
node translation kernel $\cX(x,\cdot)$
being uniform within the hexagonal cell of $x\in\Phi$
and different replication kernels $\cN(x,\cdot)$ (see
Example~\ref{ex:perturbed.lattice}).
Specifically,  assume {\em binomial} $Bin(n,1/n)$
and {\em negative binomial} $N\!Bin(n,1/(1+n))$  distribution
for $\cN(x,\cdot)$ with $n\ge 1$. Recall from
Example~\ref{ex:perturbed.lattice} that the former assumption leads to
$dcx$ increasing in $n$ family of sub-Poisson
pp $\Phi^{pert}=\Phi^{pert}_{Bin}(n)$ converging to Poisson
pp (of intensity $\lambda=2/(\sqrt3)=1.154701$) when $n\to\infty$,
while the latter assumption leads to $dcx$ decreasing  family
of super-Poisson pp  $\Phi^{pert}=\Phi^{pert}_{N\!Bin}(n)$ converging
in $n$ to the same Poisson pp. The critical radius $r_c(\Phi_\lambda)$ 
for this Poisson pp is known to be close to the value $r=0.5576495$;
\footnote{Two dimensional Boolean model with fixed
grains of radius  $r=0.5576495$ and Poisson pp of germs of intensity 
$\lambda=2/(\sqrt3)=1.154701$ has volume fraction 
$1-e^{-\lambda\pi r^2}=0.6763476$, which is given 
in~\cite{Quintanilla2007} as an estimator of the critical value 
for the percolation of the Boolean model. See also bound 
given in~\cite{Balister2005}.}.

In order to get an idea about the critical radius, we have simulated
300 realizations of the Boolean model $C(\Phi^{pert},r)$ for $r$ varying
from $r=0.5$ to $r=0.7$ in the
square window $[0,50]^2$.
The fraction of nodes in the two largest components
in the window  was calculated for each realization of the model for
each $r$ and the obtained results were averaged over 300
realizations of the model.
The resulting  {\em mean fractions of nodes in the two
  largest components} as a function of $r$ are plotted
on Figures~\ref{f.TwoComponents.subP} and~\ref{f.TwoComponents.superP}
for binomial (sub-Poisson) and negative binomial (super-Poisson)
pp, respectively. The two families are compared in
Figure~\ref{f.TwoComponents.subsuperP}.
The obtained curves support the hypothesis that the
clustering of the pp of germs negatively impacts the percolation
of the corresponding Boolean models.
Moreover, Figure~\ref{f.PhaseTran_in_n}
illustrates  what could be called the
``phase transition for percolation in the clustering parameter'' $n$
for $\Phi^{pert}_{Bin}(n)$.

\subsection{Super-Poisson point process with a trivial percolation phase transition}
\label{ss.super-percol}
The objective of this section is to show examples of highly clustered 
and well percolating pp. More precisely we show examples of
Poisson-Poisson cluster pp of arbitrarily small intensity, 
which are super-Poisson, and which 
percolate for arbitrarily small radii.

\begin{exe}[Poisson-Poisson cluster pp with annular clusters]
\label{ex.PoPoClust} 
Let $\Phi_\alpha$ be the Poisson pp of intensity $\alpha$ on the plane
$\mR^2$; we call it the process of cluster centers.
Consider its perturbation $\Phi_\alpha^{R,\delta,\mu}$
(cf. Example~\ref{ex.perturbation_of_Poisson}) with the following 
translation and replication kernels.
Let $\cX(x,\cdot)$ be the uniform distribution on the annulus
$B_O(R)\setminus B_O(R-\delta)$ centered at $x$ of inner and
outer radii $R-\delta$ and $R$ respectively, for some
$0<\delta\le R<\infty$; see Figure~\ref{f.PoPoClust}. 
Let $\cN(x,\cdot)$ be the Poisson distribution $Poi(\mu)$.
The process $\Phi_\alpha^{R,\delta,\mu}$ is a Poisson-Poisson
cluster pp; i.e., a Cox pp with the random intensity measure
$\Lam(\cdot) := \mu \sum_{X \in \Phi_{\alpha}} \cX(x,\cdot-x)$.  
By~\cite[Proposition~5.2]{snorder}, it is 
a super-Poisson pp.  More precisely,
$\Phi_{\lambda}\le_{dcx}\Phi_\alpha^{R,\delta,\mu}$, where
$\Phi_\lambda$ is homogeneous Poisson pp of intensity $\lambda=\alpha\mu$.
\end{exe}

\begin{figure}[!t]
\begin{center}
\includegraphics[width=1.\linewidth]{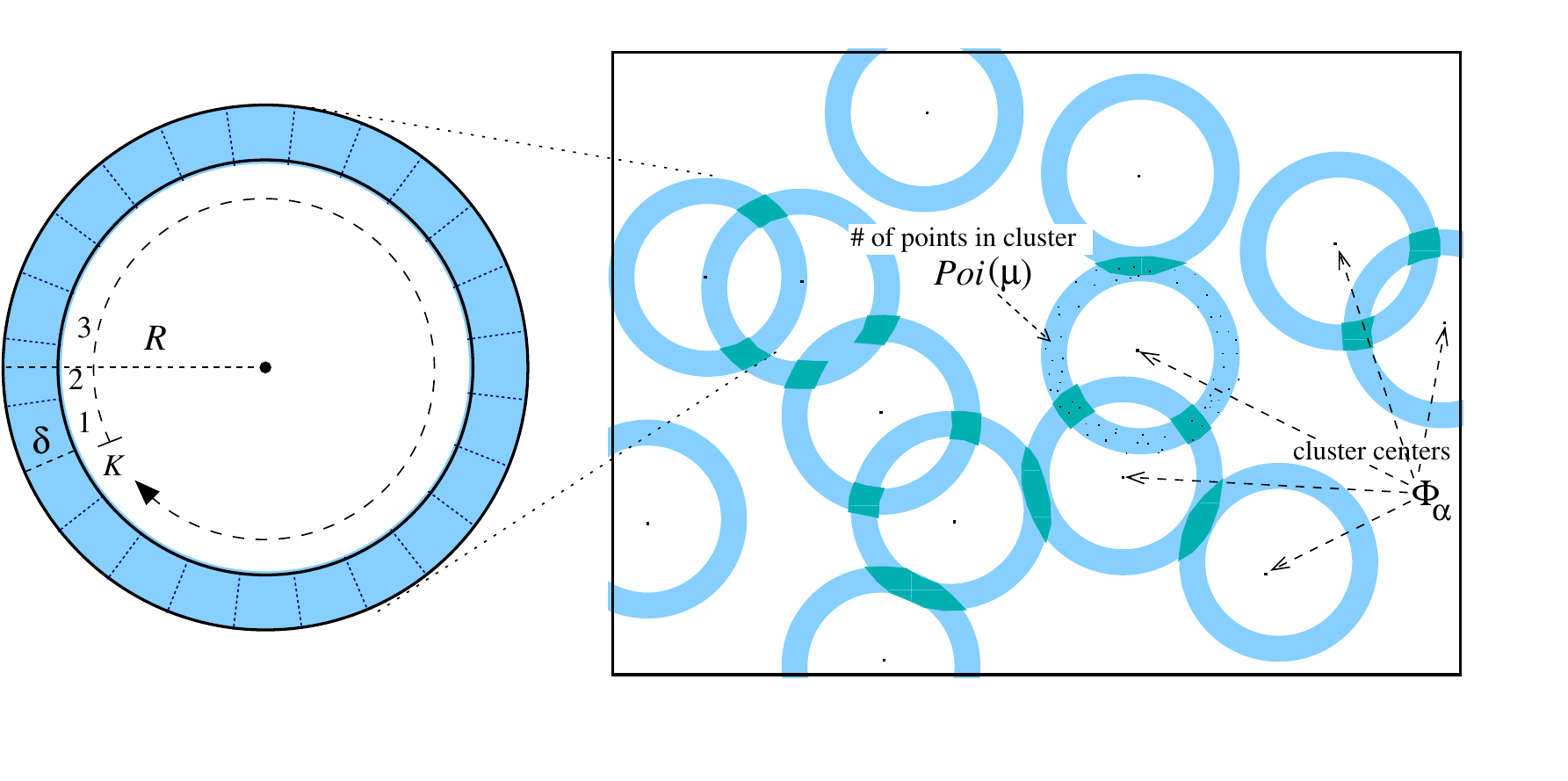}
\end{center}
\vspace{-7ex}
\caption{\label{f.PoPoClust}
Poisson-Poisson cluster process of annular cluster;
cf. Example~\ref{ex.PoPoClust}.}
  \end{figure}

For a given arbitrarily large intensity $\lambda<\infty$, 
taking sufficiently small $\alpha,R$, $\delta=R$ and sufficiently large $\mu$, it
is straightforward to construct a Poisson-Poisson cluster pp
$\Phi_\alpha^{R,R,\mu}$ with spherical clusters, 
which has an arbitrarily large critical radius $r_c$ for percolation.
It is less evident that one can construct a Poisson-Poisson cluster pp
that always percolates, i.e., with degenerate critical radius $r_c=0$.

\begin{prop}
\label{p.PopPoClust}
Let $\Phi_\alpha^{R,\delta,\mu}$ be a Poisson-Poisson cluster pp with 
annular clusters on the plane~$\mR^2$ as in
Example~\ref{ex.PoPoClust}. Given arbitrarily small $a,r>0$, there
exist constants $\alpha,\mu,\delta,R$ such that $0<\alpha,\mu,\delta,R<\infty,$ the intensity $\alpha\mu$ of
$\Phi_\alpha^{R,\delta,\mu}$ is equal to $a$ and 
the critical radius for percolation $r_c(\Phi_\alpha^{R,\delta,\mu})\le r$.
Moreover, for any $a>0$ there exists pp $\Phi$
of intensity $a$, which is $dcx$-larger than the Poisson pp of intensity
$a$, and which percolates for any $r>0$; i.e., $r_c(\Phi)=0$.
\end{prop}

\begin{proof}
Let $a,r>0$ be given. Assume $\delta=r/2$. 
We will show that there exist sufficiently
large $\mu,R$ such that $r_c(\Phi_\alpha^{R,\delta,\mu})\le r$ where $\alpha= a/\mu$.
In this regard, denote $K:=2\pi R/r$ and 
assume that $R$ is chosen such that $K$ is an
integer. For a $\alpha > 0$ and any point (cluster center) $X_i\in\Phi_\alpha,$ let us partition
the annular support  $A_{X_i}(R,\delta):=B_{X_i}(R)\setminus
B_{X_i}(R-\delta)$  of the translation kernel $X_i + \cX(X_i,\cdot)$ 
(support of the Poisson pp constituting the  cluster centered at $X_i$)
into~$K$ cells as shown in
Figure~\ref{f.PoPoClust}. We will call  $X_i$
``open'' if in each of the $K$ cells of $A_{X_i}(R,\delta),$
there exists at least one replication of the point
$X_i$ among the Poisson $Poi(\mu)$ (with $\alpha= a/\mu$) number of total replications of the
point $X_i$.
 Note that given $\Phi_\alpha$, each point
$X_i\in\Phi_\alpha$ is open with probability
$p(R,\mu):=(1-e^{-\mu/K})^K$, independently of 
other points of $\Phi_\alpha$. Consequently, open points of
$\Phi_\alpha$ form a Poisson
pp of intensity $\alpha p(R,\mu)$; call it $\Phi_{open}$.
Note that the maximal distance between any
two points in two neighbouring cells of the same cluster is not
larger than $2(\delta+2\pi R/K)=2r$. Similarly,  
 the maximal distance between any
two points in two non-disjoint cells of two different clusters is not
larger than $2(\delta+2\pi R/K)=2r$. Consequently, if the Boolean
model $C(\Phi_{open},A_{0}(R,\delta))$ with annular grains 
percolates then the Boolean model 
$C(\Phi_\alpha^{R,\delta,\mu}, r)$ with spherical grains of radius $r$
percolates as well. The former Boolean model percolates if and only
if $C(\Phi_{open},B_{0}(R))$ percolates. Hence, in order to guarantee 
 $r_c(\Phi_\alpha^{R,\delta,\mu})\le r,$ it is enough to chose
$R,\mu$ such that the volume fraction
$1 - e^{-\alpha p(R,\mu)\pi R^2} = 1 - e^{-a p(R,\mu)\pi R^2/\mu}$  
is larger than the critical volume fraction
for the percolation of the spherical Boolean model on the plane.
In what follows, we will show that by choosing appropriate 
$R,\mu$ one can make  $p(R,\mu) R^2/\mu$ arbitrarily large. Indeed, 
take 
$$\mu:=\mu(R)=\frac{2\pi R}{r}\log{\frac R{\sqrt{\log R}}}=
\frac{2\pi R}{r}\Bigl(\log R- \frac12\log\log R\Bigr)\,.
$$ 
Then, as $R\to\infty$
\begin{eqnarray*}
p(R,\mu) R^2/\mu
&=&\frac{R^2}{\mu}(1-e^{-\mu r/(2\pi R)})^{2\pi R/r}\\
&=&\frac{Rr}{2\pi(\log R- \frac12\log\log R)}
\Bigl(1-\frac{\sqrt{\log R}}{R}\Bigr)^{2\pi R/r}\\
&=&
e^{O(1)+\log R -\log(2\pi(\log R- \frac12\log\log R))-O(1)\sqrt{\log R}}\to\infty\,.
\end{eqnarray*}
This completes the proof of the first statement.

In order to prove
the second statement, for a given $a>0$, denote $a_n:=a/2^n$ and let
$r_n=1/n$. Consider a sequence of independent (super-Poisson)
Poisson-Poisson cluster 
pp $\Phi_n=\Phi_{\alpha_n}^{R_n,\delta_n,\mu_n}$ with  
intensities $\lambda_n:=\alpha_n\mu_n= a_n$, satisfying 
$r_c(\Phi_n)\le r_n$.  The existence of such pp was
shown in the first part of the proof.
By the fact that $\Phi_n$ are super-Poisson for all $n\ge0$
and by~\cite[Proposition 3.2(4)]{snorder} the superposition
$\Phi=\bigcup_{n=1}^\infty \Phi_n$ is $dcx$-larger than Poisson pp
of intensity~$a$. Obviously $r_c(\Phi)=0$. This completes the proof of
the second statement.
\end{proof}

\begin{rem}
By Proposition~\ref{p.PopPoClust} we know that
there exists pp $\Phi$ with intensity $a>0$, such that
$r_c(\Phi)=0$ and $\Phi_a\le_{dcx}\Phi$, where $\Phi_a$ is homogeneous
Poisson pp.  Since one knows that  $r_c(\Phi_a)>0$ so 
$\Phi$  is a  counterexample for the  monotonicity of 
$r_c$ in  $dcx$ ordering of pp.
\end{rem}

\subsection{Determinantal and permanental point processes}
\label{sec:det.perm}

In this section, we will give some results
related to the $dcx$ comparison of determinantal and permanental pp with respect to the Poisson pp. Refer to \cite{Ben06} for a quick
introduction to these pp and for a more elaborate reading, see
\cite{Ben09}. To this end, we will now recall a general framework (see \cite[Chapter~4]{Ben09}) which allows  us to study ordering of determinantal and permanental pp more explicitly.

\subsubsection{Integral kernels}
\label{sss.Int_ker}

Let $K:\mR^{d}\times\mR^d \to \mathbb{C}$ (where $\mathbb{C}$ are complex numbers) be a {\em locally square-integrable} kernel, with respect to  $\mu^{\otimes 2}$ on $\mR^{2d}$~\footnote{i.e., $\int_D\int_{D}|K(x,y)|^2\,\mu(dx)\mu(dy)<\infty$ for every compact $D\subset\mR^d$}. Then $K$ defines an associated integral operator $\mathcal{K}_D$ on $L^2(D,\mu)$ as $\mathcal{K}_Df(x)=\int_DK(x,y)f(y)\,\mu(dy)$ for complex-valued, square-integrable $f$ on $D$ ($f\in L^2(D,\mu)$).
This operator is compact and hence its  spectrum is discrete.
The only possible accumulation point is $0$ and every non-zero
eigenvalue has finite multiplicity.
Assume moreover that for each compact $D$ the operator $\mathcal{K}_D$
is {\em Hermitian}~\footnote{i.e., $\int_D\overline
{f(x)}\mathcal{K}_Dg(x)\,\mu(dx)=\int_D\overline
{g(x)}\mathcal{K}_Df(x)\,\mu(dx)$ for all $f,g\in\L^2(D,\mu)$},
{\em positive semi-definite}~\footnote{i.e.,
 $\int_D\overline {f(x)}\mathcal{K}_Df(x)\,\mu(dx)\ge 0$}, and
{\em trace-class}; i.e.,
$\sum_{j}|\lambda_j^D|<\infty$, where $\lambda_j^D$ denote the
  eigenvalues of $\mathcal{K}_D$. By the positive semi-definiteness of
  $\mathcal{K}_D,$ these eigenvalues are non-negative.
Further, one can show (cf.~\cite[Lemma~4.2.2]{Ben09}) that for each compact $D$, there exists a ``version''  $K_D(x,y)$
of the kernel $K$, defined on $D'\subset D$ such that $\mu(D\setminus
D')=0$, having the  same associated operator $\mathcal{K}_D$ on $L^2(D,\mu)$
~\footnote{i.e. $K_D(x,y)=K(x,y)$ for $\mu^{\otimes 2}$ almost all
$x,y\in D$}, which
is Hermitian and positive semi-definite.\footnote{%
Recall, a kernel $K(x,y)$ is Hermitian if
$K(x,y)=\overline {K(y,x)}$ for all $x,y\in\mR^d$, where $\overline z$ is the complex
conjugate of $z\in\mathbb{C}$. It is positive semi-definite
$\sum_{i=1}^k\overline z_i \sum_{j=1}^k K(x_i,x_j) z_j\ge 0$
for all $z_i\in\mathbb{C}$, $i=1,\ldots,k$, $k \geq 1.$}
Specifically one can take $K_D(x,y) = \sum_j \lam_j^D
\phi_j^D(x) \overline{\phi_j^D(y)}$ where $\phi_j^D(\cdot)$ are the
corresponding normalized eigenfunctions of $\mathcal{K}_D$.

\begin{exe}\label{ex.det}
{\em Determinantal pp.} A simple pp on $\mR^d$
is said to be a {\em determinantal pp} with a kernel $K(x,y)$ with respect to
a Radon measure $\mu$ on $\mR^d$ if the  joint intensities
of the pp with respect to the $k\,$th-order product $\mu^{\otimes k}$ of $\mu$ satisfy
$\rho^{(k)}(x_1,\ldots,x_k) = \det \big( K(x_i,x_j) \big)_{1
\leq i,j \leq k}$ for
all $k$, where $\big(a_{ij}\big)_{1 \leq i,j \leq k}$ stands for a
matrix with entries $a_{ij}$ and $\det \big( \cdot \big)$ denotes the
determinant of the matrix. Note that the mean measure of the
determinantal pp (if it exists) is  equal to
$\alpha(\cdot)=\int_{\cdot}K(x,x)\,\mu(dx)$.
Clearly, one needs assumptions on the kernel $K(x,y)$
for the above equation to define the joint intensities of a pp.
In what follows, we shall assume that the kernel $K$ is an integral
kernel satisfying the assumptions of Section~\ref{sss.Int_ker}.
Then, there exists a unique pp $\Phi^{det}$ on $\mR^d$,
such that for each compact $D$, the restriction of $\Phi^{det}$ to $D$
is a determinantal  pp with kernel $K_D$ if and only if the
eigenvalues of $\mathcal{K}_{D}$ are in $[0,1]$. This latter condition
is equivalent to $\lambda_j^D\in[0,1]$ for all compact $D$;
cf.~\cite[Theorem~4.5.5]{Ben09}. We will call this pp {\em
  determinantal pp with the trace-class integral kernel $K(x,y)$}.

\end{exe}

\begin{exe}\label{ex.perm}
{\em Permanental pp.}  Similar to determinantal pp, one  says that a simple pp is a {\em permanental pp} with a kernel $K(x,y)$ with respect to a Radon measure $\mu$ on $\mR^d$ if the  joint intensities of the pp with respect to $\mu^{\otimes k}$
satisfy $\rho^{(k)}(x_1,\ldots,x_k) = \text{per}\big( K(x_i,x_j) \big)_{1
\leq i,j \leq k}$ for all $k$, where $\text{per}\big( \cdot \big)$
stands for the permanent of a matrix. Note that the mean measure of the permanental pp is also  equal to $\alpha(\cdot)=\int_{.}K(x,x)\,\mu(dx)$.
Again, will assume that $K(x,y)$ is an integral
kernel as in Section~\ref{sss.Int_ker}.
Then, there exists a unique pp $\Phi^{perm}$ on $\mR^d$, such that for each compact $D$, the restriction of $\Phi^{perm}$ to $D$ is a permanental pp with kernel $K_D$; cf.~\cite[Corollary~4.9.9]{Ben09}.  We will call this pp {\em
 permanental pp with the trace-class integral kernel $K(x,y)$}. From~\cite[Proposition.~35 and Remark~36]{Ben06}, we also know that
$\Phi^{perm}$ is a Cox pp with intensity field $|F|^2,$ where $F$ is some complex Gaussian process on $\mR^d$.
\end{exe}

The following properties hold true for determinantal and permanental
pp with a trace-class integral kernel $K(x,y)$.
\begin{itemize}
\item $\Phi^{det}$ {\em is  $\alpha$-weakly sub-Poisson}, while
$\Phi^{perm}$ {\em is $\alpha$-weakly super-Poisson};
both comparable with respect to the Poisson pp with mean measure
$\alpha(\cdot)$ given by $\alpha(D)=\int_D
K_D(x,x)\, \mu(d)=\sum_{j}\lambda_j^D$, where
the summation is taken over all the eigenvalues $\lambda_j^D$ of
$\mathcal{K}_D$. The proof follows by the fact that $K_D(x,y)$
is Hermitian and  positive semi-definite. Hence, by the Hadamard's
inequality,
$0\le \det \big( K_D(x_i,x_j) \big)_{1 \leq i,j \leq k}
\leq \prod_{i=1}^k K_D(x_i,x_i)$ which
implies~(\ref{e.alpha-weakly-sub-Poisson}). For $\Phi^{perm}$, the
proof follows from the permanent analogue of the Hadamard's
inequality (see~\cite{Marvin64}). In fact, for the above
comparison  to be true it suffices for the kernel $K(x,y)$ to
be Hermitian and  positive semi-definite.
\item It is known that for each compact $D$,
 $\Phi^{det}(D)$ has the same distribution
as the sum of independent Bernoulli random variables
$Bin(1,\lambda_j^D)$ cf.~\cite[Theorem~4.5.3]{Ben09}, while
$\Phi^{perm}(D)$ has the same distribution as the sum of independent geometric
random variables $Geo(1/(1+\lambda_j^D))$, where the summation is taken
over all eigenvalues $\lambda_j^D$ of $\mathcal{K}_D$;
cf.~\cite[Theorems~4.5.3 and 4.9.4]{Ben09}. Consequently
\begin{equation}\label{eq.det-poi-det-marginals}
\Phi^{det}(D)\le_{cx}Poi(\sum_j\lambda_j^D)\le_{cx}\Phi^{perm}(D)\,,
\end{equation}
with the left inequality holding provided $\Phi^{det}$ exists (i.e.;
$\lambda_j^D\in[0,1]$ for all compact $D$).
\item Inequalities~(\ref{eq.det-poi-det-marginals}) imply that
$\Phi^{det}$ {\em is $\nu$-weakly sub-Poisson}, while $\Phi^{perm}$
{\em is $\nu$-weakly super-Poisson}.
Consequently,  $\Phi^{det}$ {\em is weakly sub-Poisson}, while $\Phi^{perm}$
{\em is weakly super-Poisson}.
\end{itemize}

In Proposition~\ref{prop:det-perm-dcx}, we will
strengthen~(\ref{eq.det-poi-det-marginals}) proving $dcx$ ordering of
finite-dimensional distributions of $\Phi^{det}$ and $\Phi^{perm}$ on mutually disjoint {\em simultaneously observable} sets $D_1,\ldots,D_k$. Simultaneous observability means that the eigenfunctions of $\mathcal{K}_{\bigcup D_i}$, restricted to $D_i$ are also eigenfunctions of $\mathcal{K}_{D_i}$ for every $ i=1,\ldots,k$.

\begin{prop}\label{prop:det-perm-dcx}
Let $\Phi^{det}$ and $\Phi^{per}$ be, respectively,  the
determinantal and permanental pp  with a trace-class integral kernel $K$ and
with $\Phi^{det}$ being defined only if the
spectrum of $\mathcal{K}_{\mR^d}$ is in $[0,1]$.
Denote by $\Phi^{Poi}$ the Poisson pp of mean measure
$\alpha(\cdot)$ given by $\alpha(D)=\sum_{j}\lambda_j^D$ for all
compact $D$, where
the summation is taken over all eigenvalues $\lambda_j^D$ of
$\mathcal{K}_D$. Let $D_1,\ldots,D_k$ be mutually disjoint,
simultaneously observable (with respect to the kernel $K$)
compact subsets of $\mR^d$  and $D = \bigcup D_i$.
Then
\begin{eqnarray*}
\Bigl(\Phi^{det}(D_1),\ldots,\Phi^{det}(D_k)\Bigr)
&\leq_{dcx}& \Bigl(\Phi^{Poi}(D_1),\ldots,\Phi^{Poi} (D_k)\Bigr)\\
&\le_{dcx}&
\Bigl(\Phi^{per}(D_1),\ldots,\Phi^{per}(D_k)\Bigr)\,.
\end{eqnarray*}
\end{prop}

\begin{proof}
From~\cite[Prop.~4.5.9]{Ben09}, we know that
$$ \Bigl(\Phi^{det}(D_1),\ldots,\Phi^{det}(D_k)\Bigr) \stackrel{d}{=} \sum_{j=1}^J \boldsymbol{\xi}_j,$$
with $J$ denoting the number of eigenvalues of $\mathcal{K}_{D}$
($J=\infty$ and $0$ allowed and in the latter case the sum is
understood as $0$) and $\boldsymbol{\xi}_j$,  $j=1,\ldots,J$ are
are independent vectors, with components $\boldsymbol{\xi}_j=
(\xi_{j,1},\ldots,\xi_{j,k})$  whose distributions are
$\pr{\xi_{j,l}=0\;\text{for all}\;\allowbreak l\not=i, \xi_{j,i} =1}=\lambda^D_{j,i}$ and
$\mathsf{P}(\xi_{j,l} =0\;\text{for}\allowbreak\;\text{all}\;
l=1,\ldots, k)
= 1 - \lambda^D_j$,
where $\{\lambda_{j,i}\}_{j=1,\ldots,J}$  are the eigenvalues of
$\mathcal{K}_{D_i}$ and $\lambda^D_j=\sum_{i=1}^k\lambda^D_{j,i}$,
$j=1,\ldots, J$
are the eigenvalues of $\mathcal{K}_D$.
Due to the independence of $\boldsymbol{\xi_j}$'s and the assumption
$\sum_{j=1}^J\lambda_j<\infty$ (local trace-class property of $\mathcal{K}_D$),
it is enough to prove for each $j$, that
$\boldsymbol{\xi}_j\le_{dcx}\boldsymbol{\zeta}_j$ where
$\boldsymbol{\zeta}_j=(\zeta_{j,1},\ldots,\zeta_{j,k})$ is the vector of
independent Poisson random variables $\zeta_{j,l}\sim Poi(\lambda^D_{j,l})$.
In this regard, note that the random vectors
 $\boldsymbol{\xi}_j$ and $\boldsymbol{\zeta}_j$ can be coupled as
follows: $\boldsymbol{\xi}_j=\sum_{i=1}^N\boldsymbol{\xi}'_i$,
$\boldsymbol{\zeta}_j=\sum_{i=1}^M\boldsymbol{\xi}'_i$
 where $N\sim Bin(1,\lambda^D_j)$, $M\sim Poi(\lambda^D_j)$
and given $N, M$, $\boldsymbol{\xi}'_i$, $i \geq 1$ are independent {\em
  multinomial} vectors
$Mul(1,\lambda^D_{j,1}/\lambda^D_j,\ldots,\lambda^D_{j,k}/\lambda^D_j)$~\footnote{
$Mul(k,p_1,\ldots,p_k)$ with $0\le p_i\le 1$, $\sum_{i=1}^kp_i=1$,
has  probability mass  function
$p_{Mul(n,p_1,\ldots,p_k)}(n_1,\ldots,n_k)=\frac{n!}{n_1!\dots n_k!}
p_1^{n_1}\dots p_k^{n_k}$ for $n_1+\ldots+n_k=n$ and 0 otherwise.}.
Then, $\boldsymbol{\xi}_j\le_{dcx}\boldsymbol{\zeta}_j$ follows
from the fact that $Bin(1,p)\le_{cx}Poi(p)$ (see Example \ref{ex:perturbed.lattice}) and
Lemmas~\ref{lem:MeesterSh-like}, \ref{lem:discrete-convex}.
This completes the proof of the inequality for the determinantal pp.

Regarding the permanental pp, we know from~\cite[Theorem~4.9.7]{Ben09} that the distribution of
$$ \Bigl(\Phi^{per}(D_1),\ldots,\Phi^{per}(D_k)\Bigr) \stackrel{d}{=}
\sum_{j=1}^J \boldsymbol{\eta}_j,$$
where $\boldsymbol{\eta}_j$,  $j=1,\ldots,J$ are are independent vectors,
with $\boldsymbol{\eta}_j=\sum_{i=1}^K\boldsymbol{\xi}'_i$,
where $K\sim Geo(1/(1+\lambda^D_j)$ and, given $K$, independent
multinomial vectors  $\boldsymbol{\xi}'_i$, $i=1,\ldots$ as above.
Similarly, as for the determinantal pp, the required inequality
follows from the ordering, $\boldsymbol{\zeta}_j\le_{dcx}\boldsymbol{\eta}_j$ for
all $j$, which follows from the fact that
$Poi(\lambda)\le_{cx}Geo(1/(1+\lambda)$ (see Example \ref{ex:perturbed.lattice}) and Lemmas~\ref{lem:MeesterSh-like},~\ref{lem:discrete-convex}.
This completes the proof.
\end{proof}

\begin{rem}
Note that the key observation used in the above proof
is that the number of points in disjoint,
simultaneously observable sets $D_1,\ldots,D_k$ can be
represented as a sum of independent vectors, which themselves are
binomial (for determinantal) or Poisson (for Poisson) or geometric (for permanental) sums of some further independent vectors. This is exactly the same
representation as for the perturbed pp of Section~\ref{sec:pl}
(available for any disjoint sets $B_1,\ldots,B_k$); cf the  proof of
Proposition~\ref{p.pert-lattice}.  In both cases,
this representation and Lemmas~\ref{lem:MeesterSh-like}, \ref{lem:discrete-convex} allow us to conclude $dcx$ ordering of the corresponding vectors.
\end{rem}

\begin{exe}{\em Ginibre process.}
Let $\Phi^{G}$ be the determinantal pp on $\mR^2$
with kernel
$K((x_1,x_2),(y_1,y_2))=\exp[(x_1y_1+x_2y_2)+i(x_2y_1-x_1y_2)]$,
$x_j,y_j\in\mR$,  $j=1,2$,
with respect to the measure
$\mu(d(x_1,x_2))=\pi^{-1}\exp[-x_1^2-x_2^2]\,dx_1dx_2$.
This process is known as the {\em infinite Ginibre} pp.
Denote by $\Psi^G=\{|X_i|^2: X_i\in\Phi^G\},$ the pp on $\mR^+$
of the squared radii of the points of $\Phi^G$.
We know that an arbitrary finite collection of the annuli centered
at the origin $D_i=\{(x_1,x_2): r_i\le x_1^2+x_2^2\le R_i\}$
is simultaneously observable for this pp;
cf.~\cite[Example~4.5.8]{Ben09}. Using this observation,
Proposition~\ref{prop:det-perm-dcx} and the fact that $dcx$ order
of pp on $\mR$ is generated by the semi-ring of intervals,
we conclude that $\Psi^G$ is $dcx$ smaller than the Poisson pp $\Phi^1$
of unit intensity on $\mR^+$.  A partial result,
$\Psi^G([0,r])\le_{cx}\Phi^1([0,r])$ for all $r\ge 0$, was
proved in~\cite{snorder}.
\end{exe}

\section{Concluding remarks}
We come back to the initial heuristic discussed in the Introduction ---  
clustering in a point process should increase the  critical radius 
for the percolation of the corresponding continuum percolation model.
As we have seen, even a relatively strong tool such as the $dcx$ order
falls short, when it comes to making a formal statement of this heuristic. 

The two natural questions are what would be a more suitable measure
of clustering that can be used to affirm the heuristic and
whether $dcx$ order can satisfy a weaker version of the conjecture.

As regards the first question, one might start by looking at other
dependence orders such as super-modular, component-wise convex or
convex order but it has been already shown that the super-modular and
component-wise convex orders are not suited to comparison of clustering
in point processes (cf.~\cite[Section~4.4]{Yogesh_thesis}). 
Properties of convex order on point processes are yet to be investigated fully and this research direction is interesting in its own right, apart from its relation to the above conjecture. In a similar vein, it is of
potential interest to study other stochastic orders on point processes.

On the second question, it is pertinent to note that sub-Poisson point processes surprisingly exhibited non-trivial phase transitions for percolation. 
Such well-behavedness of the sub-Poisson point processes
makes us wonder if it is possible to prove a rephrased conjecture
saying that any homogeneous sub-Poisson pp has a smaller critical radius for
percolation than the Poisson pp of the same intensity.
Such a conjecture matches well
with~\cite[Conjecture~4.6]{BenjaminiStauffer2011}.

\appendix
\section{}

The following result, similar to~\cite[Lem\-ma~2.17]{MeesterSh93}
is used in the proof of Propositions~\ref{p.pert-lattice}
and~\ref{prop:det-perm-dcx}.

\begin{lem}\label{lem:MeesterSh-like}
Let ${\boldsymbol\xi}_i=(\xi_i^1,\ldots,\xi_i^k)\in\mR^k$, ($i=1,2,\ldots$)
be independent, identically distributed  vectors of (possibly
dependent) non-negative random variables.
Suppose $f$ is a $dcx$ function on $\mR^k$. Then, the function $g$ defined
on $\mZ$ by
 $g(n) =\EXP{f(\text{sgn}(n)\sum_{i=1}^{|n|}{\boldsymbol\xi}_i)}$ for $n\not=0$ and  $g(0)=0$ is convex on~$\mZ$.
\end{lem}
\begin{proof}
We will prove that $g(n)$ has non-negative
second differences
\begin{equation}\label{eq:second-differences}
g(n-1)+g(n+1)-2g(n)\ge 0\qquad \text{for all $n\in\mZ$}
\end{equation}
and use the first part of Lemma~\ref{lem:discrete-convex}.
To prove~(\ref{eq:second-differences}),
define $G(n,m) := \sum_{i=n+1}^m{\boldsymbol\xi}_i$
for $0\le n < m$ and $G(n,n):=(0,0,\ldots,0)\in\mR^k$ for $n\ge0$.
We have for $n \ge 1$,
\begin{eqnarray*}
2g(n)&=&2\EXP{f\Bigl(G(0,n)\Bigr)}\\
&=&\EXP{f\Bigl(G(0,n-1)+G(n-1,n)\Bigr)}+
\EXP{f\Bigl(G(0,n)\Bigr)}\\
&=&\EXP{f\Bigl(G(0,n-1)+G(n,n+1\Bigr)}+
\EXP{f\Bigl(G(0,n)\Bigr)}\\
&=&\EXP{f\Bigl(G(0,n-1)+G(n,n+1\Bigr)+
f\Bigl(G(0,n)\Bigr)}\\
&\le &\EXP{f\Bigl(G(0,n-1)\Bigr)+
f\Bigl(G(0,n)+G(n,n+1)\Bigr)}\\
&=&g(n-1)+g(n+1)\,,
\end{eqnarray*}
where for the third equality
we have used mutual independence of $G(0,n-1), G(n-1,n),
G(n,n+1)$ and the fact that $G(n-1,n)$ and $G(n,n+1)$ have the same
distribution, while the inequality follows from the
$dcx$ property of $f$ and the assumption
${\boldsymbol\xi}_{i}\ge0$.
This proves~(\ref{eq:second-differences})
for $n\ge1$. Similar reasoning allows to
show~(\ref{eq:second-differences}) for $n\le-1$.
Finally, note that for  $n=0$
\begin{eqnarray*}
2g(0)&=&2f\Bigl((0,\ldots,0)\Bigr)\\
&=&\EXP{f\Bigl(-G(0,1)+G(0,1)\Bigr)+
f\Bigl((0,\ldots,0)\Bigr)}\\
&\le &\EXP{f\Bigl(-G(0,1)\Bigr)+
f\Bigl(G(0,1)\Bigr)}\\
&=&g(-1)+g(1)\,,
\end{eqnarray*}
\end{proof}

We will prove the following two technical results regarding convex
functions. We were not able to find their proofs  in the literature.
\begin{lem}\label{lem:discrete-convex}
Let $g(n)$ be a real valued function defined for all integer
$n\in\mZ$ and satisfying condition~(\ref{eq:second-differences}).
Then the following two statements are true:
\begin{enumerate}
\item For all $n\ge 2$
\begin{eqnarray}\label{eq:discrete-convex}
g\Bigl(\sum_{i=1}^n\lambda_i k_i\Bigr)&\le&
\sum_{i=1}^n\lambda_i g(k_i)\\
\noalign{\text{for all $k_i\in\mZ$ and $0\le \lambda_i\le 1$,
 $\sum_{i=1}^n\lambda_i=1$ such that
 $\sum_{i=1}^n\lambda_i k_i\in\mZ$}.}\nonumber
\end{eqnarray}
\item  Function $g(\cdot)$ can be extended
to a real valued convex function defined on real numbers~$\mR$.
\end{enumerate}
 \end{lem}
\begin{proof}
{\em Ad 1.} As mentioned in~\cite[Section~V.16.B.10.a]{MOA2009}
it is easy to see that~(\ref{eq:second-differences})
is equivalent to~(\ref{eq:discrete-convex}) with $n=2$.
Assume now that~(\ref{eq:discrete-convex}) holds true for some
$n\ge2$ (and all $0\le\lambda_i\le1$, $k_i\in\mZ$, $i=1,\ldots,n$ satisfying
$\sum_{i=1}^n\lambda_i=1$,  $\sum_{i=1}^n\lambda_i k_i\in\mZ$).
We will prove that it holds true for $n+1$ as well.
In this regard,
define  for a given  $k\in\mZ$ and distinct (otherwise we use directly
the inductive assumption) $k_1,\ldots,k_{n+1}\in\mZ$,
the following functions:
\begin{eqnarray*}
\lambda_{n}=\lambda_{n}(\lambda_1,\ldots,\lambda_{n-1})&:=&
\frac{k-k_{n+1}-\sum_{i=1}^{n-1}\lambda_i(k_i-k_{n+1})}{k_n-k_{n+1}}\\
\lambda_{n+1}=
\lambda_{n+1}(\lambda_1,\ldots,\lambda_{n-1})&:=&1-\sum_{i=1}^{n-1}\lambda_i-
\lambda_{n}(\lambda_1,\ldots,\lambda_{n-1})\\
F(\lambda_1,\ldots,\lambda_{n-1})&:=&\sum_{i=1}^{n-1}\lambda_ig(k_i)
+\lambda_{n}(\lambda_1,\ldots,\lambda_{n-1})g(k_n)\\
&&\hspace{5em}+\lambda_{n+1}(\lambda_1,\ldots,\lambda_{n-1})g(k_{n+1})\,.
\end{eqnarray*}
Note that for any $\lambda_1,\ldots,\lambda_{n-1}$ we have
$\sum_{i=1}^{n+1}\lambda_i=1$ and $\sum_{i=1}^{n+1}\lambda_ik_i=k$.
Consider the following
subset of the $n-1$-dimensional unit cube
$$C:=\Bigl\{(\lambda_1,\ldots,\lambda_{n-1})\in[0,1]^{n-1}:
0\le \lambda_{n}\le 1, 0\le \lambda_{n+1}\le 1\Bigr\}\,.$$
The proof of the inductive step will be
completed if we show that
$F(\cdot) \ge g(k)$ on $C$.
In this regard note that $C$ is {\em closed} and {\em  convex}.
Assume moreover that $C$ is not empty; otherwise the
condition~(\ref{eq:discrete-convex})
is trivially satisfied.
Note also that  $F(\cdot)$ is  an affine, real valued function
defined on $\mR^{n-1}$. Hence, by the maximum principle,
the affine (hence convex) function $-F$
attains its maximum relative to $C$ on some point
$(\lambda_1^0,\ldots,\lambda_{n-1}^0)\in\partial C$ of the boundary of
$C$. Consequently, we have
$F(\cdot)\ge F(\lambda_1^0,\ldots,\lambda_{n-1}^0)$ on $C$ and the
 proof of the inductive step will be completed if we show that
$ F(\lambda_1^0,\ldots,\lambda_{n-1}^0)\ge g(k)$.
In this regard, denote
$\lambda_{n}^0=\lambda_n(\lambda_1^0,\ldots,\lambda_{n-1}^0)$ and
$\lambda_{n+1}^0=\lambda_{n+1}(\lambda_1^0,\ldots,\lambda_{n-1}^0)$.
Using the continuity of the functions
$\lambda_{n}(\cdot)$ and $\lambda_{n+1}(\cdot)$
 is not difficult to verify that
 $(\lambda_1^0,\ldots,\lambda_{n-1}^0)\in\partial C$ implies
$\lambda_j^0=0$  for some $j=1,\ldots,n+1$.
Thus, by
our inductive assumption, $F(\lambda_1^0,\ldots,\lambda_{n-1}^0)=
\sum_{i=1, i\not=j}^{n+1}\lambda_i^0g(k_i)\ge g(k)$,
which completes the proof of~(\ref{eq:discrete-convex}) for all $n\ge2$.

{\em Ad 2.}  We will recall the arguments used in~\cite{Yan1997}
to show that a function satisfying~(\ref{eq:discrete-convex})
for all $n\ge 2$ (called {\em globally convex function} there)
has a convex extension on $\mR$. In this regard,
consider the epigraph $\text{epi}(g):=\{(k,\mu)\in\mZ\times\mR: \mu\ge
g(k)\}$ of $g$ and its convex envelope $\text{epi}^{co}(g)$.
It is easy to see that
$\text{epi}^{co}(g)=\{(x,\mu)\in\mR^2:
\mu\ge\sum_{i=1}^n\lambda_ig(k_i)$ for some $k_i\in\mZ$, $0\le
\lambda_i\le 1$, $\sum_{i=1}^n\lambda_i=1$ and $\sum_{i=1}^n\lambda_ik_i=x\}$.
Define $\tilde g(x):=\inf\{\mu:(x,\mu)\in\text{epi}^{co}(g)\}$ for all
$x\in\mR$. The convexity of $\text{epi}^{co}(g)$ implies that $\tilde
g$ is convex on $\mR$ and the global
convexity~(\ref{eq:discrete-convex}) of $g$ implies that $\tilde g$ is an
extension of $g$.
This completes the proof.
\end{proof}

\bibliographystyle{plainnat}
\bibliography{perc-dcx}

\end{document}